\documentclass[12pt, reqno]{amsart}

\usepackage{lineno}

\usepackage{amscd, amsmath, latexsym, mathtools, thmtools, amssymb, amsthm, enumerate}

\usepackage{lscape,fancyhdr}

\usepackage[top=2.5cm, left=3cm, right=3cm]{geometry}

\usepackage[dvipsnames]{xcolor}
\usepackage{epsfig,graphicx,tikz}
\usepackage{wrapfig, caption, subcaption}
\usepackage{float, multicol}
\usepackage[all,cmtip]{xy}
\usepackage{tikz-cd}
\usepackage{tkz-euclide}
\usepackage[inkscapelatex=false]{svg}

\definecolor{lnkcol}{rgb}{0,0.25,0.75}

\usepackage[alphabetic, msc-links]{amsrefs}
\hypersetup{colorlinks=true, linkcolor=lnkcol, urlcolor=magenta, citecolor=RubineRed}

\allowbreak

\newtheorem{thmA}{Theorem}
\newtheorem{corA}{Corollary}

\newtheorem{secn}{Definition}[section]
\newtheorem{thm}[secn]{Theorem}
\newtheorem{cor}[secn]{Corollary}
\newtheorem{prop}[secn]{Proposition}
\newtheorem{lem}[secn]{Lemma}
\newtheorem{defn}[secn]{Definition}
\newtheorem{ex}[secn]{Example}
\newtheorem{rem}[secn]{Remark}

\numberwithin{equation}{secn}
\numberwithin{figure}{section}

\begin{document}

\title{Maps between Boundaries of Relatively Hyperbolic Groups}

\author{Abhijit Pal}
\address{Department of Mathematics and Statistics, Indian Institute of Technology, Kanpur, India}
\email{abhipal@iitk.ac.in}

\author{Rana Sardar}
\address{Department of Mathematics and Statistics, Indian Institute of Technology, Kanpur, India}
\email{ranasardar20@iitk.ac.in}

\subjclass[2000]{}
\date{}
\dedicatory{}
\thanks{The first author was partially supported by MATRICS Research Grant MTR/2021/000194 }


\begin{abstract}

F. Paulin proved that if the Gromov boundaries of two hyperbolic groups are quasi-M\"{o}bius equivalent, then the groups themselves are quasi-isometric. The goal of this article is to extend Paulin’s result to the setting of relatively hyperbolic groups by introducing the notion of relative quasi-M\"{o}bius maps between the Bowditch boundaries of relatively hyperbolic groups. We show that any coarsely cusp-preserving quasi-isometry between two relatively hyperbolic groups induces a homeomorphism between their Bowditch boundaries, and that this induced homeomorphism is relative quasi-M\"{o}bius and linearly distorts the exit points of bi-infinite geodesics into combinatorial horoballs.
Conversely, we prove that if a homeomorphism between the Bowditch boundaries of two relatively hyperbolic groups preserves parabolic fixed points and is either relative quasi-M\"{o}bius or linearly distorts the exit points of bi-infinite geodesics into combinatorial horoballs, then it arises from a coarsely cusp-preserving quasi-isometry between the groups.

\end{abstract}


\fontfamily{ptm}\selectfont
\maketitle 
\small{\begin{center}{AMS Subject Classification : 20F65, 20F67, 20E08}\end{center}}


\section{Introduction}
In 1987, M. Gromov \cite{Gro87} introduced the notion of hyperbolic metric spaces. For $\delta\geq 0$, a geodesic metric space is said to be $\delta$-hyperbolic if, for any geodesic triangle $\triangle$, each side of $\triangle$ is contained in the closed $\delta$-neighbourhood of the union of the other two sides. A metric space is called hyperbolic if it is $\delta$-hyperbolic  for some $\delta\geq 0$.
The Gromov boundary of a proper hyperbolic metric space provides a compactification of the space by assigning a `boundary' at infinity. A quasi-isometry between two proper hyperbolic metric spaces induces a homeomorphism between their Gromov boundaries. However, the converse is not true.  There exist two hyperbolic groups, due to Pansu, Bourdon and Pajot, with homeomorphic Gromov boundaries, but they are not quasi-isometric to each other.
Paulin \cite{Pau96} proved that if the boundaries of two hyperbolic groups are homeomorphic and quasi-M\"{o}bius (or quasiconformal) equivalent, then they are quasi-isometric to each other.
Let $\partial X$ denote the Gromov boundary of a proper $\delta$-hyperbolic metric space $X$. 
Corresponding to a geodesic triangle $\triangle (a,b,c)$ with vertices $a,b,c$ in $X \cup \partial X$, there exists a point $p_{abc} \in X$, called a quasi-centre, such that the distance of $p_{abc}$ to each side of $\triangle (a,b,c)$ is uniformly bounded and the bound depends only on the hyperbolic constant $\delta$.  The definition of a quasi-M\"{o}bius equivalence was given by Paulin \cite{Pau96} in terms of `cross-ratios'.
Given any four distinct points $a,b,c,d \in \partial X$, a cross-ratio $[a,b,c,d]$ roughly captures the distance $d_X(p_{abc},p_{acd})$  with some error, given in terms of  $\delta$ (see \autoref{lem_cross_ratios_vs_quasi_centres}). A  quasi-M\"{o}bius  map $f:\partial X\to \partial Y$ is a bijection that coarsely preserves the cross-ratios $[a,b,c,d]$ and $[f(a),f(b),f(c),f(d)]$ (see \autoref{defn_QM}). 
 In \cite{CCM2019}, Charney, Cordes, and Murray generalized Paulin's results to geodesic metric spaces admitting a geometric group action and having Morse boundary with at least three points.

The notion of relatively hyperbolic groups was proposed by Gromov and initially developed by Farb and Bowditch \cites{Gro87, Farb, Bowditch}.  We refer the reader to see the article \cite{Hruska} by Hruska for various equivalent definitions of relatively hyperbolic groups. In this article, we take the definition of relatively hyperbolic groups given by Groves and Manning \cite{GM2008}. 
Let $G$ be a finitely generated group, and let $\{H_1, H_2, \dots, H_n\}$ be a finite collection of finitely generated, infinite subgroups of infinite index in $G$. 
Choose a finite generating set $S$ of $G$ such that, for each $i=1,\dots,n$, the intersection $H_i \cap S$ generates $H_i$.
Let $X$ be the Cayley graph of $G$ with respect to $S$ and $\mathcal{H}_G$ be the collection of all left cosets of $H_1, H_2,\dots, H_n$.
For each coset $H\in \mathcal{H}_G$, the combinatorial horoball $H^h$ (defined in \cite{GM2008})  is analogous to a horoball in the hyperbolic space $\mathbb{H}^n$. We say $(G,\mathcal{H}_G)$  is a $\delta$-relatively hyperbolic group, or that $G$ is $\delta$-hyperbolic relative to the collection $\mathcal{H}_G$, if the space $X^h$ obtained by attaching combinatorial horoballs $H^h$ to each $H\in \mathcal{H}_G$ is a $\delta$-hyperbolic metric space. The pair $(G,\mathcal{H}_G)$ is called a relatively hyperbolic group if it is $\delta$-relatively hyperbolic for some $\delta\geq 0$. The Gromov boundary $\partial X^h$ of $X^h$ was studied by Bowditch \cite{Bowditch}, and called the Bowditch boundary of the relatively hyperbolic group $(G, \mathcal{H}_G)$. Observe that in our definition of a relatively hyperbolic group $(G,\mathcal{H}_G)$, each subgroup in $\{H_1,\ldots,H_n\}$ is assumed to have infinite index in $G$. Consequently, the Bowditch boundary has at least three points, and thus the group is non-elementary in the sense of our definition. The set of limit points of a set $H\in\mathcal{H}_G$ in $\partial X^h$ is a singleton (follows from Lemma 5.6, Lemma 3.11 of \cite{GM2008}). A point $p \in \partial X^h$ is said to be a parabolic endpoint if there exists $H_{p}\in \mathcal{H}_G$ such that the limit point of $H_{p}$ is $p$. 

Quasisymmetric maps between two metric spaces are generalizations of bi-Lipschitz maps and control the ratios of distances between pairs of points. In several cases, quasi-M\"{o}bius equivalence is the same as quasisymmetric equivalence. Mackay and Sisto, in \cite{Mackay-Sisto}, proved that a shadow respecting quasisymmetric homeomorphism between Bowditch boundaries of relatively hyperbolic groups induces a coarsely cusp-preserving quasi-isometry (see \autoref{defn_cusp_preserving}) between the relatively hyperbolic groups and vice versa (see \cite{Mackay-Sisto}*{Corollary 1.3}). In a similar theme, Hruska-Healy \cite{Healy-Hruska}
independently proved that a coarsely cusp-preserving quasi-isometry between relatively hyperbolic groups induces a quasi-isometry between cusped spaces and a quasisymmetric map between their Bowditch boundaries. 
This article aims to generalize  Paulin's idea in \cite{Pau96} to relatively hyperbolic groups. We will give two criteria to show that a homeomorphism between Bowditch boundaries of relatively hyperbolic groups induces a quasi-isometry between relatively hyperbolic groups. The first criterion is essentially the generalizations of Paulin's results to the setting of relatively hyperbolic groups, while the second criterion is based on the internal geometry of the groups. For the first criterion, we have introduced the terms `relative cross-ratio' and `relative quasi-M\"{o}bius maps', and for the second criterion, we have introduced the notion of `linear distortion of exit points'. 

There exist hyperbolic and relatively hyperbolic groups with homeomorphic boundaries that are not quasi-isometric to each other. Let $\mathbb{H}^4_{\mathbb R}$ denote the 4-dimensional real hyperbolic space and $\mathbb{H}^2_{\mathbb C}$ denote the 2-dimensional complex hyperbolic space. The Gromov boundaries  $\partial\mathbb{H}^4_{\mathbb R}$ and $\partial\mathbb{H}^2_{\mathbb C}$ are homeomorphic to $\mathbb S^3$, but their conformal dimensions are different. Hence, $\mathbb{H}^4_{\mathbb R}$ and $\mathbb{H}^2_{\mathbb C}$ are not quasi-isometric (see \cite{Mackay-Tyson}*{Theorem 3.2.13, Theorem 3.3.3}). If $G_1$ is the fundamental group of a closed real hyperbolic $4$-manifold and $G_2$ is the fundamental group of a closed complex hyperbolic $2$-manifold, then $G_1$ and $G_2$ act properly discontinuously and cocompactly by isometries on $\mathbb{H}^4_{\mathbb R}$ and $\mathbb{H}^2_{\mathbb C}$, respectively. By the Milnor-\v{S}varc Lemma, it follows that \(G_1\) is quasi-isometric to $\mathbb{H}^4_{\mathbb R}$ and \(G_2\) is quasi-isometric to $\mathbb{H}^2_{\mathbb C}$. Hence, the groups $G_1$ and $G_2$ are hyperbolic, and their Gromov boundaries \(\partial G_1\) and \(\partial G_2\) are homeomorphic to $\partial\mathbb{H}^4_{\mathbb R}$ and $\partial\mathbb{H}^2_{\mathbb C}$, respectively. Although the Gromov boundaries of $G_1$ and $G_2$ are homeomorphic to \(\mathbb{S}^3\), the groups are not quasi-isometric. 
Given any $k\geq 1$, Kolpakov and Martelli \cite{Kolpakov-Martelli} constructed infinitely many real hyperbolic $4$-manifolds with $k$ many cusps. Thus, given a finite-volume complex hyperbolic manifold  $M_\mathbb{C}^2$ of dimension two, one can construct a real hyperbolic manifold $N_\mathbb{R}^4$ of dimension four such that $M_\mathbb{C}^2$ and $N_\mathbb{R}^4$ have the same number of cusps. The fundamental groups $\pi_1(M_\mathbb{C}^2)$ and $\pi_1(N_\mathbb{R}^4)$ are two relatively hyperbolic groups acting properly discontinuously by isometries on $\mathbb{H}^2_{\mathbb C}$ and $\mathbb{H}^4_{\mathbb R}$, respectively. The Bowditch boundaries of $\pi_1(M_\mathbb{C}^2)$ and $\pi_1(N_\mathbb{R}^4)$ are $\partial\mathbb{H}^2_{\mathbb C}$ and $\partial\mathbb{H}^4_{\mathbb R} $, respectively, and they are homeomorphic to $\mathbb{S}^3$. There does not exist any cusp-preserving quasi-isometry between $\pi_1(M_\mathbb{C}^2)$ and $\pi_1(N_\mathbb{R}^4)$, since any such quasi-isometry would extend to a quasisymmetric homeomorphism between $\partial\mathbb{H}^2_{\mathbb C}$ and  $\partial\mathbb{H}^4_{\mathbb R} $ (see \cite{Mackay-Sisto}, \cite{Healy-Hruska}), contradicting the fact that their conformal dimensions differ.
Another relevant result towards this direction is by R. Schwartz, who showed that in a rank-one Lie group $G\neq \mathrm{Isom}(\mathbb{H}^2)$, two non-uniform lattices are quasi-isometric if and only if they are commensurable \cite{Schwartz}*{Corollary 1.3}. This implies that there are finite-volume cusped real hyperbolic $3$-manifolds whose fundamental groups are not quasi-isometric, despite their Bowditch boundaries being topologically the same as the unit sphere $\mathbb S^2$. Thus, by \autoref{thm_main_theorem_1}, there does not exist any relative quasi-M\"{o}bius map between their Bowditch boundaries.

Let $(G,\mathcal{H}_G)$ be a relatively hyperbolic group with Cayley graph $X$ and cusped space $X^h$. Let $a, b, c_H\in \partial X^h$ be three distinct points, where $c_H$ is a parabolic endpoint associated to a horosphere $H \in \mathcal{H}_G$. The relative cross-ratio of $a,b,c_H$ is defined by 
\[[a,b,c_H] \coloneqq \sup \liminf_{i\to\infty} \frac{1}{2} \left\{d_{X^h}(a_i,H)+d_{X^h}(b_i,H)-d_{X^h}(a_i,b_i)\right\},\] 
where the supremum is taken over all sequences $a_i,b_i \in X^h$ such that $a_i \to a$ and $b_i \to b$, and $d_{X^h}(x,H)$ denotes the distance from a point $x \in X^h$ to $H$.  The absolute value of the relative cross-ratio $[a,b,c_H]$ coarsely measures the distance between a quasi-centre of the triple $(a,b,c_H)$ and $H$.

Let $(G_1,\mathcal{H}_{G_1})$ and $(G_2,\mathcal{H}_{G_2})$ be two relatively hyperbolic groups with associated cusped spaces $X^h$ and $Y^h$, respectively. A quasi-M\"{o}bius map $f:\partial X^h\to\partial Y^h$ that preserves parabolic endpoints is called relative quasi-M\"{o}bius if, for any distinct $a,b,c\in \partial X^h$ with $c$ parabolic, the quantities $|[a,b,c]|$ and $|[f(a),f(b),f(c)]|$ are coarsely comparable (see \autoref{defn_relative_quasi-Mobius}, \autoref{notation_coarsely_comparable}). The following theorem states that a coarsely cusp-preserving quasi-isometry induces a relative quasi-M\"{o}bius map (the proof follows from \autoref{thm_qi_implies_qm}, \autoref{thm_cusp_preserving_qi_imples_rel_QM}).

\begin{thmA}\label{thm_main_theorem_1}
    Suppose the numbers $\delta, K,\epsilon \geq 0$ and $\lambda\geq 1$ are given. Let $(G_1,\mathcal{H}_{G_1})$ and $(G_2,\mathcal{H}_{G_2})$ be two $\delta$-relatively hyperbolic groups. 
    If $\varphi: G_1 \to G_2$ is a $K$-coarsely cusp-preserving $(\lambda,\epsilon)$-quasi-isometry, then its boundary extension map $\partial\varphi^h:\partial X^h \to \partial Y^h$ is a relative $\psi$-quasi-M\"{o}bius map for some distortion function $\psi :[0,\infty) \to [0,\infty)$. Moreover, $\psi(t)$ is given by $\psi(t) = At + B$ for some constants $A$ and $B$ depending on $\delta, \lambda, \epsilon$ and $K$.
\end{thmA}

Left multiplication by any element of a finitely generated group induces an isometry of its Cayley graph, which is a $(1,0)$-quasi-isometry. In view of \autoref{thm_main_theorem_1}, there exists a continuous map $\psi :[0,\infty)\to [0,\infty)$ such that for any element $g$ of a relatively hyperbolic group $G$, the left multiplication $L_g$ induces a $\psi$-quasi-M\"{o}bius map $\partial L_g^h:\partial X^h\to\partial X^h$. Thus, the action of a relatively hyperbolic group on itself induces a uniform quasi-M\"{o}bius action on its Bowditch boundary.

We state below the converse of \autoref{thm_main_theorem_1}, whose proof is obtained from  \autoref{thm_QM_implies_qi_on_cusped_spaces}, \autoref{thm_rel_QM_implies_cusp_preserving_qi_on_cusped_spaces}, \autoref{prop_boundaries_coincide_1}, \autoref{thm_cusp_preserving_qi_implies_ambient_qi} and \autoref{prop_boundaries_coincide}
.
\begin{thmA}\label{thm_main_theorem_2}
    Let $(G_1,\mathcal{H}_{G_1})$ and $(G_2,\mathcal{H}_{G_2})$ be two $\delta$-relatively hyperbolic groups for some $\delta\geq 0$.
    Let $f:\partial X^h\to\partial Y^h$ be a relative $\psi$-quasi-M\"{o}bius map, for some distortion function $\psi :[0,\infty) \to [0,\infty)$, which preserves parabolic endpoints. Then there exists a quasi-isometry $\varphi_f: G_1 \to G_2$ such that:
    \begin{enumerate}[$(i)$]
        \item The quasi-isometry constants of  $\varphi_f$ depend only on $\delta$ and $\psi$.
        \item $\varphi_f$ is $K$-coarsely cusp-preserving, for some $K \ge 0$  that depends on $\delta$ and $\psi$.
        \item The homeomorphism $\partial\varphi_f:\partial X^h\to \partial Y^h$ induced by $\varphi_f$ coincides with $f$.
    \end{enumerate}
\end{thmA}

Every hyperbolic group is relatively hyperbolic with respect to the empty collection of subgroups. In the absence of combinatorial horoballs, a relative quasi-M\"obius map reduces to a quasi-M\"obius map, and all results of \autoref{subsec_coarse_denseness_of_quasi_centres} and \autoref{subsec_proof_of_thm_2} apply verbatim, with the sole exception of \autoref{prop_a_point_on_geodesic_vs_QP}. In place of \autoref{prop_a_point_on_geodesic_vs_QP}, we refer the reader to \autoref{rem_a_point_vs_QP}, which provides a brief argument in the setting of hyperbolic groups. As a consequence, we recover Paulin’s quasi-isometric classification of hyperbolic groups in terms of their Gromov boundaries. Moreover, Paulin \cite{Pau96} showed that a quasi-M\"obius map is, in fact, a homeomorphism. 

\begin{corA}[\cite{Pau96}]
    Let $G_1$ and $G_2$ be non-elementary hyperbolic groups. Any quasi-isometry $\varphi \colon G_1 \to G_2$ induces a quasi-M\"obius map $\partial\varphi \colon \partial G_1 \to \partial G_2$. Conversely, given any quasi-M\"obius map $f \colon \partial G_1 \to \partial G_2$,  there exists a quasi-isometry $\varphi_f \colon G_1 \to G_2$ such that \(\partial \varphi_f = f\).
\end{corA}

Next, we give the second criterion for getting a quasi-isometry from a homeomorphism between the Bowditch boundaries of relatively hyperbolic groups, using exit points of bi-infinite geodesics to combinatorial horoballs. Consider the set 
\[
\Theta(G)=\{(a,b) \mid a,b \in \partial X^h, \text{ $a$ is a parabolic endpoint and } a \ne b \}.
\]
Given a pair $(a,b)\in \Theta(G)$, for each bi-infinite geodesic from $a$ to $b$ in the cusped space $X^h$, there exists a last exit point to $H_a \in \mathcal{H}_G$ (see \autoref{subsec_exit_points}). Let $\mathcal{E}(a,b)$ denote the set of all such last exit points in $H_a$. The set $\mathcal{E}(a,b)$ has uniformly bounded diameter, independent of the choice of $(a,b)$. 
For each $(a,b) \in \Theta(G)$, let $\vartheta_{G}:\Theta(G) \to G$ be defined by sending $(a,b)$ to a point in $\mathcal{E}(a,b)$.
Consider two $\delta$-relatively hyperbolic groups $(G_1,\mathcal{H}_{G_1})$ and $(G_2,\mathcal{H}_{G_2})$. Let $f:\partial X^h \to \partial Y^h$ be a homeomorphism such that both $f$ and $f^{-1}$ preserve parabolic endpoints. We say that $f:\partial X^h\to\partial Y^h$ \emph{linearly distorts exit points} if  for all $(a,b),(c,d)\in \Theta(G_1)$, the distances $d_{G_1}(\vartheta_{G_1}(a,b),\vartheta_{G_1}(c,d))$ and $d_{G_2}(\vartheta_{G_2}(f(a),f(b)),\vartheta_{G_2}(f(c),f(d)))$ are `coarsely comparable' (see \autoref{notation_coarsely_comparable}).

\begin{thmA}\label{thm_main_theorem_3} 
    Let $\delta, K,\epsilon \geq 0$ and $\lambda\geq 1$ be given. Let $(G_1,\mathcal{H}_{G_1})$ and $(G_2,\mathcal{H}_{G_2})$ be two $\delta$-relatively hyperbolic groups. If $\phi: G_1\to G_2$ is a $K$-coarsely cusp-preserving $(\lambda,\epsilon)$-quasi-isometry,  then its boundary extension map $\partial\phi^h:\partial X^h\to\partial Y^h$ is a homeomorphism which linearly distorts exit points, and the constants of the distortion depend only on $K,\delta,\lambda$ and $\epsilon$.
\end{thmA}

\begin{thmA}\label{thm_main_theorem_4}
    Let $(G_1,\mathcal{H}_{G_1})$ and $(G_2,\mathcal{H}_{G_2})$ be two $\delta$-relatively hyperbolic groups for some $\delta\ge 0$. Let $f:\partial X^h\to\partial Y^h$ be a homeomorphism preserving parabolic endpoints. If $f$ linearly distorts exit points, then there exists a quasi-isometry $\phi_f: G_1\to G_2$ such that:
    \begin{enumerate}[$(i)$]
        \item The quasi-isometry constants of  $\phi_f$ depend only on $\delta$ and the constants of distortion.
        
        \item $\phi_f$ is $K$-coarsely cusp-preserving for some $K \ge 0$  that depends on $\delta$ and the constants of distortion. 
        
        \item The homeomorphism $\partial\phi_f:\partial X^h\to \partial Y^h$ induced by $\phi_f$ coincides with $f$.
    \end{enumerate}
\end{thmA}

In view of \autoref{thm_main_theorem_1}, \autoref{thm_main_theorem_2}, \autoref{thm_main_theorem_3}, \autoref{thm_main_theorem_4} and \cite{Mackay-Sisto}*{Corollary 1.3}, we have the following equivalence. 
\begin{corA}
    Let $(G_1,\mathcal{H}_{G_1})$ and $(G_2,\mathcal{H}_{G_2})$ be two relatively hyperbolic groups, and let $f: \partial X^h \to \partial Y^h$ be a homeomorphism. Then the following are equivalent:
    \begin{enumerate}[$(i)$]
        \item $f$ is relative quasi-M\"{o}bius,
        \item $f$ distorts exit points linearly,
        \item $f$ is a shadow-respecting quasisymmetry.  
    \end{enumerate}
\end{corA}


\subsection{Outline of the Paper}
\autoref{sec_Preliminaries} contains the necessary preliminaries on hyperbolic and relatively hyperbolic groups. In \autoref{sec_QC_and_QM}, we define \emph{quasi-centres}, \emph{cross-ratios}, and \emph{quasi-M\"{o}bius maps}, and recall Paulin's results \cite{Pau96} for hyperbolic metric spaces (or groups). \autoref{sec_QM_and_QI} is devoted to extending these results to the setting of relatively hyperbolic groups. In particular, in \autoref{subsec_rel_QM}, we introduce the notions of \emph{relative cross-ratio} and \emph{relative quasi-M\"{o}bius maps}. \autoref{thm_main_theorem_1} is proved in \autoref{subsec_proof_of_thm_1}, while \autoref{thm_main_theorem_2} is proved in \autoref{subsec_proof_of_thm_2}. Finally, in \autoref{sec_linear_distortion_and_rel_hyp}, we introduce the concept of linear distortion of exit points of a homeomorphism preserving parabolic endpoints, and establish \autoref{thm_main_theorem_3} and \autoref{thm_main_theorem_4}.

\subsection{Notations.}\label{Notations}
\begin{itemize}
    \item Let $A$ and $B$ be two real numbers. 
\begin{itemize}
    \item $A \approx_K B$ means $B-K \le A \le B + K$.
    \item $A \precapprox_{\lambda,K} B$ means $A \le \lambda B + K$.
    \item  $A \precapprox B$ means $A \precapprox_{\lambda,K} B$ for some $\lambda \ge 1$, $K \ge 0$.
    \item $A \approxeq_{\lambda,K} B$ means $A \precapprox_{\lambda,K} B$ and $B \precapprox_{\lambda,K} A$. 
    \item\label{notation_coarsely_comparable} $A$ and $B$ coarsely comparable or  $A \approxeq B$ means $A \approxeq_{\lambda,K} B$ for some $\lambda \ge 1, K\geq 0$.
\end{itemize}
\item We will occasionally use the notation $fx$ for $f(x)$ when needed.
\item $N_K(Z)$ will denote the closed $K$-neighbourhood of a set $Z$ in a metric space.
\item For $x,y\in X\cup\partial X$, where $X$ is a proper hyperbolic metric space, we write $[x,y]$ for a geodesic joining $x$ and $y$.
\end{itemize}

\subsection{Acknowledgements.} The authors sincerely thank the anonymous referee(s) for their careful review and valuable comments, which significantly improved the exposition and clarity of this article.


\section{Preliminaries } \label{sec_Preliminaries} 

\subsection{Hyperbolic Metric Spaces}


    In a metric space $(X, d_X)$, a \emph{geodesic} between two points $x$ and $y$, denoted by $[x, y]$, is a continuous map $\gamma : [0, d_X(x, y)] \to X$ such that $\gamma(0) = x$, $\gamma(d_X(x, y)) = y$, and $d_X(\gamma(s),\gamma(t))=|s-t| \text{ for all } s,t\in[0,d_X(x,y)]$. The space $(X, d_X)$ is called a \emph{geodesic metric space} if, for every pair of points $x, y \in X$, there exists a geodesic $[x, y]$ joining them. A \emph{geodesic triangle} in $X$ consists of three points $a, b, c \in X$ (called \emph{vertices}) and three geodesic segments $[a, b]$, $[b, c]$, and $[c, a]$ (called \emph{sides}) joining them. A geodesic triangle with vertices $a, b, c \in X$ is denoted by $\triangle(a, b, c)$.

\begin{defn}
    Let $(X,d_X)$ be a geodesic metric space.
    \begin{itemize}        
        \item Slim triangle: Let $\delta \ge 0$. Given points $a, b, c \in X$, a geodesic triangle $\triangle(a, b, c)$ is said to be $\delta$-slim if each side of the triangle is contained in the closed $\delta$-neighbourhood of the union of the other two sides.
        
        \item Hyperbolic metric space: A geodesic metric space $(X, d_X)$ is said to be $\delta$-hyperbolic for some $\delta \ge 0$ if all geodesic triangles in $X$ are $\delta$-slim. A geodesic metric space is said to be hyperbolic if it is $\delta$-hyperbolic for some $\delta \ge 0$.
        
        \item Hyperbolic group: A finitely generated group $G$ is said to be hyperbolic if its Cayley graph with respect to some finite generating set is a hyperbolic metric space. 
    \end{itemize}
\end{defn}

\begin{defn}[Quasi-isometry] \label{defn_qi}
    Let $\lambda \geq 1$ and $\epsilon \geq 0$. A map $f: X \to Y$ (which may not be continuous) between two metric spaces $(X,d_X)$ and $(Y,d_Y)$ is said to be a $(\lambda, \epsilon)$-quasi-isometric embedding if $d_Y(f(x_1),f(x_2)) \approxeq_{\lambda,\epsilon} d_X(x_1,x_2)$ for all $x_1,x_2 \in X$.
    Moreover, if $f$ is coarsely surjective; that is, for every $y \in Y$ there exists an $x\in X$ satisfying $d_Y(f(x),y) \le \epsilon$, then $f$ is said to be a $(\lambda, \epsilon)$-quasi-isometry. Two metric spaces $X$ and $Y$ are said to be quasi-isometric if there exists a quasi-isometry between them. 
\end{defn}

Note that in \autoref{defn_qi}, we have slightly deviated from the customary definition of a quasi-isometric embedding, which is usually stated as 
\[
\frac{1}{\lambda'}d_X(x_1,x_2)-\epsilon' \le d_Y(f(x_1),f(x_2)) \le  \lambda' d_X(x_1,x_2)+\epsilon',
\] 
for all $x_1,x_2 \in X$. This modification is made to introduce a convenient shorthand notation for a quasi-isometric embedding. However, by adjusting the constants, both definitions are equivalent.

\begin{defn}[Quasigeodesic] 
    Let $\lambda \geq 1$ and $\epsilon \geq 0$. A $(\lambda, \epsilon)$-quasigeodesic in a metric space $X$ is a $(\lambda, \epsilon)$-quasi-isometric embedding $\gamma: I \to X$, where $I \subseteq \mathbb{R}$ is an interval (not necessarily closed or bounded). A $(\lambda, \lambda)$-quasigeodesic is also called a $\lambda$-quasigeodesic.
\end{defn}
 
\begin{ex}\label{ex_(3,0)_quasigeodesic}
    Let $[a,b]$ be a geodesic segment in a metric space $X$, and let $z \in [a,b]$ be a nearest point projection of a point $p \in X$ onto $[a,b]$. Then the paths $[p,z] \cup [z,a]$ and $[p,z] \cup [z,b]$ are $(3,0)$-quasigeodesics.
\end{ex}

In a hyperbolic metric space, any geodesic and any quasigeodesic with the same endpoints remain within a uniformly bounded neighbourhood of each other. This property is known as the Morse Lemma, or the stability of quasigeodesics. As a consequence, hyperbolicity is invariant under quasi-isometries, which ensures that the notion of a hyperbolic group is well defined regardless of the choice of finite generating set. Below, we state the precise form of the stability of quasigeodesics, including the involved constants.

\begin{prop}[Stability of quasigeodesics, \cite{BH99}] \label{Prop_Stability_of_quasigeodesics}
    For all $\delta, \epsilon \ge 0$ and $\lambda \ge 1$, there exists a constant $K_0 = K_0(\delta, \epsilon, \lambda) \ge 0$ such that the following holds: If $X$ is a $\delta$-hyperbolic metric space and $\gamma$ is a $(\lambda, \epsilon)$-quasigeodesic in $X$ with endpoints $x, y \in X$, then any geodesic joining $x$ and $y$ and the quasigeodesic $\gamma$ lie within the $K_0$-neighbourhood of each other.
\end{prop} 

It is a standard fact that in a hyperbolic metric space, if a geodesic stays outside a sufficiently large neighbourhood of a quasi-convex set, then the set of nearest-point projections of the geodesic onto the quasi-convex set has uniformly bounded diameter, and the bound depends only on the hyperbolicity constant (see \cite{arzhan}*{Corollary 7.2}). In particular, if two geodesics are a little far away, then we have the following result.

\begin{lem}[Bounded projection lemma] \label{lem_Bounded_projection_lemma}
    Given $\delta \ge 0$, there exist constants $P = P(\delta) \ge 0$ and $K_1 = K_1(\delta) \ge 0$ such that the following holds: Let $\gamma$ be a geodesic in a $\delta$-hyperbolic metric space $X$. If $\alpha$ is another geodesic in $X$ lying outside the $P$-neighbourhood of $\gamma$, then the set of nearest-point projections of $\alpha$ onto $\gamma$ has diameter at most $K_1$.
\end{lem}


\subsection{Gromov Boundary of a Hyperbolic Metric Space}

A metric space $(X, d_X)$ is said to be proper if every closed ball in $X$ is compact. Throughout, let $(X, d_X)$ be a proper hyperbolic metric space.

\begin{defn}[Gromov product]
   Let $a, b, w \in X$. The Gromov product of $a$ and $b$ with respect to $w$ is defined by $$(a,b)_w = \frac{1}{2}(d_X(a,w) + d_X(b,w) - d_X(a,b)).$$
\end{defn} 

Two geodesic rays $\gamma_1, \gamma_2 : [0, \infty) \to X$ are said to be asymptotic if $\sup\{d_X(\gamma_1(t),\gamma_2(t))\,:\, t \ge 0\}$ is finite. Being asymptotic defines an equivalence relation on the set of geodesic rays in $X$. The equivalence class of a geodesic ray $\gamma$ is denoted by $[\gamma]$. 

\begin{defn}[Gromov boundary] 
    The Gromov boundary of $X$ is defined to be
    $$\partial X \coloneqq \{[\gamma] \mid \gamma:[0,\infty)\to X \text{ is a geodesic ray}\}.$$
\end{defn}

Fix a basepoint $w \in X$. For any $a \in \partial X$ and $r \geq 0$, define the set 
\begin{gather*}
    V(a,r)\coloneqq\{ b\in\partial X \mid \liminf_{t\to\infty} (\gamma_1(t),\gamma_2(t))_w \ge r, \text{ for some geodesic rays $\gamma_1,\gamma_2$ with } \\ [\gamma_1]=a,\; [\gamma_2]=b,\; \gamma_1(0)=\gamma_2(0)=w \}
\end{gather*} 
Let $\overline{X} = X \cup \partial X$. We equip $\overline{X}$ with the topology defined in \cite{BH99}*{Definition 3.5}, in which the family $\{V(a,r) \mid r \ge 0\}$ forms a basis of neighbourhoods at $a \in \partial X$.
This topology on $\overline{X}$ is independent of the choice of basepoint $w$ (see \cite{BH99}*{Proposition 3.7}). The boundary $\partial X$ is visible; that is, any pair of distinct points in $\partial X$ can be joined by a bi-infinite geodesic (see \cite{BH99}*{Chapter III.H, Lemma 3.2}). Moreover, any two geodesics in $X$ with the same endpoints in $\overline{X}$ lie within a $2\delta$-neighbourhood of each other (see \cite{BH99}*{Chapter III.H, Lemma 3.3}). 
For a hyperbolic group $G$, the Gromov boundary $\partial G$ is defined to be the Gromov boundary of any Cayley graph of $G$ with respect to a finite generating set. This is well-defined up to homeomorphism, since the Gromov boundary is invariant under quasi-isometry (see the following theorem).

\begin{thm}[Quasi-isometry invariants, \cite{BH99}]
    Let $X$ and $Y$ be proper $\delta$-hyperbolic geodesic metric spaces for some $\delta \geq 0$. If $\varphi: X \to Y$ is a quasi-isometric embedding, then $\varphi$ induces a continuous injective map $\partial \varphi: \partial X \to \partial Y$, given by $[\gamma] \mapsto [\varphi \circ \gamma]$. Moreover, if $\varphi: X \to Y$ is a quasi-isometry, then $\partial \varphi$ is a homeomorphism.
\end{thm}


\subsection{Relatively Hyperbolic Groups and Boundaries}

In this subsection, we introduce relatively hyperbolic groups in the sense of Groves and Manning \cite{GM2008}, together with the notions of coarsely cusp-preserving maps and Bowditch boundaries. We also recall several important results, including the visual boundedness of horoballs and the fact that coarsely cusp-preserving quasi-isometries between relatively hyperbolic groups induce homeomorphisms between their Bowditch boundaries.

\begin{defn}[Combinatorial horoball, \cite{GM2008}] \label{defn_horoball}
    Let $H$ be a locally finite graph with vertex set $V(H)$ and all edges of length one. Let $d_H$ denote the associated path metric on $H$. The combinatorial horoball based on $H$, denoted by $\mathcal{H}(H,d_H)$, is the graph defined as follows:
    (1)  The vertex set is $\mathcal{H}^{(0)} :=  V(H) \times \{\mathbb{N}\cup\{0\}\}$. \\ 
    (2) The edge set $\mathcal{H}^{(1)}$ consists of two types of edges:
    \begin{itemize}
        \item Vertical edges: For each $n \in \mathbb{N}\cup \{0\}$ and $x \in V(H)$, there is an edge between $(x, n)$ and $(x, n+1)$. 

        \item Horizontal edges: For each $x, y \in V(H)$,
        \begin{itemize}
            \item if \(e\) is an edge of $H$ joining \(x\) to \(y\), then there is a corresponding edge \(\overline{e}\) connecting \((x,0)\) to \((y,0)\), 

            \item  for each $n \in \mathbb{N}$, if $0 < d_H(x, y) \leq 2^n$, then there is a single edge between $(x, n)$ and $(y, n)$. 
        \end{itemize}
    \end{itemize}    
\end{defn}

Consider all the edges of $\mathcal{H}(H,d_H)$ to have unit length, and equip the graph with the induced path metric. With this metric, $\mathcal{H}(H, d_H)$ is a hyperbolic metric space (see \cite{GM2008}). For convenience, we henceforth denote this hyperbolic metric space by $(H^h, d_{H^h})$ and refer to it simply as a horoball, instead of a combinatorial horoball.


\begin{defn}[Relatively hyperbolic group] \label{defn_rel_hyp}   
    Let $G$ be a finitely generated group, and let $\{H_1, H_2, \dots, H_n\}$ be a finite collection of finitely generated, infinite subgroups of infinite index in $G$. Choose a finite generating set $S$ of $G$ such that, for each $i=1,\dots,n$, the intersection $H_i \cap S$ generates $H_i$. Let $X$ be the Cayley graph of $G$ with respect to the generating set $S$, and let $\mathcal{H}_G$ denote the collection of all left cosets of the subgroups $H_1, \dots, H_n$.
    We say that $G$ is $\delta$-hyperbolic relative to $\{H_1, H_2, \dots, H_n\}$, or simply that $(G, \mathcal{H}_G)$ is a $\delta$-relatively hyperbolic group, if the cusped space $X^h$, obtained by attaching a combinatorial horoball $H^h$ to each coset $H \in \mathcal{H}_G$, is a $\delta$-hyperbolic metric space. The group $G$ is said to be hyperbolic relative to $\{H_1, H_2, \dots, H_n\}$ if there exists $\delta \geq 0$ such that $(G, \mathcal{H}_G)$ is a $\delta$-relatively hyperbolic group.  
    The elements of $\mathcal{H}_G$ are referred to as horospheres, and the conjugates of the subgroups $H_1, \dots, H_n$ are called the parabolic subgroups of $G$.
\end{defn}


Let $(G, \mathcal{H}_G)$ be a $\delta$-relatively hyperbolic group with a Cayley graph $X$, and let $X^h$ denote the cusped space of $X$. Suppose $H = gH_i \in \mathcal{H}_G$ is a left coset and $a,b \in H$. We denote by $d_H(a,b)$ the distance between $a$ and $b$ measured in the word metric on $H_i$. 

Next, we describe the notion of ‘preferred paths’, introduced by Groves and Manning \cite{GM2008}, in a horoball. 
 Suppose $x$ and $y$ are two points lying in the horoball $H^h$, for some $H \in \mathcal{H}_G$. Then there exist points $x_0, y_0 \in H$ and integers $n_1, n_2 \geq 0$ such that $x = (x_0, n_1)$ and $y = (y_0, n_2)$. 

If $x_0 = y_0$, then the vertical path 
\[\gamma_{x,y}\coloneqq [(x_0,n_1),(x_0,n_2)]\] 
joining $(x_0,n_1)$ and $(x_0,n_2)$ forms a geodesic in $X^h$. 
If $x_0 \ne y_0$, then there exists an integer $n \ge 0$ satisfying $2^{n-1} < d_H(x_0,y_0) \le 2^n$. Choose $N = \max\{n_1,n_2,n\}$ and consider the path 
\[\gamma_{x,y}\coloneqq[(x_0,n_1),(x_0,N)] \cup [(x_0,N),(y_0,N)] \cup [(y_0,N),(y_0,n_2)].\]
This path consists of (two) vertical segments and a horizontal segment at height $N$ (see \autoref{fig_preferred_path}).

\begin{figure}[H]
    \centering
    \begin{tikzpicture}
        \draw (-2,0) -- (2,0);
        \draw[gray] (-1.3,0) --(-.8,3) -- (.8,3) -- (1.3,0);
        \draw[gray] (-.8,3) -- (-.8,4);
        \draw[dashed,gray] (-.8,4) -- (-.8,5);
        \draw[gray] (.8,3) -- (.8,4);
        \draw[dashed,gray] (.8,4) -- (.8,5); 
        \draw[gray] (-.8,3.7) -- (.8,3.7);
        \draw[gray] (-.8,4.4) -- (.8,4.4);
        \draw[-stealth, gray] (-1.8,2) -- (-1.8,3);
        \draw[-stealth,gray] (-1.8,1.5) -- (-1.8,0);        
        \node at (-1.8,1.7) {$n$};
        \node at (0,3.3) {$1$};
        \node at (0,4) {$1$};
        \node at (0,4.7) {$1$};

        \node at (0,-.4) {$N = n_2 = n$};
        \draw[gray!20,line width=1.5 mm] (-1.13,1) --(-.8,3) -- (.8,3);
        \draw[thick] (-1.13,1) --(-.8,3) -- (.8,3);
        \filldraw (-1.13,1) circle (1pt);
        \filldraw (.8,3) circle (1pt);
        \node at (-.8,1) {$x$};
        \node at (1.1,3) {$y$};
    \end{tikzpicture} 
    \hspace{.3cm}
    \begin{tikzpicture}
        \draw (-2,0) -- (2,0);
        \draw[gray] (-1.3,0) --(-.8,3) -- (.8,3) -- (1.3,0);
        \draw[gray] (-.8,3) -- (-.8,4);
        \draw[dashed,gray] (-.8,4) -- (-.8,5);
        \draw[gray] (.8,3) -- (.8,4);
        \draw[dashed,gray] (.8,4) -- (.8,5); 
        \draw[gray] (-.8,3.7) -- (.8,3.7);
        \draw[gray] (-.8,4.4) -- (.8,4.4);
        \draw[-stealth, gray] (-1.8,2) -- (-1.8,3);
        \draw[-stealth,gray] (-1.8,1.5) -- (-1.8,0);
        \node at (-1.8,1.7) {$n$};
        \node at (0,3.3) {$1$};
        \node at (0,4) {$1$};
        \node at (0,4.7) {$1$};

        \node at (0,-.4) {$N = n_2 > n$};;
        \draw[gray!20,line width=1.5 mm] (-1.13,1) --(-.8,3) -- (-.8,3.7) -- (.8,3.7);
        \draw[thick] (-1.13,1) --(-.8,3) -- (-.8,3.7) -- (.8,3.7);
        \filldraw (-1.13,1) circle (1pt);
        \filldraw (.8,3.7) circle (1pt);
        \node at (-.8,1) {$x$};
        \node at (1.1,3.7) {$y$};
    \end{tikzpicture}
    \hspace{.3cm}
    \begin{tikzpicture}
        \draw (-2,0) -- (2,0);
        \draw[gray] (-1.3,0) --(-.8,3) -- (.8,3) -- (1.3,0);
        \draw[gray] (-.8,3) -- (-.8,4);
        \draw[dashed,gray] (-.8,4) -- (-.8,5);
        \draw[gray] (.8,3) -- (.8,4);
        \draw[dashed,gray] (.8,4) -- (.8,5); 
        \draw[gray] (-.8,3.7) -- (.8,3.7);
        \draw[gray] (-.8,4.4) -- (.8,4.4);
        \draw[-stealth, gray] (-1.8,2) -- (-1.8,3);
        \draw[-stealth,gray] (-1.8,1.5) -- (-1.8,0);
        \node at (-1.8,1.7) {$n$};
        \node at (0,3.3) {$1$};
        \node at (0,4) {$1$};
        \node at (0,4.7) {$1$};

        \node at (0,-.4) {$N = n > n_1,n_2$};
        \draw[gray!20,line width=1.5 mm] (-1.13,1) --(-.8,3) -- (.8,3) -- (1.005,1.8);
        \draw[thick] (-1.13,1) --(-.8,3) -- (.8,3) -- (1.005,1.8);
        \filldraw (-1.13,1) circle (1pt);
        \filldraw (1.005,1.8) circle (1pt);
        \node at (-.8,1) {$x$};
        \node at (1.3,1.8) {$y$};       
    \end{tikzpicture}
    \caption{}
    \label{fig_preferred_path}
\end{figure} 

In the first two cases illustrated in \autoref{fig_preferred_path}, the path $\gamma_{x,y}$ is a geodesic in the cusped space $X^h$. In the last case, although $\gamma_{x,y}$ may not be a geodesic, it remains uniformly close to one within $X^h$. In each case, such a path $\gamma_{x,y}$ is called a \textit{preferred path} in $X^h$ between $x$ and $y$. 

\begin{prop}[\cite{GM2008}*{Lemma 5.6}] \label{lem_preferred_path_vs_geodesics}
    Suppose $x$ and $y$ lie in the same horoball $H^h$, for some $H \in \mathcal{H}_G$. Then the preferred path $\gamma_{x,y}$ is at Hausdorff distance at most $5$ from any geodesic connecting $x$ and $y$ in $X^h$.
\end{prop} 


For any $x,y\in H\in \mathcal{H}_G$ if $d_H(x,y) = 2^n$, then the length of the preferred path $\gamma_{x,y}$ is $2n+1$. Every $H \in \mathcal{H}_G$ is uniformly properly embedded in the Cayley graph $X$. Thus, by applying \autoref{lem_preferred_path_vs_geodesics}, we have the following lemma.

\begin{lem}\label{lem_Uniformly_Properly_embedded_Lemma}
    The Cayley graph $X$ is properly embedded in $X^h$. That is, for any $r \geq 0$, there exists a constant $K_2 = K_2(r) \ge 0$ such that for any $x, y \in X$, if $d_{X^h}(x, y) \le r$, then $d_X(x, y) \le K_2$.   
\end{lem}
 
Note that all groups appearing in the pair $(G,\mathcal{H}_G)$ are finitely generated. Hence, the associated cusped space is locally finite and proper (see \cite{GM2008}*{Remark~3.14}).

\begin{defn}[Bowditch boundary, \cite{Bowditch}] 
    Let $(G,\mathcal{H}_G)$ be a relatively hyperbolic group with a Cayley graph $X$. The Gromov boundary $\partial X^h$ of the hyperbolic metric space $X^h$ is called the Bowditch boundary of $G$ (or of $X$). A point $a \in \partial X^h$ is said to be a parabolic endpoint if it is a Gromov boundary of some horoball $H^h$ for $H\in\mathcal{H}_{G}$.
\end{defn}

\noindent Let $(G,\mathcal{H}_G)$ be a group as in \autoref{defn_rel_hyp} (not necessarily relatively hyperbolic). The identity map on $G$, when $G$ is equipped with word metrics corresponding to two different finite generating sets, is $\lambda$-bi-Lipschitz for some $\lambda \geq 1$. 
This identity map fixes each element of $\mathcal{H}_G$ pointwise and extends naturally to an identity map on the vertex sets of the associated cusped spaces, preserving the lengths of vertical segments. 
Moreover, the lengths of preferred paths joining the same pair of points in a combinatorial horoball in the two cusped spaces differ by at most $2\log_2 \lambda$. 
It follows that the associated cusped spaces are quasi-isometric. 
In particular, the definitions of relative hyperbolicity and of the Bowditch boundary are independent of the choice of finite generating set.

Dru\c{t}u \cite{Drutu}*{Theorem 1.2} proved that relative hyperbolicity is a quasi-isometry invariant in the following sense. Let $\varphi : G_1 \to G_2$ be a quasi-isometry between two finitely generated groups $G_1$ and $G_2$, and suppose that $G_1$ is hyperbolic relative to a finite collection of subgroups ${H_1, \dots, H_n}$. Then there exists a finite collection of subgroups ${H_1', \dots, H_m'}$ of $G_2$ such that $G_2$ is hyperbolic relative to $\{H'_1, \dots, H'_m\}$, and each $H'_i$ is quasi-isometrically embedded in some $H_j$ for $j=j(i)\in\{1,\dots,n\}$. 

\begin{defn}[Coarsely cusp-preserving quasi-isometry] \label{defn_cusp_preserving}
    Let $(G_1,\mathcal{H}_{G_1})$ and $(G_2,\mathcal{H}_{G_2})$ be two relatively hyperbolic groups, with Cayley graphs $X$ and $Y$, respectively. Given $K \ge 0$, a quasi-isometry $\varphi : X \to Y$ is said to be $K$-coarsely cusp-preserving if the following conditions hold:
    \begin{enumerate}[$(i)$]
        \item For every $H \in \mathcal{H}_{G_1}$, there exists $H' \in \mathcal{H}_{G_2}$ such that the image $\varphi(H)$ is contained in the $K$-neighbourhood of $H'$ in $Y$. 

        \item For every $H' \in \mathcal{H}_{G_2}$, there exists $H \in \mathcal{H}_{G_1}$ such that $\varphi^{-1}(H')$ is contained in the $K$-neighbourhood of $H$ in $X$, where $\varphi^{-1}$ is a quasi-isometry inverse of the quasi-isometry $\varphi$.
    \end{enumerate}
\end{defn}

  
\begin{prop}[\cite{Mackay-Sisto}*{Proposition 5.5, Corollary 1.2}, \cite{Pal}*{ Lemma 1.2.31}]\label{thm_cusp_prev_qi_implies_homeo}
    Given $\lambda \ge 1$ and $\epsilon, K \ge 0$, there exist constants $\lambda'=\lambda'(\epsilon, \lambda, K) \ge 1$ and $\epsilon' = \epsilon'( \epsilon, \lambda, K) \ge 0$ such that the following holds. Let $(G_1,\mathcal{H}_{G_1})$ and $(G_2,\mathcal{H}_{G_2})$ be two relatively hyperbolic groups with Cayley graphs $X$ and $Y$, respectively, and suppose that $\varphi: X \to Y$ is a $K$-coarsely cusp-preserving $(\lambda, \epsilon)$-quasi-isometry. Then:
    \begin{enumerate}[$(i)$]
        \item The map $\varphi$ extends to a $(\lambda', \epsilon')$-quasi-isometry $\varphi^h: X^h \to Y^h$ with the following property. Let $H_1 \in \mathcal{H}_{G_1}$ and suppose that $\varphi(H_1)$ is contained in the $K$-neighbourhood of some  $H_2\in \mathcal{H}_{G_2}$. Then for every $(w,t) \in H_1^h$, we have $\varphi^h(w,t)=\gamma_{\varphi(w)}(t)$, where $\gamma_{\varphi(w)}:[0,\infty)\to Y^h$ is a geodesic ray in $Y^h$ with $\gamma_{\varphi(w)}(0)=\varphi(w)$, and there exists $t_0\le K$ such that $\gamma_{\varphi(w)}\big|_{[t_0,\infty)}$ is a vertical geodesic ray in $H_2^h$ with $\gamma_{\varphi(w)}(t_0)\in H_2$.

        \item The boundary extension map $\partial \varphi^h: \partial X^h \to \partial Y^h$ is a homeomorphism.
    \end{enumerate}
\end{prop}

Below, we give examples related to coarsely cusp-preserving quasi-isometries that illustrate the conclusions of \autoref{thm_main_theorem_1}.

\begin{ex}
    \noindent $(i)$  Consider the free groups \(G_1=\mathbb{F}(a,b)\) and \(G_2=\mathbb{F}(c,d)\), where the generators $a,b,c,$ and $d$ act as parabolic isometries of the Poincar\'{e} disc $\mathbb{D}$, as illustrated in \autoref{fig:poincare_isometries}. 

    \begin{figure}[h]    
    \centering
    \begin{tikzpicture}[scale=.8]
        \draw (0,0) circle (2.5 cm);

        \draw (2.5,0) arc (270:180:2.5);
        \draw (2.5,0) arc (90:180:2.5);
        \draw (-2.5,0) arc (90:0:2.5);
        \draw (-2.5,0) arc (-90:0:2.5);

        \foreach \i in {0,1,2,3} {
        \draw
        ({2.5*cos(35+90*\i)}, {2.5*sin(35+90*\i)}) arc ({305+90*\i}:{145+90*\i}:.45);
        }
        \foreach \i in {0,1,2,3} {
        \draw[bend left=50] ({2.5* cos (90*\i)}, {2.5*sin (90*\i)}) edge ({2.5* cos (35+90*\i)}, {2.5*sin (35+90*\i)});
        \draw[bend left=50] ({2.5* cos (55.55+90*\i)}, {2.5*sin (55.5+90*\i)}) edge ({2.5* cos (90+90*\i)}, {2.5*sin (90+90*\i)});
        }

        
        \draw (.8*cos 135, .8*sin 135) edge[->, bend right=50, >=stealth] (.8*cos 45, .8*sin 45);
        \draw (.8*cos 225, .8*sin 225) edge[<-, bend left=50, >=stealth] (.8*cos 315, .8*sin 315);

        \node at (0,0.7) {a};
        \node at (0,-0.7) {b};

    \end{tikzpicture}    
    \hspace{1cm}
    \begin{tikzpicture}[scale=.8]
        \draw (0,0) circle (2.5 cm); 

        \draw (0,2.5) arc (180:282.8:2);
        \draw (0,2.5) arc (0:-102.8:2);
        \draw (0,-2.5) arc (0:102.8:2);
        \draw (0,-2.5) arc (180:77.2:2); 

        \draw (0,2.5) arc (180:341:.4);
        \draw (0,2.5) arc (0:-161:.4);
        \draw (0,-2.5) arc (0:161:.4);
        \draw (0,-2.5) arc (180:19:.4); 

        \foreach \i in {0,1,2,3} {
        \draw ({2.5*cos (45+90*\i)}, {2.5*sin (45+90*\i)}) arc ({90+45+90*\i}:{296+90*\i}:.4);     
        \draw ({2.5*cos (45+90*\i)}, {2.5*sin (45+90*\i)}) arc ({315+90*\i}:{153+90*\i}:.4);
        }

        \draw (cos 135, sin 135) edge[->, bend right=50, >=stealth] (cos 45, sin 45);
        \draw (cos 225, sin 225) edge[<-, bend left=50, >=stealth] (cos 315, sin 315);

        \node at (0,0.7) {c};
        \node at (0,-0.7) {d};
    \end{tikzpicture}
    \caption{}    \label{fig:poincare_isometries}
    \end{figure} 
    
    Both $G_1$ and $G_2$ act freely and properly discontinuously by isometries on the Poincar\'{e} disc $\mathbb{D}$. The quotient space $\mathbb{D}/G_1$ is a finite-volume hyperbolic surface with three cusps, whereas the quotient space $\mathbb{D}/G_2$ is an infinite-volume hyperbolic surface with two cusps. Both surfaces are homeomorphic to the thrice-punctured sphere $\Sigma_{0,3}$.

    In $G_1$, the isometry $ab$ is parabolic, whereas in $G_2$ the isometry $cd$ is hyperbolic. By Definition~3.3 of \cite{Hruska}, the group $G_1$ is hyperbolic relative to the collection of cyclic subgroups \( \{\langle a\rangle,\; \langle b\rangle,\; \langle ab\rangle\},\) and its Bowditch boundary is homeomorphic to $\partial\mathbb{D}=\mathbb{S}^1$.

    For the group $G_2$, let $U$ denote the closed convex half-plane in $\mathbb{D}$ whose topological boundary is the axis of $cd$. Define
    \[ 
    \mathcal{S}=\{gU \mid g\in G_2\}, \qquad X=\mathbb{D}\setminus \bigsqcup_{V\in\mathcal{S}} \operatorname{int}(V), 
    \]
    where $\operatorname{int}(V)$ denotes the interior of $V$. The space $X$ is a Gromov hyperbolic metric space, and the quotient $X/G_2$ is a surface with one topological boundary component and two cusps. By Definition~3.3 of \cite{Hruska}, the group $G_2$ is hyperbolic relative to $\{\langle c\rangle, \langle d\rangle\}$.
    The Bowditch boundary of $G_2$ is homeomorphic to the Gromov boundary $\partial X$, which is totally disconnected. Therefore, the Bowditch boundaries of $G_1$ and $G_2$ are not homeomorphic. Consequently, by \autoref{thm_main_theorem_1}, there does not exist any coarsely cusp-preserving quasi-isometry between $G_1$ and $G_2$, even though the groups themselves are quasi-isometric. 
   
    \medskip
    \noindent $(ii)$ $(a)$ The free group $\mathbb{F}(u,v)$ is hyperbolic relative to the cyclic subgroup $\langle uvu^{-1}v^{-1}\rangle$, with Bowditch boundary $\partial \mathbb{D}=\mathbb{S}^1$, when $\mathbb{F}(u,v)$ is realized as the fundamental group of a torus with a single puncture at infinity. Consider the automorphism \( T:\mathbb{F}(u,v)\to \mathbb{F}(u,v) \) defined on generators by \( T(u)=uv^{-1}, ~T(v)=v . \) The automorphism $T$ preserves the parabolic subgroup $\langle uvu^{-1}v^{-1}\rangle$, and since $T$ is an isomorphism, it is a quasi-isometry of $\mathbb{F}(u,v)$.
    
    \medskip   
    \noindent $(b)$ Let $\Sigma_{1,2}$ be a twice-punctured torus and $\Sigma_{1,1}$ a once-punctured torus. The fundamental groups $\pi_1(\Sigma_{1,2}) $ and $\pi_1(\Sigma_{1,1})$ are isomorphic to the free groups $\mathbb F(x,y,z)$ and $\mathbb F(u,v)$ , respectively. 
    There exists a two-sheeted covering map $p:\Sigma_{1,2}\to\Sigma_{1,1}$ such that the induced injective homomorphism $p_{*}:\mathbb F(x,y,z)\to \mathbb F(u,v) $ satisfies $p_{*}(x)=u^2, p_{*}(y)=v, p_{*}(z)=uvu$
    and $p_{*}(x^{-1}yz^{-1})=u^{-1}vuv^{-1} $, $ p_{*}(xy^{-1}z)=uv^{-1}u^{-1}v$. The group $\mathbb F(x,y,z)$ is hyperbolic relative to $\{<x^{-1}yz^{-1}>, <xy^{-1}z>\}$. The subgroup $p_{*}(\mathbb F(x,y,z))$ is of index two in $\mathbb F(u,v)$. Hence, the map $p_{*}$ is a quasi-isometry that preserves the cusps. 
\end{ex}

Farb \cite{Farb}*{Lemma 4.4} proved that horospheres in a pinched Hadamard manifold are visually bounded, as described in the statement of \autoref{prop_visually_bounded}. The same result holds for horospheres in a relatively hyperbolic group \cite{Pal}*{Lemma 1.2.39}.

\begin{prop}[Visual boundedness of horospheres] \label{prop_visually_bounded} 
     There exists a constant $K_3 = K_3(\delta) \geq 0$ such that the following holds.    
     For any \(H \in \mathcal{H}_G\) and any geodesic \(\gamma\) in $X^h$ that does not intersect \(H\), let \(T_{\gamma}\) denote the set of points \(p \in H\) for which there exists some \(t \in \mathbb{R}\) such that a geodesic segment \([\gamma(t), p]\) intersects \(H\) only at \(\{p\}\). Let $\mathrm{diam}(T_{\gamma})$ denote the diameter of $T_{\gamma}$. Then 
     \(
     \sup\{\mathrm{diam}(T_{\gamma}) \,:\, \gamma\mbox{ is a geodesic not intersecting }H\},
     \) 
     called the visual size of $H$, is at most $K_3$.   
\end{prop}


\section{Quasi-centres, Cross Ratios and Quasi-M\"{o}bius maps}\label{sec_QC_and_QM}
   
This section outlines Paulin’s contributions concerning quasi-M\"{o}bius maps between Gromov boundaries of hyperbolic metric spaces. In particular, for a geodesic triangle in a hyperbolic metric space, Paulin introduced the concept of a quasi-projection, which remains within a uniformly bounded distance from a quasi-centre of the triangle. While quasi-projections play a key role in his theory, their definition involves certain technical subtleties. In contrast, quasi-centres provide a more geometric and accessible alternative. For this reason, we adopt quasi-centres in place of quasi-projections in our exposition.  

\subsection{Quasi-centres} \label{subsec_quasi_centres} 

Consider a geodesic triangle $\triangle$ in a proper $\delta$-hyperbolic metric space $X$. Each side of $\triangle$ contains a point whose distance to the other two sides is at most $\delta$. If the triangle $\triangle$ is taken to be an ideal triangle (i.e., all its vertices lie in $\partial X$), then it can be shown that there exists a point in $X$ whose distance to each side of $\triangle$ is at most $3\delta$. It is a standard fact that any two bi-infinite geodesics joining the same pair of points in the Gromov boundary $\partial X$ lie in $2\delta$-neighbourhood of each other. Thus, given any three distinct points $a,b,c\in \partial X$, there exists a point $p\in X$ whose distance to each side of any ideal triangle with vertices $a,b,c$ is at most $5\delta$. 

\begin{defn}[Quasi-centre, \cite{Mackay-Sisto}] 
    Let $(X,d_X)$ be a proper $\delta$-hyperbolic metric space and let $\overline{X} = X \cup \partial X$. Given three distinct points $a,b,c \in \overline{X}$, a quasi-centre for $a,b,c$ is a point $\pi_X({a,b,c}) \in X$ that lies within distance $5\delta$ from each geodesic joining $a$ to $b$, $b$ to $c$, or $a$ to $c$.
\end{defn} 





\begin{lem}[\cite{Bow1991}*{Lemma 3.1.5}] \label{lem_bounded_diameter_of_points_on_a_geodesic}
    For each $\delta, r \ge 0$, there exists a constant $C_1 = C_1(\delta,r) \ge 0$ such that the following holds. Let $X$ be a proper $\delta$-hyperbolic metric space. Then for any geodesic triangle $\triangle(a, b, c)$ with vertices $a,b,c \in \overline{X}$, the set of points in $X$ that lie within distance $r$ from each side of $\triangle(a,b,c)$ has diameter at most $C_1$. 
\end{lem}
  
In the following proposition, we state some basic properties about quasi-centres. These follow from standard arguments using \autoref{ex_(3,0)_quasigeodesic}, the stability of quasigeodesics \autoref{Prop_Stability_of_quasigeodesics}, and \autoref{lem_bounded_diameter_of_points_on_a_geodesic}.

\begin{prop}\label{prop_basic_properties_of_quasi_centres}
    Let $X$ be a proper $\delta$-hyperbolic metric space. Then there exists a constant $C_2 = C_2(\delta) \ge 0$ such that for any three distinct points $a,b,c \in \overline{X}$, the following properties hold: 
    \begin{enumerate}[$(i)$]        
        \item The set of all quasi-centres for $a,b,c$ is non-empty and has diameter at most $C_2$.

        \item If $b\in X$, then the distance between a nearest point projection of $b$ onto $[a,c]$ and a quasi-centre for $a,b,c$ is at most $C_2$.
        
        \item If $a,b,c \in X$, then the Gromov product $(a,c)_b$ satisfies ${(a,c)}_b \approx_{C_2} d_X(b, \pi_X(a,b,c)).$   
    \end{enumerate}
\end{prop}

For a proper hyperbolic metric space $X$, consider the set of distinct boundary triples:
\begin{displaymath}
     \partial^3X \coloneqq \{(a,b,c) \mid a,b,c \in \partial X, \ a \ne b, b \ne c, c \ne a\}.
\end{displaymath}

For each triple $(a,b,c) \in \partial^3X$, an associated quasi-centre $\pi_X(a,b,c) \in X$ defines a map $\pi_X: \partial^3X \to X$. By \autoref{prop_basic_properties_of_quasi_centres}, the map $\pi_X$ is uniquely defined up to a bounded discrepancy. Furthermore, we choose the map $\pi_X$ in a way so that it remains invariant under the permutation of the coordinates $a,b,c$. 


\begin{lem}[Quasi-isometries coarsely preserve quasi-centres, \cite{Mackay-Sisto}*{Lemma 3.14}] \label{lem_qi_coarsely_preserves_the_quasi-centres}
    Let $(X,d_X)$ and $(Y,d_Y)$ be two proper $\delta$-hyperbolic metric spaces, and let $\varphi: X \to Y$ be a $(\lambda,\epsilon)$-quasi-isometry, where $\delta,\epsilon \ge 0$ and $\lambda \ge 1$. Then quasi-centres in $X$ are mapped coarsely to quasi-centres in $Y$ under $\varphi$: that is, there exists a constant $C_3 = C_3(\delta,\lambda,\epsilon) \ge 0$ such that for any three distinct points $a,b,c \in \overline{X}$, we have $$d_Y\left(\varphi(\pi_X(a,b,c)), \pi_Y(\overline{\varphi}(a),\overline{\varphi}(b),\overline{\varphi}(c))\right) \le C_3,$$ 
    where $\overline{\varphi}: \overline{X} \to \overline{Y}$ denotes the extension of $\varphi$ to the compactification. 
\end{lem}

The following lemma asserts that $\pi_X$ is coarsely continuous. Its proof relies on standard arguments concerning the convergence of geodesics in proper hyperbolic metric spaces. We omit the proof and refer the reader to Paulin's work, where it has been stated in terms of ‘quasi-projections’.

\begin{lem}[\cite{Pau96}*{Lemma 5.1}] \label{lem_quasi-continuity_of_quasi-centres}
    Given  $\delta \ge 0$, there exists a constant $C_4=C_4(\delta) \ge 0$ such that the following holds. Let $X$ be a proper $\delta$-hyperbolic metric space, and let $(a,b,c) \in \partial^3X$. Suppose $\{(a_i,b_i,c_i)\}$ is a  sequence in $\partial^3X$ converging to $(a,b,c)$. Then, for all sufficiently large $i$, we have 
    $d_X(\pi_X(a,b,c), \pi_X(a_i, b_i, c_i)) \le C_4.$
\end{lem}


\subsection{Cross-Ratios and Quasi-M\"{o}bius Maps} \label{subsec_cross-ratios_and_QM}

\begin{defn}[Cross-ratio, \cite{Pau96}]
    Let $(X, d_X)$ be a proper hyperbolic metric space. Given four distinct points $a,b,c,d \in X$, the cross-ratio $[a,b,c,d]$ is defined by
    $$[a,b,c,d] \coloneqq \frac{1}{2}\left\{d_X(a,d) - d_X(c,d) + d_X(b,c) - d_X(a,b)\right\}.$$
    \begin{figure}[H]
    \centering
    \begin{tikzpicture}[scale=.5]
        \draw (1.9,5)..controls(2.3,2.5)and(2.1,1.8)..(2,1.5);
        \draw (1.9,5)..controls(2.4,2.9)and(2.6,2.6)..(3.5,2.3);
        \draw (0,0)..controls(1,0.5)and(1.7,.8)..(2,1.5);
        \draw (0,0)..controls(1,0.2)and(2.4,0.55)..(3,0.6);
        
        \draw (8.1,-2)..controls(7.7,0.5)and(7.9,1.2)..(8,1.5);
        \draw (8.1,-2)..controls(8,-1)and(7.2,0.2)..(6.5,0.5);
        \draw (10,3)..controls(9,2.5)and(8.3,2.1)..(8,1.5);
        \draw (10,3)..controls(8.5,2.6)and(7.4,2.3)..(7,2.3);
        
        \draw (3.5,2.3)..controls(4,2.2)and(5,2.1)..(7,2.3);
        \draw (6.5,0.5)..controls(6,0.7)and(5,0.9)..(3,0.6);
        
        \filldraw[black] (0,0) circle (1.5pt);
        \filldraw[black] (10,3) circle (1.5pt);
        \filldraw[black] (1.9,5) circle (1.5pt);
        \filldraw[black] (8.1,-2) circle (1.5pt);
      
        \node at (-0.5,0) {$a$};
        \node at (1.5,5) {$b$};
        \node at (10,3.5) {$c$};
        \node at (7.5,-2) {$d$};
        
        \node at (3.3,-0.1) {$+$};
        \node at (8.5,1) {$-$};
        \node at (5.5,3) {$+$};
        \node at (1.4,2.1) {$-$};
    \end{tikzpicture}
    \caption{}
    \end{figure}
    
     This definition extends naturally to the Gromov boundary $\partial X$ by taking the limit inferior over sequences of interior points converging to the boundary. Specifically, for four distinct points $a,b,c,d \in \partial X$, the cross-ratio is defined by 
     $$[a,b,c,d]= \sup \liminf_{i\to\infty}\ [a_i,b_i,c_i,d_i],$$ 
     where the supremum is taken over all sequences $\{a_i\},\{b_i\},\{c_i\}$ and $\{d_i\}$ in \(X\) converging to $a$, $b$, $c$ and $d$, respectively.
\end{defn}

For distinct $a,b,c,d \in X$, we have $[a,b,c,d] = (c,d)_a-(b,c)_a$. By \autoref{prop_basic_properties_of_quasi_centres}, we have ${(c,d)}_a \approx_{C_2} d_X(a, \pi_X(a,c,d))$
and ${(b,c)}_a \approx_{C_2} d_X(a, \pi_X(a,b,c))$. Therefore, the quantities $|[a,b,c,d]|$ and $d_X(\pi_X(a,b,c), \pi_X(a,c,d) )$ are coarsely the same. The same conclusion holds when $a,b,c,d \in \partial X$. We formalize this in the following lemma, which was also proved by Paulin in \cite{Pau96}, where quasi-projections are used instead of quasi-centres..

\begin{lem}[\cite{Pau96}*{Lemma 4.2}] \label{lem_cross_ratios_vs_quasi_centres}
    Given $\delta \geq 0$, there exists a constant $C_5 = C_5(\delta) \ge 0$ such that the following holds. Let $X$ be a proper $\delta$-hyperbolic metric space, and let $a, b, c, d \in \partial X$ be four distinct points. Then,
    $$|[a,b,c,d]| \approx_{C_5} d_X(\pi_X(a,b,c),\pi_X(a,c,d) ).$$    
\end{lem}

The value of the cross ratio $[a,b,c,d]$ changes by changing the ordering of $a,b,c,d$. The following lemma provides an essential control over cross-ratios in a hyperbolic metric space when the ordering is changed.

\begin{lem}[\cite{CCM2019}*{Lemma 4.2}]\label{lem_one_of_three_cross-ratio_is_bdd}
  Given $\delta \geq 0$, there exists a constant $C_6 = C_6(\delta) \ge 0$ such that the following holds. Let $X$ be a proper $\delta$-hyperbolic metric space. Then for any four distinct points $a,b,c,d \in \partial X$, at least one of the three cross-ratios $[a,b,c,d]$, $[a,c,b,d]$, or $[c,a,b,d]$ has absolute value at most $C_6$.
\end{lem}


\begin{defn}[Quasi-M\"{o}bius map]\label{defn_QM}
    Let $(X,d_X)$ and $(Y,d_Y)$ be proper hyperbolic metric spaces, and let $\psi:[0,\infty)\to[0,\infty)$ be a continuous function. A bijection $f:\partial X\to\partial Y$ is called $\psi$-quasi-M\"{o}bius if for distinct $a,b,c,d \in \partial X$, 
    \begin{displaymath}
        |[f(a),f(b),f(c),f(d)]| \le \psi(|[a,b,c,d]|),
    \end{displaymath}
    and for distinct $a',b',c',d' \in \partial Y$,
    \begin{displaymath}
        |[f^{-1}(a'),f^{-1}(b'),f^{-1}(c'),f^{-1}(d')]|\le \psi(|[a',b',c',d']|).
    \end{displaymath}
    
    A bijection $f:\partial X \to \partial Y$ is called quasi-M\"{o}bius if it is $\psi$-quasi-M\"{o}bius for some continuous function $\psi:[0,\infty)\to[0,\infty)$. In this case, the function $\psi$ is called a distortion function associated with $f$.
\end{defn}

\begin{prop}[\cite{Pau96}*{Lemma 4.4}] \label{prop_QM_is_homeo}
    Let $X$ and $Y$ be two proper hyperbolic metric spaces. Then any quasi-M\"{o}bius map $f:\partial X\to\partial Y$ is a homeomorphism.
\end{prop} 

We next state Paulin's result that a quasi-isometry between two proper hyperbolic metric spaces induces a 
quasi-M\"{o}bius map between their Gromov boundaries.

\begin{thm}[\cite{Pau96}*{Proposition 4.5}]\label{thm_qi_implies_qm}
    Given constants $\delta, \epsilon \ge 0$ and $\lambda \ge 1$, there exist constants $A_1 = A_1(\delta, \epsilon, \lambda) \ge 1$ and $B_1 = B_1(\delta, \epsilon, \lambda) \ge 0$ such that the following holds. 
    Let $(X,d_X)$ and $(Y,d_Y)$ be proper $\delta$-hyperbolic metric spaces, and let $\varphi: X \to Y$ be a $(\lambda,\epsilon)$-quasi-isometry. Then the induced boundary map $\partial \varphi: \partial X \to \partial Y$ is a $\psi$-quasi-M\"{o}bius homeomorphism with distortion function
    $\psi(t) = A_1t+B_1,$ for all $t \ge 0$.
\end{thm}


\section{Quasi-Isometries Between Relatively Hyperbolic Groups and Relative Quasi-M\"{o}bius Homeomorphisms Between Their Bowditch Boundaries} \label{sec_QM_and_QI}


\subsection{Relative Quasi-M\"{o}bius Maps} \label{subsec_rel_QM}
Let $(G,\mathcal{H}_G)$ be a relatively hyperbolic group with a Cayley graph $X$. Let $X^h$ and $\partial X^h$ denote the cusped space and the Bowditch boundary, respectively. Let $a, b, c_H \in \partial X^h$ be three distinct points, where $c_H$ is a parabolic endpoint associated with a horosphere $H \in \mathcal{H}_G$. We define below the relative cross-ratio $[a,b,c_H]$, whose absolute value coarsely measures the distance between a quasi-centre of $(a,b,c_H)$ and $H$.

\begin{defn} (Relative cross-ratio) 
    The relative cross-ratio of $a,b,c_H$ is defined by 
    $$[a,b,c_H] \coloneqq \sup \liminf_{i\to\infty} \frac{1}{2} \left\{d_{X^h}(a_i,H) + d_{X^h}(b_i,H) - d_{X^h}(a_i,b_i)\right\},$$ 
    where the supremum is taken over all sequences $\{a_i\}$ and $\{b_i\}$ in $X^h$ with $a_i \to a$ and $b_i \to b$, and $d_{X^h}(x,H)$ denotes the distance from a point $x \in X^h$ to $H$. 
\end{defn}

\begin{prop}\label{prop_relative_cross_ratio_vs_quasi_centres}
    Given $\delta \ge 0$, there exists a constant $C_7 = C_7(\delta) \ge 0$ such that the following holds. Let $a,b,c_H \in \partial X^h$ be three distinct points, where $c_H$ is a parabolic endpoint corresponding to some $H \in \mathcal{H}_G$. Then the relative cross-ratio satisfies the coarse equivalence $$|[a,b,c_H]| \approx_{C_7} d_{X^h}(\pi_{X^h}(a,b,c_H),H).$$
\end{prop} 

\begin{proof} 
    Let $\{a_i\}$ and $\{b_i\}$ be sequences in $X^h$ converging to $a$ and $b$, respectively. Since $a$, $b$, and $c_H$ are distinct, for all sufficiently large $i$ the points $a_i$ and $b_i$ lie outside the horoball $H^h$. By discarding finitely many terms, we may assume this holds for all $i$.    
    For each $i$, let $z_i, w_i \in H$ be such that
    \[
    d_{X^h}(a_i,z_i) = d_{X^h}(a_i,H)\quad \text{and} \quad d_{X^h}(b_i,w_i) = d_{X^h}(b_i,H).
    \] 
    Consider the geodesic segment $[a_i,b_i]$ in $X^h$. There are two cases: either $[a_i,b_i]$ lies entirely outside $H^h$, or it passes through $H^h$.

    \medskip\noindent\textbf{Case~1.} Suppose that the geodesic $[a_i,b_i]$ lies entirely outside the horoball $H^h$. 

    By \autoref{prop_visually_bounded}, there exists a constant $K_3=K_3(\delta) \geq 0$ such that $d_{X^h}(z_i,w_i)\le K_3$. Therefore, by \autoref{lem_cross_ratios_vs_quasi_centres}, we obtain
    \begin{align*}
        \frac{1}{2} |d_{X^h}(a_i,H)+d_{X^h}(b_i,H)-d_{X^h}(a_i,b_i)| 
        & = \frac{1}{2} |d_{X^h}(a_i,b_i)-d_{X^h}(a_i,z_i)-d_{X^h}(b_i,w_i)| \\
        & \approx_{K_3} |[a_i,z_i,w_i,b_i]| \\
        & \approx_{K_3 + C_5} d_{X^h}(\pi_{X^h}(a_i,z_i,w_i), \pi_{X^h}(a_i,w_i,b_i)),
    \end{align*}
    
    \begin{figure}[H]
        \centering
        \begin{tikzpicture}[scale=.8]
            \shade[top color=white,bottom color=gray!20] (-5,2.5) rectangle (5,5);
            \draw (-5,2.5) -- (5,2.5);
            \draw (2,0) arc (20:160:2cm); 
            \draw[dashed] (2,2.5) arc (70:175:2.79cm);

            \draw (2,0) -- (2,4);
            \draw (-1.75,0) -- (-1.75,4);
            \draw[dashed] (2,4)--(2,5);
            \draw[dashed] (-1.75,4)--(-1.75,5); 

            \filldraw (2,2.5) circle (1pt);
            \filldraw (2,0) circle (1pt);
            \filldraw (-1.75,0) circle (1pt);
            \filldraw (-1.75,2.5) circle (1pt);

            \node at (-2.05,0) {$a_i$};
            \node at (2.3,0) {$b_i$};
            \node at (-2.05,2.2) {$z_i$};
            \node at (2.3,2.2) {$w_i$}; 
            \node at (3,4) {$H^h$};
        \end{tikzpicture} 
        \caption{}
    \end{figure} 

    By definition, the quasi-centre $\pi_{X^h}(a_i,w_i,b_i)$ lies within distance $5\delta$ of each side of the triangle $\triangle(a_i,w_i,b_i)$. Since $d_{X^h}(z_i,w_i)\le K_3$, the $\delta$-slimness of $\triangle(a_i,z_i,w_i)$ implies that $\pi_{X^h}(a_i,w_i,b_i)$ lies within distance $K_3+6\delta$ of each side of $\triangle(a_i,b_i,c_H)$. Hence, by \autoref{lem_bounded_diameter_of_points_on_a_geodesic}, there exists a constant $C_1=C_1(\delta,K_3+6\delta)\ge0$ such that 
    \[ 
    d_{X^h}\!\left(\pi_{X^h}(a_i,w_i,b_i),\,\pi_{X^h}(a_i,b_i,c_H)\right)\le C_1. 
    \]  

    By the definition of quasi-centres, there exists a point $p_i\in[a_i,b_i]$ such that 
    \[
    d_{X^h}\!\left(\pi_{X^h}(a_i,b_i,c_H),p_i\right)\le5\delta.
    \] 
    Let $p_i'\in H$ be a nearest-point projection of $p_i$ onto $H$. By \autoref{prop_visually_bounded}, we have $d_{X^h}(p_i',z_i)\le K_3$. Moreover, since $d_{X^h}(z_i,w_i)\le K_3$ and quasi-centres lie within $5\delta$ of each side of the corresponding triangle, we obtain
    \[
    d_{X^h}\!\left(\pi_{X^h}(a_i,z_i,w_i),z_i\right)\le K_3+5\delta.
    \]

    Combining the above estimates, we get 
    \begin{align*}
        d_{X^h}(\pi_{X^h}(a_i,w_i,b_i), \pi_{X^h}(a_i,z_i,w_i)) 
        & \le d_{X^h}(p_i,p_i') +  2K_3 + C_1 + 10\delta\\
        & \le d_{X^h}(\pi_{X^h}(a_i,b_i,c_H), H) + 2 K_3 + C_1 + 15\delta.
    \end{align*} 
    
    On the other hand, 
    \begin{align*}
        d_{X^h}(\pi_{X^h}(a_i,b_i,c_H), H) 
        & \le d_{X^h}(\pi_{X^h}(a_i,b_i,c_H), \pi_{X^h}(a_i,z_i,w_i)) + d_{X^h}(\pi_{X^h}(a_i,z_i,w_i), H) \\
        & \le (d_{X^h}(\pi_{X^h}(a_i,w_i,b_i), \pi_{X^h}(a_i,z_i,w_i)) + C_1) + (K_3 + 5\delta).
    \end{align*}

    Setting $C_7' = 3K_3 + C_1 + C_5 + 15\delta$, we conclude that
    $$\frac{1}{2} | d_{X^h}(a_i,H)+d_{X^h}(b_i,H)-d_{X^h}(a_i,b_i) | \approx_{C_7'} d_{X^h}(\pi_{X^h}(a_i,b_i,c_H), H).$$

    \medskip\noindent \textbf{Case~2.} Suppose the geodesic $[a_i,b_i]$ passes through the horoball $H^h$. 
    
    \begin{figure}[H] 
        \centering
        \begin{tikzpicture}[scale=.8]
            \shade[top color=white,bottom color=gray!20] (-5,1.5) rectangle (5,5);
            \draw (-5,1.5) -- (5,1.5);
            \draw (4,0) arc (20:160:4cm); 

            \draw (4,0) -- (4,1.5);
            \draw (3.04,1.5) -- (3.04,4);
            \draw[dashed] (3.04,4)--(3.04,5);
            
            \draw (-3.52,0) -- (-3.52,1.5); 
            \draw (-2.55,1.5) -- (-2.55,4);
            \draw[dashed] (-2.55,4)--(-2.55,5);

            \filldraw (-3.52,0) circle (1pt);
            \filldraw (4,0) circle (1pt);
            
            \filldraw (-3.52,1.5) circle (1pt);
            \filldraw (-2.55,1.5) circle (1pt);
            \filldraw (4,1.5) circle (1pt);            
            \filldraw (3.04,1.5) circle (1pt);
            
            \filldraw (0,2.62) circle (1pt);
            \filldraw (-2.55,3.8) circle (1pt);
            \filldraw (3.04,3.8) circle (1pt);

            \draw[bend left,dashed] (-2.55,3.8) edge (0,2.62);
            \draw[bend right,dashed] (3.04,3.8) edge (0,2.62);

            \node at (-3.82,0) {$a_i$};
            \node at (4.3,0) {$b_i$};
            \node at (-3.82,1.2) {$z_i$};
            \node at (4.3,1.2) {$w_i$}; 
            \node at (-2.4,1.2) {$u_i$};
            \node at (2.9,1.2) {$v_i$};
            \node at (0,2.32) {$p_1$};
            \node at (-2.85,3.8) {$p_2$};
            \node at (3.34,3.8) {$p_3$};
        \end{tikzpicture} 
        \caption{}
    \end{figure} 

    Let $u_i,v_i\in H$ be the first entry point and last exit point, respectively, of $[a_i,b_i]$ into $H^h$. By \autoref{prop_visually_bounded}, there exists $K_3=K_3(\delta)\ge0$ such that 
    \[
    d_{X^h}(z_i,u_i) \le K_3 \quad \text{and} \quad d_{X^h}(v_i,w_i) \le K_3.
    \] 
    Therefore, 
    \begin{align*}
        |d_{X^h}(a_i,H)+d_{X^h}(b_i,H)-d_{X^h}(a_i,b_i)| & = |d_{X^h}(a_i,z_i) + d_{X^h}(b_i,w_i) - d_{X^h}(a_i,b_i)| \\
        & \approx_{2K_3} d_{X^h}(u_i,v_i).
    \end{align*} 

    Now consider the geodesic triangle $\triangle(u_i,v_i,c_H)$. There exist points $p_1 \in [u_i,v_i]$, $p_2 \in [u_i,c_H]$, and $p_3 \in [v_i,c_H]$ such that for each $j=1,2,3$, we have $$d_{X^h}(\pi_{X^h}(u_i,v_i,c_H),p_j) \le 5\delta.$$ 
    Then, by the triangle inequality, we obtain 
    \begin{align*}
        d_{X^h}(u_i,v_i) = d_{X^h}(u_i,p_1) + d_{X^h}(p_1,v_i) & \approx_{2K_3 + 20\delta} d_{X^h}(p_2,u_i) + d_{X^h}(p_3,v_i) \\ 
        & \approx_{2K_3 + 30\delta} 2~d_{X^h}(\pi_{X^h}(u_i,v_i,c_H),H).
    \end{align*}
    Also, $\pi_{X^h}(u_i,v_i,c_H)$ is uniformly close to $\pi_{X^h}(a_i,b_i,c_H)$. In particular, there exists a constant $r = r(\delta, K_3) \ge 0$ such that \(d_{X^h}(\pi_{X^h}(u_i,v_i,c_H), \pi_{X^h}(a_i,b_i,c_H)) \le r.\)
    Setting $C_7'' = K_3 + 15\delta+r$, we conclude that 
    $$\frac{1}{2} | d_{X^h}(a_i,H)+d_{X^h}(b_i,H)-d_{X^h}(a_i,b_i) | \approx_{C_7''} d_{X^h}(\pi_{X^h}(a_i,b_i,c_H),H).$$ 

    Finally, by the quasi-continuity of quasi-centres (\autoref{lem_quasi-continuity_of_quasi-centres}), there exists a constant $C_4 = C_4(\delta) \ge 0$ such that for all sufficiently large $i$, we have $$d_{X^h}(\pi_{X^h}(a_i,b_i,c_H),\pi_{X^h}(a,b,c_H)) \le C_4.$$
    Combining both cases, and setting $C_7=\max\{C_7' , C_7''\} + C_4$, we conclude that 
    $$\frac{1}{2} | d_{X^h}(a_i,H)+d_{X^h}(b_i,H)-d_{X^h}(a_i,b_i) | \approx_{C_7} d_{X^h}(\pi_{X^h}(a,b,c_H),H),$$
    for all sufficiently large $i$. Taking the supremum over all such sequences yields the required estimate:
    $$|[a,b,c_H]| \approx_{C_7} d_{X^h}(\pi_{X^h}(a,b,c_H),H).$$
\end{proof} 

We now define relative quasi-M\"{o}bius maps between Bowditch boundaries of relatively hyperbolic groups, which coarsely preserve both classical and relative cross-ratios.

\begin{defn}[Relative quasi-M\"{o}bius map] \label{defn_relative_quasi-Mobius}
    Let $(G_1,\mathcal{H}_{G_1})$ and $(G_2, \mathcal{H}_{G_2})$ be two $\delta$-relatively hyperbolic groups with respective Cayley graphs $X$ and $Y$, where $\delta \ge 0$. Let $\psi:[0,\infty)\to [0,\infty)$ be a continuous map. A bijection $f : \partial X^h \to \partial Y^h$ between Bowditch boundaries is called relative $\psi$-quasi-M\"{o}bius if the following hold:
    \begin{enumerate}[$(i)$]
        \item $f$ and $f^{-1}$ preserve parabolic endpoints,

        \item $f$ is $\psi$-quasi-M\"{o}bius (i.e., it coarsely distorts cross-ratios),

        \item for distinct $a,b,c \in \partial X^h$, where $c$ is a parabolic endpoint,
        $$|[f(a),f(b),f(c)]| \le \psi(|[a,b,c]|),$$
        and for distinct $a',b',c' \in \partial Y^h$ with $c'$ a parabolic endpoint,
        $$|[f^{-1}(a'),f^{-1}(b'),f^{-1}(c')]| \le \psi(|[a',b',c']|).$$
    \end{enumerate}
    
    A bijection $f : \partial X^h \to \partial Y^h$ is called relative quasi-M\"{o}bius if it is relative $\psi$-quasi-M\"{o}bius for some continuous distortion function $\psi: [0,\infty) \to [0,\infty)$.
\end{defn} 

Every relative quasi-M\"{o}bius map between Bowditch boundaries is, by definition, a quasi-M\"{o}bius map, and hence is a homeomorphism.


\subsection{Relative Quasi-M\"{o}bius Maps Induced by Quasi-isometries} \label{subsec_proof_of_thm_1}

The following theorem shows that the homeomorphism between Bowditch boundaries induced by a coarsely cusp-preserving quasi-isometry between relatively hyperbolic groups coarsely preserves relative cross-ratios.

\begin{thm}\label{thm_cusp_preserving_qi_imples_rel_QM}
    Given constants $\delta, \epsilon, K \ge 0$ and $\lambda \ge 1$, there exist constants $A_2 = A_2(\delta, \epsilon, \lambda, K) \ge 1$ and $B_2 = B_2(\delta, \epsilon, \lambda, K) \ge 0$ such that the following holds. Let $(G_1, \mathcal{H}_{G_1})$ and $(G_2, \mathcal{H}_{G_2})$ be two $\delta$-relatively hyperbolic groups with Cayley graphs $X$ and $Y$, respectively. Suppose $\varphi : G_1 \to G_2$ is a $(\lambda,\epsilon)$-quasi-isometry that is also $K$-coarsely cusp-preserving. Let $\varphi^h: X^h \to Y^h$ be an induced quasi-isometry between the corresponding cusped spaces, and let $\partial \varphi^h: \partial X^h \to \partial Y^h$ be the induced homeomorphism on the Bowditch boundaries.
    Then, for any three distinct points $a,b,c \in \partial X^h$, where $c$ is a parabolic endpoint, the relative cross-ratio satisfies
    $$|[\partial \varphi^h (a), \partial \varphi^h (b), \partial \varphi^h (c)]| \approxeq_{A_2,B_2} |[a,b,c]|.$$ 
\end{thm}

\begin{proof}
    Let $a' = \partial \varphi^h(a)$, $b' = \partial \varphi^h(b)$, and $c' = \partial \varphi^h(c)$. Let $H \in \mathcal{H}_{G_1}$ and $H' \in \mathcal{H}_{G_2}$ be the horospheres associated with the parabolic endpoints $c$ and $c'$, respectively. By \autoref{prop_relative_cross_ratio_vs_quasi_centres}, there exists a constant $C_7 = C_7(\delta) \ge 0$ such that
    $$|[a,b,c]| \approx_{C_7} d_{X^h}(\pi_{X^h}(a,b,c), H) \quad \text{and} \quad |[a',b',c']| \approx_{C_7} d_{Y^h}(\pi_{Y^h}(a',b',c'), H').$$ 
    
    Let $h \in H$ be a point such that $d_{X^h}(\pi_{X^h}(a,b,c), h) = d_{X^h}(\pi_{X^h}(a,b,c), H)$. Since $\varphi$ preserves cusps $K$-coarsely, 
    $$d_{Y^h}(\varphi(h), H') \le K.$$
    Moreover, by \autoref{lem_qi_coarsely_preserves_the_quasi-centres}, there exists a constant $C_3 = C_3(\delta, \lambda, \epsilon) \ge 0$ such that $$d_{Y^h}(\varphi(\pi_{X^h}(a,b,c)), \pi_{Y^h}(a',b',c')) \le C_3.$$ 

    Since $\varphi^h: X^h \to Y^h$ is an induced quasi-isometry induced by $\varphi$, there exists $\lambda' = \lambda'(\lambda,\epsilon, \delta) \ge 1$ and $\epsilon' = \epsilon'(\lambda,\epsilon,\delta) \ge 0$ such that $\varphi^h$ is $(\lambda',\epsilon')$-quasi-isometry. 
    Therefore,
    \begin{align*}
        |[a',b',c']| & \le  d_{Y^h}(\pi_{Y^h}(a',b',c'), H') + C_7 \\ 
        & \le d_{Y^h}(\varphi^h(\pi_{X^h}(a,b,c)), H') + C_3 + C_7 \\
        & \le  d_{Y^h}(\varphi^h(\pi_{X^h}(a,b,c)), \varphi^h(h)) + d_{Y^h}(\varphi^h(h), H') + C_3 + C_7\\
        & \le ( \lambda' d_{X^h}(\pi_{X^h}(a,b,c), h) + \epsilon') + K + C_3 + C_7\\
        & \le \lambda' |[a,b,c]| + \lambda' C_7 + \epsilon' + K + C_3 + C_7.
    \end{align*}

    Hence, setting $A_2' = \lambda'$ and $B_2' = \lambda' C_7 + \epsilon' + K + C_3 + C_7$, we conclude that 
    $$|[a',b',c']| \precapprox_{A_2', B_2'} |[a,b,c]|.$$
    An analogous argument using an inverse quasi-isometry of $\varphi^h$, we get constants $A_2''=A_2''(\delta,\lambda,\epsilon, K) \ge 1$ and $B_2''= B_2''(\delta, \lambda, \epsilon, K) \ge 0$ such that 
    $$|[a,b,c]| \precapprox_{A_2'', B_2''} |[a',b',c']|.$$
    
    Finally, setting $A_2 = max\{A_2',A_2''\}$ and $B_2= \max\{B_2',B_2''\}$, we obtain 
    $$|[a',b',c']| \approxeq_{A_2,B_2} |[a,b,c]|,$$
    as required.
\end{proof}

From \autoref{prop_QM_is_homeo}, it follows that the map $\varphi^h$ is a quasi-M\"{o}bius homeomorphism. Hence, \autoref{thm_main_theorem_1} is an immediate consequence of \autoref{prop_QM_is_homeo} and \autoref{thm_cusp_preserving_qi_imples_rel_QM}.

\subsection{Coarse Denseness of Quasi-centres}\label{subsec_coarse_denseness_of_quasi_centres}

\begin{prop} \label{prop_a_point_on_geodesic_vs_QP}
    Given $\delta \ge 0$, there exists a constant $R_1 = R_1(\delta) \ge 0$ such that the following holds. Let $(G,\mathcal{H}_G)$ be a $\delta$-relatively hyperbolic group with a Cayley graph $X$. Suppose $x$ is a point on a geodesic $[a,b]$, where $a,b \in \partial X^h$. Then, there exists $c \in \partial X^h$, distinct from both $a$ and $b$ such that $d_{X^h}(x,\pi_{X^h}(a,b,c)) \le R_1.$
\end{prop} 

Before proving \autoref{prop_a_point_on_geodesic_vs_QP}, we first establish the following auxiliary result, \autoref{lem_a_point_and_a_ray}, which will play a crucial role in the proof of the proposition.

\begin{lem}\label{lem_a_point_and_a_ray}
Let $x$ be a point on a geodesic $[a,b]$, where $a,b \in \partial X^h$. Then there exists a constant $M = M(\delta) \ge 0$ and a geodesic ray $\alpha \colon [0,\infty) \to X^h$ such that:
\begin{itemize}
    \item $d_{X^h}(\alpha(0), x) \le M$;
    \item $\alpha$ lies entirely outside the $Q$-neighbourhood of $[a,b]$, where \(Q = \max\{P, 2\delta\} + 1,\)  and $P$ is the constant obtained from \autoref{lem_Bounded_projection_lemma};
    \item the boundary point $[\alpha] \in \partial X^h$ corresponding to $\alpha$ is distinct from both $a$ and $b$.
\end{itemize}
\end{lem}

\begin{proof}        
    Without loss of generality, we may assume that $x$ is a vertex of $X^h$. Since the cosets of parabolic subgroups cover the entire group $G$, there exists $H \in \mathcal{H}_G$ such that $x \in H^h$. We divide the proof into two cases, depending on whether either $a$ or $b$ is the limit point of $H$ in $\partial X^h$. 

    \medskip\noindent \textbf{Case~1.} Suppose that one of $a$ or $b$ is the limit point of $H$.

    Without loss of generality, assume that $a$ is the limit point of $H^h$. Since $x \in H^h$, there exist a point $x_0 \in H \cap [a,b]$ and an integer $k \ge 0$ such that $x = (x_0,k)$. Choose a point $y_0 \in H$ satisfying
    \(d_H(x_0,y_0)=2^k.\) Since $H \in \mathcal{H}_G$, it is a coset of some parabolic subgroup of $G$ (see \autoref{defn_rel_hyp}); that is, there exists an element $g \in G$ and a parabolic subgroup $H_i \le G$ such that $H = g H_i$. Then, the action of the subgroup $g H_i g^{-1}$ on $G$ stabilizes the coset $H=gH_i$ and fixes the parabolic endpoint $a$. If $x_0 = gh_1$ and $y_0 = gh_2$ for some $h_1,h_2 \in H_i$, then the element $g(h_2 h_1^{-1})g^{-1} \in gH_ig^{-1}$ maps the geodesic $[a,b]$ to a geodesic $[a,c]$ passing through $y_0$, where $c = g(h_2 h_1^{-1})g^{-1}(b) \in \partial X^h$. 

    For each integer $n \ge 0$, define $x_n=(x_0,n)\in [a,b]$ and $y_n=(y_0,n)\in [a,c]$. Since $d_H(x_0,y_0)=2^k$, it follows that \(d_{X^h}(x_k,y_k)=1\).

    \begin{figure}[H]
        \centering
        \begin{tikzpicture}[scale = .6]
            \draw[gray] (-4,0) -- (4,0);
            \draw[gray] (-2,0) -- (-1.1,4) -- (1.1,4);
            \draw[gray,dashed,bend right=5] (-1.1,4) edge (-1,5);
            \draw[gray] (1.34,3)--(1.1,4);
            \draw[gray,dashed,bend left=5] (1.1,4) edge (1,5);
            \draw[gray,dashed,bend left] (-2,0) edge (-3,-1);

            \draw[gray!20, line width = 1.5mm] (1.34,3) -- (2,0);
            \draw[thick] (1.34,3) -- (2,0);
            \draw[gray!20, line width = 1.5mm, bend right] (2,0) edge (3,-1);
            \draw[thick, dashed, bend right] (2,0) edge (3,-1);

            \filldraw (1.34,3) circle (1pt);
            \filldraw (-1.1,4) circle (1pt);

            \node at (-2.4,0.3) {$x_0$};
            \node at (2.4,0.3) {$y_0$};
            \node at (-2.2,4) {$x=x_k$};
            \node at (1.6,4) {$y_k$};
            \node at (1.85,3) {$y_j$};            

            \node at (0,5.5) {$a$};
            \node at (-3.4,-1) {$b$};
            \node at (3.4,-1) {$c$};
            \node at (0,4.3) {$1$};
            \node at (0,0.4) {$2^k$};
            \node at (-4.4,0) {$H$};
            \node[white] at (4.4,0) {$H$};
        \end{tikzpicture}
        \caption{}
    \end{figure} 

    By \autoref{lem_preferred_path_vs_geodesics}, the preferred path $\gamma_{x_0,y_0} = [x_0,x_k] \cup [x_k,y_k] \cup [y_k,y_0]$ is at Hausdorff distance at most $5$ from any geodesic $[x_0,y_0]$ joining $x_0$ and $y_0$. The path $\gamma_{x_0,y_0}$ is a concatenation of two geodesics $\gamma_{x_0,x_k} = [x_0,x_k]$ and $\gamma_{x_k,y_0} = [x_k,y_k] \cup [y_k,y_0]$. It can be shown that the length($\gamma_{x_0,y_0}$) $\approx_{20}$ length($[x_0,y_0]$). Thus, without loss of generality, we may assume that $\gamma_{x_0,y_0}$ is a geodesic. 

    Recall from the Bounded Projection Lemma (\autoref{lem_Bounded_projection_lemma}) that there exist constants \(P=P(\delta)\ge 0\) and \(K_1=K_1(\delta)\ge 0\) such that if a geodesic lies entirely outside the $P$-neighbourhood of a geodesic $\gamma$, then its image under a nearest-point projection onto $\gamma$ has diameter at most $K_1$. We have taken \(Q=\max\{P,2\delta\}+1\). 

    \medskip
    \noindent {\textbf{Sub-case~1.1.}} Suppose $k \ge Q + K_1 +1$.
    
    Let $0\le j \le k$ be the largest integer such that $k-j \ge Q + K_1 + 1$, and consider the geodesic ray $[y_j,c] \subseteq [a,c]$. We claim that $[y_j,c]$ lies entirely outside the $Q$-neighbourhood of $[a,b]$. 

    Let $y_n \in [y_0,y_j]$ be an arbitrary vertex. For each $i \ge k$, the preferred path $\gamma_{y_n,x_i} \coloneqq [y_n,y_i] \cup [y_i,x_i]$ is a geodesic, and hence $d_{X^h}(y_n,x_i) = (i-n) + 1 \ge k-j+1 \ge Q$. Similarly, for each $0 \le i \le k$, since the preferred path $\gamma_{y_0,x_0}$ is a geodesic, we have $d_{X^h}(y_n,x_i) = (k-n) + 1 + (k-i) \ge k-j+1 \ge Q$. Thus, $[y_j,y_0]$ lies outside the $Q$-neighbourhood of $[x_0,a]$. If $d_{X^h}(y_n,z) \le Q$ for some $z \in [x_0,b]$, then $d_{X^h}(z,x_0) \le Q$. Hence, $2Q \ge d_{X^h}(y_n,x_0) = (k-n) + 1 + k \ge 2(Q + K_1 +1) +1$, which is a contradiction. Therefore, $[y_j,y_0]$ lies outside the $Q$-neighbourhood of $[x_0,b]$. Combining the above, we conclude that $[y_j,y_0]$ lies outside the $Q$-neighbourhood of $[a,b]$. 

    Now suppose that $[y_0,c]$ enters the $Q$-neighbourhood of $[a,b]$, and let $z \in [y_0,c]$ be the first entry point. Let $z' \in [a,b]$ satisfy $d_{X^h}(z,z') \le Q$. Then the subsegment $[y_j,z]$ lies outside the $Q$-neighbourhood of $[a,b]$, and hence, by \autoref{lem_Bounded_projection_lemma}, \(d_{X^h}(x_k,z') \le K_1,\) which implies \(d_{X^h}(x_k,z) \le K_1 + Q\). This contradicts the fact that \(d_{X^h}(x_k,z) \ge d_{X^h}(x_k,H) \ge K_1 + Q + 1\). Hence, the geodesic ray $[y_0,c]$, and therefore $[y_j,c]$, lies entirely outside the $Q$-neighbourhood of $[a,b]$. 

    Since $Q > 2\delta$, we have $c \neq a,b$. Therefore, setting $\alpha = [y_j,c]$ and $M' = K_1 + Q + 2$ completes the sub-case.

    \medskip 
    \noindent \textbf{Sub-case~1.2.} Suppose $k < Q + K_1 + 1$. 
    
    Choose a point $x' \in [x_0,a]$ such that \(d_{X^h}(x_0,x') = Q + K_1 + 1\). Applying Sub-case~1.1 to the point $x'$, we obtain a geodesic ray $\alpha: [0,\infty) \to X^h$ that lies entirely outside the $Q$-neighbourhood of $[a,b]$ and satisfies \(d_{X^h}(\alpha(0),x') \le M'\). Hence, $d_{X^h}(\alpha(0),x) \le M' + K_1 + Q +1 = 2K_1 + 2Q + 3$. This completes the sub-case.

    \medskip\noindent \textbf{Case~2.} Suppose that neither $a$ nor $b$ is the limit point of $H$.

    Let $\partial H^h=\{c_H\}$. Let $u_0,v_0\in H$ be the entry and exit points, respectively, of the geodesic $[a,b]$ into the horoball $H^h$. For each integer $n\ge 0$, define $u_n= (u_0,n) \in [u_0,c_H]$ and $v_n= (v_0,n) \in [v_0,c_H]$. Suppose that $2^{l-1} < d_{H}(u_0,v_0) \le 2^l$, for some integer $l \ge 0$. By \autoref{lem_preferred_path_vs_geodesics}, we may assume that the preferred path $\gamma_{u_0,v_0} = [u_0,u_l] \cup [u_l,v_l] \cup [v_l,v_0]$ is a geodesic joining $u_0$ and $v_0$ in $X^h$. Suppose $x = u_k$ for some $0 \le k \le l$.
    
    \begin{figure}[H]
        \centering
        \begin{tikzpicture}[scale=.7]
            \draw[gray] (-4,0) -- (4,0);
            \draw[gray] (-3,0) -- (-2,4) -- (2,4) -- (3,0);
            \draw[gray, bend left,dashed] (-3,0) edge (-4,-1);
            \draw[gray, bend right,dashed] (3,0) edge (4,-1);
            \draw[gray, bend right=5,dashed] (-2,4) edge (-1.8,5);
            \draw[gray, bend left=5,dashed] (2,4) edge (1.8,5);
            \draw[gray] (-2.25,3) -- (0,3)-- (0,2);
            \draw[gray!20, line width = 1.5mm] (0,2) edge (0,-1);
            \draw[thick] (0,2) edge (0,0);
            \draw[thick, dashed] (0,0)-- (0,-1);
            \draw[gray,dashed] (0,3)-- (2.25,3);

            \filldraw (-2.25,3) circle (1pt);
            \filldraw (0,2) circle (1pt); 

            \node at (-3.4,0.3) {$u_0$};
            \node at (3.4,0.3) {$v_0$};
            \node at (0.4,0.3) {$y_0$};
            \node at (0.4,2) {$y_j$};
            \node at (0.4,3.3) {$y_k$};
            \node at (-3.3, 3) {$x = u_k$};
            \node at (2.7,3) {$v_k$};
            \node at (-2.4,4) {$u_l$};
            \node at (2.4,4) {$v_l$};

            \node at (-4.4,-1.1) {$a$};
            \node at (4.4,-1.1) {$b$};
            \node at (0,-1.3) {$d$};
            \node at (0,4.3) {$1$};
            \node at (-1,3.3) {$1$};
            \node at (-1.6,0.3) {$2^k$};
            \node at (0, 5.5) {$c_H$};
        \end{tikzpicture}
        \caption{}
    \end{figure}

    \medskip\noindent \textbf{Sub-case~2.1.} Suppose $Q + K_1 + 1 \le k < l$.
    
    Choose a point $y_0$ on a geodesic in $H$ between $u_0$ and $v_0$ with $d_H(u_0,y_0) = 2^k$. As before, there exists a bi-infinite geodesic $[c_H,d]$ in $X^h$ passing through $y_0$, where $d \in \partial X^h$. For each integer $n \ge 0$, denote $y_n =(y_0,n) \in [y_0,c_H]$. As in Sub-case~1.1, let $0 \le j \le k$ be the largest integer such that $k-j \ge Q + K_1 +1$. Then the geodesic ray $[y_j,d]$ lies entirely outside the $Q$-neighbourhood of $[a,u_l]$ and $d(x,y_j) \le M'$. Furthermore, since $d_H(y_0,v_0) = d_H(u_0,v_0)-d_H(u_0,y_0) \ge 2^{l-1}-2^k \ge 2^k,$ since $l>k$. Hence $d_{X^h}(y_k,v_k)\ge 1$, and therefore the geodesic ray $[y_j,d]$ also lies outside the $Q$-neighbourhood of $[v_l,b]$. This completes the sub-case.

    \medskip\noindent \textbf{Sub-case~2.2.} Suppose $k < Q + K_1 + 1 <l$.
    
    Choose a point $x' \in [u_0,u_l]$ with $d_{X^h}(u_0,x') = Q + K_1 + 1$. Applying Sub-case~2.1 to $x'$, we obtain a geodesic ray $\alpha$ that entirely lies outside the $Q$-neighbourhood of $[a,b]$ and $d_{X^h}(\alpha(0),x') \le M'$. Thus, $d_{X^h}(\alpha(0),x) \le M' + Q + K_1 + 1$. This completes the sub-case. 

    \medskip\noindent \textbf{Sub-case~2.3.} Suppose $Q + K_1 + 1 \le k = l$ or $ k \le l \le Q + K_1 + 1$.
    
    Consider the vertical ray $\alpha \coloneq [y_{l+Q+1},c_H]$. Then $\alpha$ lies outside the $Q$-neighbourhood of $[a,b]$ and $d_{X^h}(\alpha(0),x) \le 2Q + K_1 + 2$. This completes the sub-case.

    \medskip
    In all (sub) cases, there exists a geodesic ray $\alpha$ based at a point within distance $M$ of $x$ that lies entirely outside the $Q$-neighbourhood of $[a,b]$, where $M = M' + Q + K_1 + 1$ is a constant depending only on $\delta$. This completes the proof of the lemma.
\end{proof}
    
\begin{proof}[Proof of \autoref{prop_a_point_on_geodesic_vs_QP}]
    Let $\alpha$ be a geodesic ray as considered in \autoref{lem_a_point_and_a_ray}, based at a point $y \in X^h$ and approaching a boundary point $c \in \partial X^h$ such that $\alpha$ lies entirely outside the $Q$-neighbourhood of the geodesic $[a,b]$, where $c \ne a,b$, $Q=\max\{P,2\delta\}+1$ and $P = P(\delta) \ge 0$ is the constant given by the Bounded Projection Lemma (\autoref{lem_Bounded_projection_lemma}). Consider the geodesic triangle $\triangle(a,b,c)$.

    \begin{figure}[H]
        \centering
        \begin{tikzpicture}
            \draw (-3.5,0)--(3.5,0);
            \draw[bend left=40] (-1.5,2.5) edge (-3.5,0);
            \draw[bend right=25] (-1.5,2.5) edge (3.5,0);
            \draw[bend right=20] (-1.5,2.5) edge (0,0.4); 
            \draw[dashed] (0,0) -- (0,0.4); 

            \filldraw (-3.5,0) circle (1pt);
            \filldraw (3.5,0) circle (1pt);
            \filldraw (-1.5,2.5) circle (1pt);
            \filldraw (0,0.4) circle (1pt); 
            \filldraw (0,0) circle (1pt);

            \node at (-3.8,0) {$a$};
            \node at (3.8,0) {$b$};
            \node at (-1.5,2.8) {$c$};
            \node at (-1,0.4) {$\alpha(0)=y$};
            \node at (0,-0.3) {$x$};
        \end{tikzpicture}
        \caption{}
    \end{figure}
    
    We claim that the paths $[a,x] \cup [x,y] \cup [y,c]$ and $[c,y] \cup [y,x] \cup [x,b]$ are uniform quasigeodesics. To verify this, we consider the following cases.

    If $w \in [a,x]$ and $z \in [x,y]$, then
    \[
    d_{X^h}(w,x) + d_{X^h}(x,z) \le d_{X^h}(w,z) + 2M.
    \]

    If $w \in [c,y]$ and $z \in [x,y]$, then
    \[
    d_{X^h}(w,y) + d_{X^h}(y,z) \le d_{X^h}(w,z) + 2M.
    \] 

    If $w \in [a,x]$ and $z \in [y,c]$, then the segment $[y,z] \subseteq [y,c]$ lies entirely outside the $Q$-neighbourhood of $[a,b]$, with $Q>P$. Let $z'$ be a nearest point projection of $z$ onto $[a,b]$. By the Bounded Projection Lemma (\autoref{lem_Bounded_projection_lemma}), we have \(d_{X^h}(y,z) \le d_{X^h}(z,z') + K_1 + M\). Therefore, 
    \begin{align*}
        d_{X^h}(w,x) + d_{X^h}(x,y) + d_{X^h}(y,z) 
        & \le (d_{X^h}(w,z) + d_{X^h}(z,x)) + M + d_{X^h}(y,z) \\
        & \le d_{X^h}(w,z) + 2M + 2d_{X^h}(y,z)\\
        & \le d_{X^h}(w,z) + 2M + 2(d_{X^h}(z,z') + K_1 + M)\\
        & \le 3 d_{X^h}(z,w) + 2K_1+ 4M.
    \end{align*}   
    Thus, the path $[a,x] \cup [x,y] \cup [y,c]$ is a $(3,4M + 2K_1)$-quasigeodesic. A similar argument shows that $[c,y] \cup [y,x] \cup [x,b]$ is also a $(3,4M + 2K_1)$-quasigeodesic. 

    By \autoref{Prop_Stability_of_quasigeodesics}, there exists a constant $K_0 = K_0(\delta, 4M + 2K_1) \ge 0$ such that the point $x$ lies within distance $K_0$ of both sides $[a,c]$ and $[b,c]$ of $\triangle(a,b,c)$. Consequently, by \autoref{lem_bounded_diameter_of_points_on_a_geodesic}, there exists a constant $C_1 = C_1(\delta, K_0) \ge 0$ such that
    \[
    d_{X^h}(\pi_{X^h}(a,b,c), x) \le C_1.
    \]    
    Setting $R_1 = C_1(\delta, K_0)$ completes the proof of the proposition.
\end{proof}

\begin{cor}[Quasi-centres are coarsely dense]\label{cor_Quasi_centres_are_Coarsely_dense}
    The map $\pi_{X^h}: \partial^3 X^h \to X^h$ is coarsely surjective; that is, for every point $x \in X^h$, there exists a triple $(a,b,c) \in \partial^3 X^h$ such that $$d_{X^h}(\pi_{X^h}(a,b,c),x) \leq R_1,$$ 
    where $R_1 = R_1(\delta) \ge 0$ is the constant obtained from \autoref{prop_a_point_on_geodesic_vs_QP}. Moreover, the point $c \in \partial X^h$ can be chosen to be a parabolic endpoint. 
\end{cor}

\begin{proof}
    Without loss of generality, assume that $x$ is a vertex of $X^h$. Since cosets of parabolic subgroups cover the entire group $G$, there exists $H \in \mathcal{H}_G$ such that $x \in H^h$. Let $c_H \in \partial X^h$ denote the parabolic endpoint corresponding to $H$. Let $[a,c_H]$ be a geodesic, where $a \in \partial X^h\setminus\{c_H\}$ is chosen so that $x \in [a,c_H]$. By \autoref{prop_a_point_on_geodesic_vs_QP}, there exists a constant $R_1 = R_1(\delta) \ge 0$ and a point $b \in \partial X^h$, distinct from both $a$ and $c_H$, such that 
    $$d_{X^h}(\pi_{X^h}(a,b,c_H),x) \le R_1.$$
    This completes the proof.
\end{proof}

\begin{rem}\label{rem_a_point_vs_QP}
     The conclusion of \autoref{prop_a_point_on_geodesic_vs_QP} also holds in the setting of non-elementary Gromov-hyperbolic groups. But due to the absence of parabolic subgroups, the proof will differ. We will give a sketch of the proof.  

     Let $G$ be a non-elementary hyperbolic group with Cayley graph $X$. Let $a,b \in \partial X$, and let $x \in X$ be a point lying on a geodesic $[a,b]$ joining $a$ and $b$. It is well known that the set of pairs \(\{(g^{+\infty}, g^{-\infty}) \mid g \in G \text{ is hyperbolic}\} \subset \partial X \times \partial X \) is dense in $\partial X \times \partial X$. Hence, there exists a sequence of hyperbolic elements $\{g_k\} \subset G$ such that $(g_k^{+\infty}, g_k^{-\infty})$ converges to $(a,b)$ in $\partial X \times \partial X$. Then, there exists $k$ large enough and a point \(x_k \in [g_k^{+\infty}, g_k^{-\infty}]\) such that $d_X(x,x_k)$ is uniformly bounded. Let $c \in \partial X$ be any point distinct from $a,b,g_k^{+\infty},g_k^{-\infty}$. Consider the quasi-centre \(\pi_X(g_k^{+\infty}, g_k^{-\infty}, c),\) which lies in a uniformly bounded neighbourhood of the geodesic $[g_k^{+\infty}, g_k^{-\infty}]$. Since $g_k$ acts by translation along this geodesic, there exists a natural number $n(k)$ such that \(g_k^{n(k)}\bigl(\pi_X(g_k^{+\infty}, g_k^{-\infty}, c)\bigr) = \pi_X(g_k^{+\infty}, g_k^{-\infty}, g_k^{n(k)}(c))\) lies in a uniformly bounded neighbourhood of $x_k$. Consequently, the point $x \in [a,b]$ lies in a uniformly bounded neighbourhood of \(\pi_X(a,b,g_k^{n(k)}(c))\). This establishes the desired conclusion. 
\end{rem} 

We now generalize the previous results to pairs of points in the cusped space.
 
\begin{prop}\label{prop_two_points_vs_quasi_projections}
    Given $\delta \ge 0$, there exists a constant $R_2 = R_2(\delta) \ge 0$ such that the following holds. Let $G$ be a $\delta$-relatively hyperbolic group with a Cayley graph $X$. For any two distinct points $x, y \in X^h$, the following statements hold:
    \begin{enumerate}[$(i)$]
        \item There exists a bi-infinite geodesic $\gamma$ in $X^h$ such that both $x$ and $y$ lie within distance at most $R_2$ from $\gamma$. 

        \item There exist four distinct points $a,b,c,d \in \partial X^h$ such that 
        $$d_{X^h}(x, \pi_{X^h}(a,b,c)) \le R_2 \quad \text{and} \quad d_{X^h}(y, \pi_{X^h}(a,c,d)) \le R_2.$$ 

    \end{enumerate}
\end{prop} 

\begin{proof} 
    $(i)$ Let $x,y \in X^h$ be two distinct points. By \autoref{cor_Quasi_centres_are_Coarsely_dense}, there exists a constant $R_1 = R_1(\delta) \ge 0$ and triples $(a,b,c), (u,v,w) \in \partial^3 X^h$ such that $$d_{X^h}(x, \pi_{X^h}(a,b,c)) \le R_1 \quad \text{and} \quad d_{X^h}(y, \pi_{X^h}(u,v,w)) \le R_1.$$ 
    Therefore, the point $x$ (resp. $y$) lies within distance $R_1 + 5\delta$ from each side of the geodesic triangle $\triangle(a,b,c)$ (resp. $\triangle(u,v,w)$). We consider the following cases: 

    \medskip\noindent \textbf{Case~1.} Suppose $\{a,b,c\} = \{u,v,w\}$. 
    
    Then $\pi_{X^h}(a,b,c) = \pi_{X^h}(u,v,w)$, and hence both $x$ and $y$ lie within distance $R_1 + 5\delta$ from each of the three geodesics $[a,b],[b,c]$ and $[a,c]$. 

    \noindent \textbf{Case~2.} Suppose that $\{a,b,c\} \cap \{u,v,w\} \neq \emptyset$. 
    
    Without loss of generality, assume $b=v$. Consider the geodesic triangle $\triangle(a,b,u)$ if $a \ne u$; otherwise, choose any geodesic triangle of the form $\triangle(p,b,q)$, where $p \in \{a,c\}$ and $q \in \{u,w\}$, and $p \ne q$. 
    
    \begin{figure}[H]
        \centering
        \begin{tikzpicture}
            \draw (0,0) circle (2cm);
            \draw[bend right=70, thick] (2*cos 180,2*sin 180) edge (2*cos 140,2*sin 140);
            \draw[bend right=50, thick] (2*cos 140,2*sin 140) edge (2*cos 60,2*sin 60);
            \draw[bend right, thick] (2*cos 180,2*sin 180) edge (2*cos 60,2*sin 60);

            \draw[bend left, thick] (2*cos 300,2*sin 300) edge (2*cos 60,2*sin 60);
            \draw[bend right=50, thick] (2*cos 60,2*sin 60) edge (2*cos 0,2*sin 0);
            \draw[bend left=50, thick] (2*cos 300,2*sin 300) edge (2*cos 0,2*sin 0);

            \draw[bend right, thick,dashed] (2*cos 300,2*sin 300) edge (2*cos 180,2*sin 180);

            \filldraw (-.95,.6) circle (1pt);
            \filldraw (1,0) circle (1pt);

            \node at (-2.3,0) {$a$};
            \node at (1.5,2) {$b = v$};
            \node at (-1.78,1.42) {$c$};
            \node at (1.1,-2) {$u$};
            \node at (2.3,0) {$w$};

            \node at (-.7,.6) {$x$};
            \node at (.95,.3) {$y$};
            \end{tikzpicture}
        \caption{}
    \end{figure}
    
    Since $\triangle(a,b,u)$ is $\delta$-slim, the point $x$ lies within distance $R_1 + 6\delta$ from either $[a,b]$ or $[a,u]$, and similarly,  the point $y$ lies within distance $R_1 + 6\delta$ from either $[b,u]$ or $[a,u]$. Hence, both $x$ and $y$ lie within distance $R_1 + 6\delta$ of a common geodesic among $[a,b], [b,u]$ or $[a,u]$. 

    \medskip\noindent \textbf{Case~3.} Suppose $\{a,b,c\} \cup \{u,v,w\} = \emptyset$.
    
    Consider the four distinct points $a,b,c,v \in \partial X^h$. By \autoref{lem_one_of_three_cross-ratio_is_bdd}, one of the cross-ratios $|[a,b,c,v]|$, $|[a,c,b,v]|$ or $|[c,a,b,v]|$ is uniformly bounded by at most $C_6$. 

    \begin{figure}[H]
    \centering        
        \begin{tikzpicture}[scale=.9]
            \draw (0,0) circle (2cm);
            \draw[bend right=70, thick] (2*cos 180,2*sin 180) edge (2*cos 140,2*sin 140);
            \draw[bend right=50, thick] (2*cos 140,2*sin 140) edge (2*cos 60,2*sin 60);
            \draw[bend right, thick] (2*cos 180,2*sin 180) edge (2*cos 60,2*sin 60);

            \draw[bend left, thick] (2*cos 250,2*sin 250) edge (2*cos 0,2*sin 0);
            \draw[bend right=50, thick] (2*cos 0,2*sin 0) edge (2*cos 310,2*sin 310);
            \draw[bend left=50, thick] (2*cos 250,2*sin 250) edge (2*cos 310,2*sin 310);

            \draw[thick,dashed] (2*cos 180,2*sin 180) edge (2*cos 0,2*sin 0);
            \draw[bend right=20, thick,dashed] (2*cos 140,2*sin 140) edge (2*cos 0,2*sin 0);

            \filldraw (-.95,.6) circle (1pt);
            \filldraw (.7,-.8) circle (1pt);

            \node at (-2.3,0) {$a$};
            \node at (1.15,2) {$b$};
            \node at (-1.78,1.42) {$c$};
            \node at (2.3*cos 310,2.3*sin 310) {$w$};
            \node at (2.3,0) {$v$};
            \node at (2.3*cos 250,2.3*sin 250) {$u$};

            \node at (-1.1,.38) {$x$};
            \node at (1,-.55) {$y$};
        \end{tikzpicture}
        \hspace{.1cm}
        \begin{tikzpicture}[scale=.9]
            \draw (0,0) circle (2cm);
            \draw[bend right=70, thick] (2*cos 180,2*sin 180) edge (2*cos 140,2*sin 140);
            \draw[bend right=50, thick] (2*cos 140,2*sin 140) edge (2*cos 60,2*sin 60);
            \draw[bend right, thick] (2*cos 180,2*sin 180) edge (2*cos 60,2*sin 60);

            \draw[bend left, thick] (2*cos 250,2*sin 250) edge (2*cos 0,2*sin 0);
            \draw[bend right=50, thick] (2*cos 0,2*sin 0) edge (2*cos 310,2*sin 310);
            \draw[bend left=50, thick] (2*cos 250,2*sin 250) edge (2*cos 310,2*sin 310);

            \draw[thick,dashed] (2*cos 180,2*sin 180) edge (2*cos 0,2*sin 0);
            \draw[bend right=50, thick,dashed] (2*cos 60,2*sin 60) edge (2*cos 0,2*sin 0);

            \filldraw (-.95,.6) circle (1pt);
            \filldraw (.7,-.8) circle (1pt);

            \node at (-2.3,0) {$a$};
            \node at (1.15,2) {$b$};
            \node at (-1.78,1.42) {$c$};
            \node at (2.3*cos 310,2.3*sin 310) {$w$};
            \node at (2.3,0) {$v$};
            \node at (2.3*cos 250,2.3*sin 250) {$u$};

            \node at (-1.1,.38) {$x$};
            \node at (1,-.55) {$y$};
        \end{tikzpicture}
        \hspace{.1cm}
        \begin{tikzpicture}[scale=.9]
            \draw (0,0) circle (2cm);
            \draw[bend right=70, thick] (2*cos 180,2*sin 180) edge (2*cos 140,2*sin 140);
            \draw[bend right=50, thick] (2*cos 140,2*sin 140) edge (2*cos 60,2*sin 60);
            \draw[bend right, thick] (2*cos 180,2*sin 180) edge (2*cos 60,2*sin 60);

            \draw[bend left, thick] (2*cos 250,2*sin 250) edge (2*cos 0,2*sin 0);
            \draw[bend right=50, thick] (2*cos 0,2*sin 0) edge (2*cos 310,2*sin 310);
            \draw[bend left=50, thick] (2*cos 250,2*sin 250) edge (2*cos 310,2*sin 310);

            \draw[bend right=20, thick,dashed] (2*cos 140,2*sin 140) edge (2*cos 0,2*sin 0);
            \draw[bend right=50, thick,dashed] (2*cos 60,2*sin 60) edge (2*cos 0,2*sin 0);

            \filldraw (-.95,.6) circle (1pt);
            \filldraw (.7,-.8) circle (1pt);

            \node at (-2.3,0) {$a$};
            \node at (1.15,2) {$b$};
            \node at (-1.78,1.42) {$c$};
            \node at (2.3*cos 310,2.3*sin 310) {$w$};
            \node at (2.3,0) {$v$};
            \node at (2.3*cos 250,2.3*sin 250) {$u$};

            \node at (-1.1,.38) {$x$};
            \node at (1,-.55) {$y$};
        \end{tikzpicture}
    \caption{}
    \end{figure}
    
    Suppose, $|[a,b,c,v]| \le C_6$. Then, by \autoref{lem_cross_ratios_vs_quasi_centres}, we have $$d_{X^h}(\pi_{X^h}(a,b,c), \pi_{X^h}(a,c,v) \le C_5 + C_6.$$
    Hence, $x$ lies within distance $R_1 + C_5 + C_6$ to the quasi-centres of $\triangle(a,c,v)$. Similarly, if $|[a,c,b,v]|$ or $|[c,a,b,v]|$ is bounded above by $C_6$, then the point $x$ lies within distance $R_1 + C_5 + C_6$ to the quasi-centres of $\triangle(a,b,v)$ (respectively $\triangle(b,c,v)$).

    In each case, $x$ lies within distance $R_1 + C_5 + C_6$ to the quasi-centres of a triangle that shares a vertex $v$ with the triangle $\triangle(u,v,w)$. Therefore, by applying Case~2, we conclude that both $x$ and $y$ lie within distance $R_1 + C_5 + C_6 + 6\delta$ from some bi-infinite geodesic. 

    Setting $R_2' = R_1 + C_5 + C_6 + 6\delta$ completes the proof of part $(i)$.

    \medskip 
    $(ii)$ By part $(i)$, both $x$ and $y$ lie within distance $R_2'$ from some bi-infinite geodesic $[a,c]$, where $a,c \in \partial X^h$. Applying \autoref{prop_a_point_on_geodesic_vs_QP}, there exist $b,d \in \partial X^h$, each distinct from both $a$ and $c$, such that 
    $$d_{X^h}(x,\pi_{X^h}(a,b,c)) \le R_1 + R_2' \quad \text{and} \quad d_{X^h}(y,\pi_{X^h}(a,c,d)) \le R_1 + R_2'.$$

    If $b \ne d$, then we are done. Otherwise, choose a point $y' \in [a,c]$ such that $d_{X^h}(y,y') = 2R_1 + R_2' + 1$. Again by \autoref{prop_a_point_on_geodesic_vs_QP}, there exists a point $d' \in \partial X^h$, distinct from $a$ and $c$, such that $d_{X^h}(y',\pi_{X^h}(a,c,d')) \le R_1.$    
    Now, suppose for contradiction that $ b = d'$. Then we would have 
    $$d_{X^H}(y,y') \le d_{X^h}(y,\pi_{X^h}(a,c,b)) + d_{X^h}(y',\pi_{X^h}(a,c,d')) \le 2R_1+R_2',$$
    which contradicts the choice of $y'$. Therefore, $b \ne d'$, and we obtain four distinct points $a,b,c,d'$ such that
    $$d_{X^h}(x,\pi_{X^h}(a,b,c)) \le R+R_2' \quad \text{and} \quad d_{X^h}(y,\pi_{X^h}(a,c,d')) \le 3R_1 + R_2' + 1.$$ 
    Setting $R_2''=3R + R_2' + 1$ completes the proof of part $(ii)$.
    

    Hence, setting $R_2 = \max\{R_2' + R_2'' \}$, completes the proof of the proposition.
\end{proof} 


\subsection{Quasi-isometries Induced by Relative Quasi-M\"{o}bius Maps} \label{subsec_proof_of_thm_2} 

Let $(G_1, \mathcal{H}_{G_1})$ and $(G_2, \mathcal{H}_{G_2})$ be two $\delta$--relatively hyperbolic groups, with Cayley graphs $X$ and $Y$, respectively. Suppose that \(f \colon \partial X^h \to \partial Y^h\) is a $\psi$-quasi-M\"{o}bius homeomorphism between their Bowditch boundaries, for some continuous distortion function $\psi \colon [0,\infty) \to [0,\infty)$. In this subsection, we prove \autoref{thm_main_theorem_2}. We first show that the control function $\psi$ of the quasi-M\"obius homeomorphism $f$ can be effectively replaced by a linear function. 

\begin{lem}\label{lem_linearity_for_QM}
    Let $f:\partial X^h\to\partial Y^h$ be a $\psi$-quasi-M\"{o}bius homeomorphism between the Bowditch boundaries, for some continuous distortion function $\psi : [0,\infty) \to [0,\infty)$. Then there exists a linear distortion function $\psi'(t) = At+B$, where the constants $A \ge 1$ and $B \ge 0$ depend only on $\delta$ and $\psi$, such that $f$ is also a $\psi'$-quasi-M\"{o}bius homeomorphism.
\end{lem}

\begin{proof}
    Let $a,b,c,d \in \partial X^h$ be pairwise distinct points. Suppose $p,q \in [a,c]$ are points such that \(d_{X^h}(p,\pi_{X^h}(a,b,c))\le 5\delta\) and \(d_{X^h}(q,\pi_{X^h}(a,c,d))\le 5\delta\). 
    Choose points $p=p_0,p_1,\dots,p_n=q$ on the geodesic $[a,c]$ such that for each $1\le i \le n-1$,  $p_i \in [p_{i-1},p_{i+1}]$, $d_{X^h}(p_i,p_{i-1}) = 2R_1+1$, and $d_{X^h}(p_n,p_{n-1}) \le 2R_1+1$, where $R_1 = R_1(\delta) \ge 0$ is the constant as considered in \autoref{prop_a_point_on_geodesic_vs_QP}.
    
    \begin{figure}[H]
        \centering
        \begin{tikzpicture}[scale=.9]
            \draw[line width=1pt] (-3.5,0) -- (3.5,0);
            \draw (-3.5,0) -- (-3,2.5) -- (3.5,0);
            \draw[line width=.4pt] (-3.5,0) -- (-2,2.5) -- (3.5,0);
            \draw[line width=.8pt] (-3.5,0) -- (2,2.5) -- (3.5,0);
            \draw[line width=1pt] (-3.5,0) -- (3,2.5) -- (3.5,0); 

            \filldraw (-2,0) circle (1pt);
            \filldraw (-1,0) circle (1pt);
            \filldraw (1,0) circle (1pt);
            \filldraw (2,0) circle (1pt);

            \node at (-3.8,0) {$a$};
            \node at (3.8,0) {$b$};
            \node at (-2,-.3) {$p_0$};
            \node at (-1,-.3) {$p_1$};
            \node at (0,-.3) {$\dots$};
            \node at (1,-.3) {$p_{n-1}$};
            \node at (2,-.3) {$p_n$};

            \node at (-3,2.8) {$b_0$};
            \node at (-2,2.8) {$b_1$};
            \node at (0,2.8) {$\dots$};
            \node at (2,2.8) {$b_{n-1}$};
            \node at (3,2.8) {$b_n$};
        \end{tikzpicture}
        \caption{}
    \end{figure} 
    
    By \autoref{prop_a_point_on_geodesic_vs_QP}, there exist boundary points $b=b_0,b_1,\dots,b_n=d$ on $\partial X^h$ such that for each $1 \le i \le n-1$,
    $$d_{X^h}(p_i,\pi_{X^h}(a,b_i,c)) \le R_1.$$ 
    Moreover, all points $b_i$ are distinct; otherwise, for some $i\neq j$, we would have \(d_{X^h}(p_i,p_j) \le 2R_1\), contradicting the construction of the sequence $\{p_i\}$. Hence, by \autoref{lem_cross_ratios_vs_quasi_centres}, we have
    $|[a,b_{i-1},c,b_i]| \le 4R_1 + C_5+1$. Since $f$ is $\psi$-quasi-M\"{o}bius, then $|[fa,fb_{i-1},fc,fb_{i}]| \le \psi (|[a,b_{i-1},c,b_i]|)$.    
    Set $L = \sup \{\psi(t)\;:\; t \in [0,4R_1 + C_5 + 1] \}$. By  \autoref{lem_cross_ratios_vs_quasi_centres}, we have
    \[
    d_{Y^h}(\pi_{Y^h}(fa,fb_{i-1},fc), \pi_{Y^h}(fa,fb_i,fc)) \le L + C_5.
    \]
    Therefore,
    \begin{align*}
        |[fa,fb,fc,fd]| & \le d_{Y^h}(\pi_{Y^h}(fa,fb,fc),\pi_{Y^h}(fa,fc,fd)) + C_5 \\
        & \le \sum\limits_{i=1}^n d_{Y^h}(\pi_{Y^h}(fa,fb_{i-1},fc), \pi_{Y^h}(fa,fb_i,fc)) + C_5 \\
        & \le n(L + C_5) + C_5 \\
        & \le \frac{(L+C_5)}{2R_1 + 1} (d_{X^h}(p,q) + 1) + C_5 \\
        &\le \frac{(L+C_5)}{2R_1 + 1}( |[a,b,c,d]| + 10\delta + C_5 + 1) + C_5.
    \end{align*}
    Hence, setting $A = \max\{ \frac{L + C_5}{2R_1 + 1},1\}$ and $B = \frac{(L+C_5)}{2R_1 + 1}(10\delta + C_5 + 1) + C_5$, we conclude that $$|[fa,fb,fc,fd]| \precapprox_{A,B} |[a,b,c,d]|.$$
    
    Since $f^{-1}$ is also $\psi$-quasi-M\"{o}bius, a similar argument applies to show that for any four distinct points $a',b',c',d' \in \partial Y^h$, we have
    $$|[f^{-1}(a'),f^{-1}(b'),f^{-1}(c'),f^{-1}(d')]| \precapprox_{A,B} |[a',b',c',d']|.$$
    This proves the lemma.
\end{proof}


Our next aim is to construct a quasi-isometry $\Phi f: X^h\to Y^h$ from a given $\psi$-quasi-M\"{o}bius homeomorphism $f: \partial X^h \to \partial Y^h$. For that, we first need the following construction of a uniformly bounded subset in $Y^h$.

\begin{defn}\label{defn_defining_set_E(.)}
    Let $R = \max \{R_1, R_2\}$, where $R_1 = R_1(\delta) \ge 0$ and $R_2 = R_2(\delta) \ge 0$ are the constants obtained from \autoref{cor_Quasi_centres_are_Coarsely_dense} and \autoref{prop_two_points_vs_quasi_projections}, respectively. For each point $x \in X^h$, define the set 
    \[ 
    E(x) \coloneqq \pi_{Y^h} \circ f \circ \pi_{X^h}^{-1}(\overline{B}(x,R)),
    \]
    where $\overline{B}(x,R)$ denotes the closed ball of radius $R$ centred at $x \in {X}^h$. Here, $\pi_{X^h}^{-1}(\overline{B}(x,R))$ denotes the collection of all boundary triples $(a,b,c) \in \partial^3X^h$ such that their quasi-centre lies within distance $R$ of $x$; that is, $d_{X^h}(\pi_{X^h}(a,b,c),x) \le R$.
\end{defn}


\begin{lem} \label{lem_E(x)_is_bdd}
    There exists a constant $D = D(\delta, \psi) \ge 0$ such that for every point $x \in X^h$, the set $E(x)$ is a nonempty, bounded subset of $Y^h$, which has diameter at most $D$.
\end{lem}

\begin{proof}
    For each \(x \in X^h\), the non-emptiness of the set \(E(x)\) follows from \autoref{cor_Quasi_centres_are_Coarsely_dense}. To verify boundedness, let $y_1=\pi_{Y^h}(fa,fb,fc)$ and $y_2 = \pi_{Y^h}(fu,fv,fw)$ be two arbitrary points in $E(x)$, where $(a,b,c), (u,v,w) \in \partial^3X^h$ satisfying
    \[
        d_{X^h}(\pi_{X^h}(a,b,c), x) \le R \quad \text{and} \quad d_{X^h}(\pi_{X^h}(u,v,w), x) \le R.
    \]

    If $\{a,b,c\} = \{u,v,w\}$, then $y_1 = y_2$, and hence there is nothing more to prove. 
    Otherwise, consider the case where $\{a,b,c\} \cap \{u,v,w\} = \emptyset$; the remaining cases follow similarly. 

    We apply \autoref{lem_one_of_three_cross-ratio_is_bdd} to the quadruple $(a,b,c,v)$ of distinct points. After possibly permuting $(a,b,c)$, we may assume that
    \[
    |[a,b,c,v]| \le C_6.
    \] 
    Then, by \autoref{lem_cross_ratios_vs_quasi_centres}, we have 
    \[
    d_{X^h}(\pi_{X^h}(a,b,c), \pi_{X^h}(a,c,v)) \le C_5 + C_6.
    \]
    Next, we apply \autoref{lem_one_of_three_cross-ratio_is_bdd} to the quadruple $(a,c,v,w)$ of distinct points. Then, at least one of the three cross-ratios \(|[a,c,v,w]|\), \(|[c,a,v,w]|\), or \(|[a,v,c,w]|\) is bounded by $C_6$.

    \medskip\noindent\textbf{Case~1.} Suppose that $|[a,c,v,w]| \le C_6$. 
    
    By \autoref{lem_cross_ratios_vs_quasi_centres}, the quasi-centres of the triangles $\triangle(a,c,v)$ and $\triangle(a,v,w)$ are within distance at most $C_5 + C_6$ of each other.
    
    \begin{figure}[H]
        \centering
        \begin{tikzpicture}[scale=2.1]
            \coordinate (O) at (0,0);
            \coordinate (A) at (1,0);
             \tkzDrawCircle[black](O,A);
            \node at ({1.1*cos (150)},{1.1*sin (150)}) {$a$};
            \node at ({1.1*cos (100)},{1.1*sin (100)}) {$b$};
            \node at ({1.1*cos (50)},{1.1*sin (50)}) {$c$};
            \node at ({1.1*cos (-20)},{1.1*sin (-20)}) {$v$};
            \node at ({1.1*cos (-120)},{1.1*sin (-120)}) {$w$};
            \node at ({1.1*cos (-160)},{1.1*sin (-160)}) {$u$};

            \coordinate(z1) at({cos (100+50)},{sin (100+50)});
            \coordinate(z2) at({cos (100)},{sin (100)});
            \tkzDefCircle[orthogonal through=z1 and z2](O,A) \tkzGetPoint{B1};
            \tkzClipCircle(O,A);
             \tkzDrawCircle[black](B1,z2);

            \coordinate(z1) at({cos (100)},{sin (100)});
            \coordinate(z2) at({cos (100-50)},{sin (100-50)});
            \tkzDefCircle[orthogonal through=z1 and z2](O,A) \tkzGetPoint{B1};
            \tkzClipCircle(O,A);
             \tkzDrawCircle[black](B1,z2);
            
            \coordinate(z1) at({cos (100-50)},{sin (100-50)});
            \coordinate(z2) at({cos (100+50)},{sin (100+50)});
            \tkzDefCircle[orthogonal through=z1 and z2](O,A) \tkzGetPoint{B1};
            \tkzClipCircle(O,A);
             \tkzDrawCircle[black](B1,z2);            
            \coordinate(z1) at({cos (-20)},{sin (-20)});
            \coordinate(z2) at({cos (50)},{sin (50)});
            \tkzDefCircle[orthogonal through=z1 and z2](O,A) \tkzGetPoint{B1};
            \tkzClipCircle(O,A);
            \tkzDrawCircle[dashed, black](B1,z2);

            \coordinate(z2) at({cos (150)},{sin (150)});
            \coordinate(z1) at({cos (340)},{sin (340)});
            \tkzDefCircle[orthogonal through=z1 and z2](O,A) \tkzGetPoint{B1};
            \tkzClipCircle(O,A);
            \tkzDrawCircle[dashed, black](B1,z2);
            
            \coordinate(z1) at({cos (-120)},{sin (-120)});
            \coordinate(z2) at({cos (150)},{sin (150)});
            \tkzDefCircle[orthogonal through=z1 and z2](O,A) \tkzGetPoint{B1};
            \tkzClipCircle(O,A);
            \tkzDrawCircle[dashed, black](B1,z2);
            \coordinate(z1) at({cos (-20)},{sin (-20)});
            \coordinate(z2) at({cos (-120)},{sin (-120)});
            \tkzDefCircle[orthogonal through=z1 and z2](O,A) \tkzGetPoint{B1};
            \tkzClipCircle(O,A);
             \tkzDrawCircle[black](B1,z2);
            
            \coordinate(z1) at({cos (-120)},{sin (-120)});
            \coordinate(z2) at({cos (-160)},{sin (-160)});
            \tkzDefCircle[orthogonal through=z1 and z2](O,A) \tkzGetPoint{B1};
            \tkzClipCircle(O,A);
             \tkzDrawCircle[black](B1,z2);

            \coordinate(z1) at({cos (-160)},{sin (-160)});
            \coordinate(z2) at({cos (-20)},{sin (-20)});
            \tkzDefCircle[orthogonal through=z1 and z2](O,A) \tkzGetPoint{B1};
            \tkzClipCircle(O,A);
             \tkzDrawCircle[black](B1,z2);
            \fill (-.1,.5) circle (.4pt);
            \fill (.1,0.2) circle (.4pt);
            \fill (-.1,-.1) circle (.4pt);
            \fill (-.27,-.37) circle (.4pt);
        \end{tikzpicture}
        \caption{}
    \end{figure}

    On the other hand, the triangle $\triangle(a,v,w)$ shares an edge $[v,w]$ with $\triangle(u,v,w)$. Hence, applying \autoref{lem_cross_ratios_vs_quasi_centres}, the cross-ratio satisfies  
    \begin{align*}
        |[v,u,w,a]| 
        &\le d_{X^h}(\pi_{X^h}(u,v,w),\pi_{X^h}(a,v,w)) + C_5 \\ 
        & \le d_{X^h}(\pi_{X^h}(u,v,w),\pi_{X^h}(a,b,c)) +d_{X^h}(\pi_{X^h}(a,b,c),\pi_{X^h}(a,v,c))\\ 
        &\qquad +d_{X^h}(\pi_{X^h}(a,v,c),\pi_{X^h}(a,v,w)) + C_5 \\
        &\le 2(R + C_5 + C_6) + C_5.
    \end{align*}
    Since $f$ is $\psi$-quasi-M\"{o}bius, we conclude
    \begin{align*}
        d_{Y^h}(y_1,y_2)&= d_{Y^h}(\pi_{Y^h}(fa,fb,fc),\pi_{Y^h}(fu,fv,fw))\\
        &\le d_{Y^h}(\pi_{Y^h}(fa,fb,fc),\pi_{Y^h}(fa,fv,fc))\\
        &\qquad +d_{Y^h}(\pi_{Y^h}(fa,fv,fc),\pi_{Y^h}(fa,fv,fw))\\
        &\qquad +d_{Y^h}(\pi_{Y^h}(fa,fv,fw),\pi_{Y^h}(fu,fv,fw))\\
        &\le |[fa,fb,fc,fv]|+|[fa,fc,fv,fw]|+|[fv,fu,fw,fa]|+3C_5\\
        &\le \psi(|[a,b,c,v]|)+\psi(|[a,c,v,w]|)+\psi(|[v,u,w,a]|)+3C_5
    \end{align*} 

    All the cross-ratios involved on the right-hand side are uniformly bounded above by constants depending only on $\delta$ and $R$, and $R$ is fixed. Hence, the whole expression is uniformly bounded by some constant, say $D' = D'(\delta,\psi)\ge 0$, which completes the proof for this case.

    \medskip\noindent\textbf{Case~2.} Suppose that $|[c,a,v,w]| \le C_6$.
    
    By applying \autoref{lem_cross_ratios_vs_quasi_centres}, we conclude that quasi-centres of $\triangle(a,c,v)$ and $\triangle(c,v,w)$ are bounded above by $C_6 + C_5$.

    \begin{figure}[H]
        \centering
        \begin{tikzpicture}[scale=2.1]
            \coordinate (O) at (0,0);
            \coordinate (A) at (1,0);
             \tkzDrawCircle[black](O,A);
            \node at ({1.1*cos (150)},{1.1*sin (150)}) {$a$};
            \node at ({1.1*cos (100)},{1.1*sin (100)}) {$b$};
            \node at ({1.1*cos (50)},{1.1*sin (50)}) {$c$};
            \node at ({1.1*cos (-20)},{1.1*sin (-20)}) {$v$};
            \node at ({1.1*cos (-120)},{1.1*sin (-120)}) {$w$};
            \node at ({1.1*cos (-160)},{1.1*sin (-160)}) {$u$};

            \coordinate(z1) at({cos (100+50)},{sin (100+50)});
            \coordinate(z2) at({cos (100)},{sin (100)});
            \tkzDefCircle[orthogonal through=z1 and z2](O,A) \tkzGetPoint{B1};
            \tkzClipCircle(O,A);
             \tkzDrawCircle[black](B1,z2);

            \coordinate(z1) at({cos (100)},{sin (100)});
            \coordinate(z2) at({cos (100-50)},{sin (100-50)});
            \tkzDefCircle[orthogonal through=z1 and z2](O,A) \tkzGetPoint{B1};
            \tkzClipCircle(O,A);
             \tkzDrawCircle[black](B1,z2);
            
            \coordinate(z1) at({cos (100-50)},{sin (100-50)});
            \coordinate(z2) at({cos (100+50)},{sin (100+50)});
            \tkzDefCircle[orthogonal through=z1 and z2](O,A) \tkzGetPoint{B1};
            \tkzClipCircle(O,A);
             \tkzDrawCircle[black, black](B1,z2);            
            \coordinate(z1) at({cos (-20)},{sin (-20)});
            \coordinate(z2) at({cos (50)},{sin (50)});
            \tkzDefCircle[orthogonal through=z1 and z2](O,A) \tkzGetPoint{B1};
            \tkzClipCircle(O,A);
            \tkzDrawCircle[dashed, black](B1,z2);

            \coordinate(z2) at({cos (150)},{sin (150)});
            \coordinate(z1) at({cos (340)},{sin (340)});
            \tkzDefCircle[orthogonal through=z1 and z2](O,A) \tkzGetPoint{B1};
            \tkzClipCircle(O,A);
            \tkzDrawCircle[dashed, black](B1,z2);
            
            \coordinate(z1) at({cos (-120)},{sin (-120)});
            \coordinate(z2) at({cos (50)},{sin (50)});
            \tkzDefCircle[orthogonal through=z1 and z2](O,A) \tkzGetPoint{B1};
            \tkzClipCircle(O,A);
            \tkzDrawCircle[dashed, black](B1,z2);
            \coordinate(z1) at({cos (-20)},{sin (-20)});
            \coordinate(z2) at({cos (-120)},{sin (-120)});
            \tkzDefCircle[orthogonal through=z1 and z2](O,A) \tkzGetPoint{B1};
            \tkzClipCircle(O,A);
             \tkzDrawCircle[black](B1,z2);
            
            \coordinate(z1) at({cos (-120)},{sin (-120)});
            \coordinate(z2) at({cos (-160)},{sin (-160)});
            \tkzDefCircle[orthogonal through=z1 and z2](O,A) \tkzGetPoint{B1};
            \tkzClipCircle(O,A);
             \tkzDrawCircle[black](B1,z2);

            \coordinate(z1) at({cos (-160)},{sin (-160)});
            \coordinate(z2) at({cos (-20)},{sin (-20)});
            \tkzDefCircle[orthogonal through=z1 and z2](O,A) \tkzGetPoint{B1};
            \tkzClipCircle(O,A);
             \tkzDrawCircle[black](B1,z2);
            \fill (-.1,.5) circle (.4pt);
            \fill (.1,0.2) circle (.4pt);
            \fill (.2,-.1) circle (.4pt);
            \fill (-.27,-.37) circle (.4pt);
        \end{tikzpicture}
        \caption{}
    \end{figure}

    Moreover, in this situation, $\triangle(c,v,w)$ shares an edge $[v,w]$ with $\triangle(u,v,w)$. Similar to Case~1, there exists a constant $D''=D''(\delta,\psi) \ge 0$ such that $d_{Y^h}(y_1,y_2)\le D'',$ which provides a uniform bound $D''$ for this case.

    \medskip\noindent\textbf{Case~3.} Suppose that $|[a,v,c,w]| \le C_6$. 
    
    By applying \autoref{lem_cross_ratios_vs_quasi_centres}, we get
    $$d_{X^h}(\pi_{X^h}(a,v,c),\pi_{X^h}(a,c,w))\le C_6 + C_5.$$ 
    In this case, neither $\triangle(a,c,v)$ nor $\triangle(a,c,w)$ shares an edge with $\triangle(u,v,w)$. 

    \begin{figure}[H]
        \centering
        \begin{tikzpicture}[scale=2.1]
            \coordinate (O) at (0,0);
            \coordinate (A) at (1,0);
            \tkzDrawCircle[black](O,A);
            \node at ({1.1*cos (150)},{1.1*sin (150)}) {$a$};
            \node at ({1.1*cos (100)},{1.1*sin (100)}) {$b$};
            \node at ({1.1*cos (50)},{1.1*sin (50)}) {$c$};
            \node at ({1.1*cos (-20)},{1.1*sin (-20)}) {$v$};
            \node at ({1.1*cos (-120)},{1.1*sin (-120)}) {$w$};
            \node at ({1.1*cos (-160)},{1.1*sin (-160)}) {$u$};

            \coordinate(z1) at({cos (100+50)},{sin (100+50)});
            \coordinate(z2) at({cos (100)},{sin (100)});
            \tkzDefCircle[orthogonal through=z1 and z2](O,A) \tkzGetPoint{B1};
            \tkzClipCircle(O,A);
             \tkzDrawCircle[black](B1,z2);

            \coordinate(z1) at({cos (100)},{sin (100)});
            \coordinate(z2) at({cos (100-50)},{sin (100-50)});
            \tkzDefCircle[orthogonal through=z1 and z2](O,A) \tkzGetPoint{B1};
            \tkzClipCircle(O,A);
             \tkzDrawCircle[black](B1,z2);
            
            \coordinate(z1) at({cos (100-50)},{sin (100-50)});
            \coordinate(z2) at({cos (100+50)},{sin (100+50)});
            \tkzDefCircle[orthogonal through=z1 and z2](O,A) \tkzGetPoint{B1};
            \tkzClipCircle(O,A);
             \tkzDrawCircle[black](B1,z2);            
            \coordinate(z1) at({cos (-20)},{sin (-20)});
            \coordinate(z2) at({cos (50)},{sin (50)});
            \tkzDefCircle[orthogonal through=z1 and z2](O,A) \tkzGetPoint{B1};
            \tkzClipCircle(O,A);
            \tkzDrawCircle[dashed, black](B1,z2);

            \coordinate(z2) at({cos (150)},{sin (150)});
            \coordinate(z1) at({cos (340)},{sin (340)});
            \tkzDefCircle[orthogonal through=z1 and z2](O,A) \tkzGetPoint{B1};
            \tkzClipCircle(O,A);
            \tkzDrawCircle[dashed, black](B1,z2);
            
            \coordinate(z1) at({cos (-120)},{sin (-120)});
            \coordinate(z2) at({cos (150)},{sin (150)});
            \tkzDefCircle[orthogonal through=z1 and z2](O,A) \tkzGetPoint{B1};
            \tkzClipCircle(O,A);
            \tkzDrawCircle[dashed, black](B1,z2);

             \coordinate(z1) at({cos (-120)},{sin (-120)});
            \coordinate(z2) at({cos (50)},{sin (50)});
            \tkzDefCircle[orthogonal through=z1 and z2](O,A) \tkzGetPoint{B1};
            \tkzClipCircle(O,A);
            \tkzDrawCircle[dashed, black](B1,z2);
            \coordinate(z1) at({cos (-20)},{sin (-20)});
            \coordinate(z2) at({cos (-120)},{sin (-120)});
            \tkzDefCircle[orthogonal through=z1 and z2](O,A) \tkzGetPoint{B1};
            \tkzClipCircle(O,A);
             \tkzDrawCircle[black](B1,z2);
            
            \coordinate(z1) at({cos (-120)},{sin (-120)});
            \coordinate(z2) at({cos (-160)},{sin (-160)});
            \tkzDefCircle[orthogonal through=z1 and z2](O,A) \tkzGetPoint{B1};
            \tkzClipCircle(O,A);
             \tkzDrawCircle[black](B1,z2);

            \coordinate(z1) at({cos (-160)},{sin (-160)});
            \coordinate(z2) at({cos (-20)},{sin (-20)});
            \tkzDefCircle[orthogonal through=z1 and z2](O,A) \tkzGetPoint{B1};
            \tkzClipCircle(O,A);
             \tkzDrawCircle[black](B1,z2);
            \fill (-.1,.5) circle (.4pt);
            \fill (.1,0.2) circle (.4pt);
            \fill (-0.25,0) circle (.4pt);
            \fill (-.27,-.37) circle (.4pt);
        \end{tikzpicture}
        \caption{}
    \end{figure}  

    We therefore apply \autoref{lem_one_of_three_cross-ratio_is_bdd} to the quadruple $(u,v,w,a)$. It follows that one of the cross-ratios $|[v,u,w,a]|$, $|[v,w,u,a]|$ or $|[w,v,u,a]|$ is bounded above by $C_6$. Applying \autoref{lem_cross_ratios_vs_quasi_centres} once again, we conclude that the quasi-centre $\pi_{X^h}(u,v,w)$ lies within distance at most $C_6 + C_5$ of one of the quasi-centres $\pi_{X^h}(a,v,w)$, $\pi_{X^h}(u,a,w)$, or $\pi_{X^h}(u,v,a)$. Moreover, the corresponding triangles share an edge with either $\triangle(a,v,c)$ or $\triangle(a,w,c)$. 

    Without loss of generality, assume that $\pi_{X^h}(u,v,w)$ lies within distance $C_6 + C_5$ of $\pi_{X^h}(a,v,w)$ (the remaining cases are analogous), and that $|[v,u,w,a]|\le C_6.$ 

    In this situation, the triangle $\triangle(a,v,w)$ shares the edge $[a,v]$ with $\triangle(a,c,v)$. Hence, by \autoref{lem_cross_ratios_vs_quasi_centres}, the cross-ratio satisfies
    \begin{align*}
        |[a,c,v,w]| 
        & \le d_{X^h}(\pi_{X^h}(a,v,w),\pi_{X^h}(a,v,c)) + C_5 \\ 
        &\le d_{X^h}(\pi_{X^h}(a,v,w),\pi_{X^h}(u,v,w)) \\ 
        &\qquad +d_{X^h}(\pi_{X^h}(u,v,w),\pi_{X^h}(a,b,c))\\ 
        &\qquad +d_{X^h}(\pi_{X^h}(a,b,c),\pi_{X^h}(a,v,c)) + C_5 \\
        &\le 2 ( C_6 + C_5 + R) + C_5.
    \end{align*}
    As in Case~1, using the $\psi$-quasi-M\"{o}bius property of $f$, we obtain a constant $D''' = D'''(\delta,\psi) \ge 0$ such that $d_{Y^h}(y_1,y_2) \le D'''$. This provides a uniform bound $D'''$ for this case.
    
    Finally, setting, $D = \max\{D', D'', D'''\}$, we conclude that for every $x \in X^h$, the set $E(x)$ is uniformly bounded and has diameter at most $D$.
\end{proof} 

We are now set to prove \autoref{thm_main_theorem_2}, whose proof  follows
 from  \autoref{thm_QM_implies_qi_on_cusped_spaces},  \autoref{thm_rel_QM_implies_cusp_preserving_qi_on_cusped_spaces}, \autoref{prop_boundaries_coincide_1}, \autoref{thm_cusp_preserving_qi_implies_ambient_qi} and \autoref{prop_boundaries_coincide},.
 
\begin{thm}\label{thm_QM_implies_qi_on_cusped_spaces}
    There exist constants $\lambda_1 = \lambda_1(\delta,\psi) \ge 1$ and $\epsilon_1 = \epsilon_1(\delta,\psi) \ge 0$ such that the map $\Phi f: X^h \to Y^h$, which assigns to each point $x \in X^h$ a point in $E(x)$, is a $(\lambda_1,\epsilon_1)$-quasi-isometry. Moreover, any two such choices of the maps $\Phi f$ differ by at most the uniform constant $D$ from \autoref{lem_E(x)_is_bdd}.
\end{thm}

\begin{proof} 
    By \autoref{lem_linearity_for_QM}, we may assume without loss of generality that $\psi$ is linear; that is, \(\psi(t) = At + B\) for some constants $A \ge 1$ and $B \ge 0$. 

    Let $x,y \in X^h$ be arbitrary points. By \autoref{prop_two_points_vs_quasi_projections}, there exists a constant $R_2 = R_2(\delta) \ge 0$ and four distinct points $a,b,c,d \in \partial X^h$ such that 
    \[ 
    d_{X^h}(\pi_{X^h}(a,b,c),x) \le R_2 \quad \text{and} \quad d_{X^h}(\pi_{X^h}(a,c,d),y) \le R_2.
    \]
    Hence, by \autoref{lem_cross_ratios_vs_quasi_centres}, we obtain 
    \begin{align*}
        |[a,b,c,d]| &\approx_{C_5} d_{X^h}(\pi_{X^h}(a,b,c),\pi_{X^h}(a,c,d)) \\ 
        &\approx_{C_5 + 2R_2} d_{X^h}(x,y).
    \end{align*}    
    Since $\pi_{X^h}(fa,fb,fc) \in E(x)$ and $\pi_{X^h}(fa,fc,fd) \in E(y)$, by \autoref{lem_cross_ratios_vs_quasi_centres} and \autoref{lem_E(x)_is_bdd}, we have 
    \begin{align*}
        |[fa,fb,fc,fd]| & \approx_{C_5} d_{Y^h}(\pi_{Y^h}(fa,fb,fc), \pi_{Y^h}(fa,fc,fd)) \\
        & \approx_{C_5 + 2D} d_{Y^h}(\Phi f (x), \Phi f(y)).
    \end{align*}
    Therefore,    
    \begin{align*}
        d_{Y^h}(\Phi f (x), \Phi f(y)) 
        & \le |[fa,fb,fc,fd]| + 2D + C_5 \\
        & \le A |[a,b,c,d]| + B + 2D + C_5 \\
        & \le A d_{X^h}(x,y) + A (2R_2 + C_5) + (B + 2D + C_5).
    \end{align*}
    Similarly,
    \begin{align*}
        d_{X^h}(x,y) 
        & \le |[a,b,c,d]| + 2R_2 + C_5 \\
        & \le A |[fa,fb,fc,fd]| + B + 2R_2 + C_5 \\
        & \le A d_{Y^h}(\Phi f (x), \Phi f(y)) + A (2D + C_5) + (B + 2R_2 + C_5).
    \end{align*}     
    Setting $\lambda_1 = A$ and $\epsilon_1' = \max\{A (2R_2 + C_5) + (B + 2D + C_5), A (2D + C_5) + (B + 2R_2 + C_5)\}$, we see that $\Phi f$ is a $(\lambda_1,\epsilon_1')$–quasi-isometric embedding. 
    
    To show coarse surjectivity of $\Phi f$, let $y\in {Y^h}$ be arbitrary. By \autoref{cor_Quasi_centres_are_Coarsely_dense}, there exists a constant $R_1 = R_1(\delta) \ge 0$ and a triple $(a',b',c') \in \partial^3 {Y^h}$ such that $$d_{Y^h}(y,\pi_{Y^h}(a',b',c'))\le R_1.$$ 

    Choose $x = \pi_{X^h}(f^{-1}(a'),f^{-1}(b'),f^{-1}(c')) \in X^h$. Then both $\Phi f(x)$ and $\pi_{Y^h}(a',b',c')$ lie in $E(x)$, which by \autoref{lem_E(x)_is_bdd} has diameter at most $D$. Thus, $d_{Y^h}(\Phi f(x),\pi_{Y^h}(a',b',c')) \le D,$ and consequently, $$d_{Y^h}(\Phi f(x),y) \le D + R_1.$$
    Setting $\epsilon_1'' = D + R_1$, we conclude that $\Phi f$ is $\epsilon_1''$-surjective.

    Finally, setting $\epsilon_1 = \max\{\epsilon_1',\epsilon_1''\}$, we conclude that $\Phi f$ is a $(\lambda_1,\epsilon_1)$-quasi-isometry. This completes the proof.
\end{proof}

\begin{rem} 
    Given a quasi-M\"{o}bius homeomorphism $f:\partial X^h\to \partial Y^h$, we have, by \autoref{thm_QM_implies_qi_on_cusped_spaces}, obtained a  $(\lambda_1,\epsilon_1)$-quasi-isometry $\Phi f :X^h\to Y^h$. Since $f^{-1}:\partial Y^h\to \partial X^h$ is also $\psi$-quasi-M\"{o}bius, the same argument applied to $f^{-1}$ yields a $(\lambda_1,\epsilon_1)$-quasi-isometry $\Phi f^{-1}: Y^h \to X^h$. 
    Now, for each $x \in X^h$, we can write $\Phi f(x) = \pi_{Y^h}(fa,fb,fc)$ for some $(a,b,c) \in \partial^3 X^h$ with $d_{X^h}(\pi_{X^h}(a,b,c),x) \le R$. By \autoref{defn_defining_set_E(.)} of the set $E(\cdot)$, the set $E(\Phi f(x))$ is uniformly bounded in $X^h$ and contains both $\Phi f^{-1}(\Phi f(x))$ and $\pi_{X^h}(a,b,c)$. Hence, $x$ and $\Phi f^{-1}\circ\Phi f (x)$ lie within a uniformly bounded distance, with the bound depending only on $\delta$ and $\psi$. Similarly, for every $y \in Y^h$, $\Phi f\circ\Phi f^{-1} (y)$ lies within a uniformly bounded distance of $y$. By adjusting the constants, we may assume that $$d_{X^h}(\Phi f^{-1} \circ \Phi f (x),x) \le \epsilon_1 \quad \text{and} \quad d_{Y^h}(\Phi f \circ \Phi f^{-1} (y), y) \le \epsilon_1.$$
    Therefore, $\Phi f^{-1}$ is a quasi-isometry inverse of $\Phi f$.
    
\end{rem}

In order to make the quasi-isometry $\Phi f$ coarsely preserves cusps, we now impose the additional assumption that $f$ is relative quasi-M\"{o}bius.

\begin{prop} \label{thm_rel_QM_implies_cusp_preserving_qi_on_cusped_spaces} 
    If the homeomorphism $f \colon \partial X^h \to \partial Y^h$ is relative $\psi$-quasi-M\"obius, 
then there exists a constant $K = K(\delta,\psi) \ge 0$ such that the induced quasi-isometry 
$\Phi f \colon X^h \to Y^h$ is $K$-coarsely cusp-preserving;  that is, $\Phi f$ satisfies 
conditions~(i) and~(ii) of \autoref{defn_cusp_preserving}.
\end{prop}

\begin{proof}
    By the definition of relative quasi-M\"{o}bius maps (\autoref{defn_relative_quasi-Mobius}), the map $f$ is, in particular, quasi-M\"{o}bius. Hence, by \autoref{lem_linearity_for_QM}, we may assume without loss of generality that the distortion function $\psi \colon [0,\infty) \to [0,\infty)$ is linear; that is, \(\psi(t) = At + B\) for some constants $A \ge 1$ and $B \ge 0$. Therefore, for any triple $(a,b,c) \in \partial^3 X^h$ with $c$ a parabolic endpoint, we have
    \[
    |[fa,fb,fc]| \approx_{A,B} |[a,b,c]|.
    \]

    Let $H \in \mathcal{H}_{G_1}$ be a horosphere in $X^h$, and let $c_H \in \partial X^h$ denote the parabolic endpoint corresponding to $H^h$. Let $h \in H$ be an arbitrary point. Then there exists a geodesic ray $[a,c_H]$ passing through $h$, where $a \in \partial X^h$. By \autoref{prop_a_point_on_geodesic_vs_QP}, there exist a point $b \in \partial X^h$ and a constant $R_1 = R_1(\delta) \ge 0$ such that
    \[
    d_{X^h}(\pi_{X^h}(a,b,c_H),h) \le R_1.
    \]

    Let $H'\in \mathcal{H}_{G_2}$ be the horosphere in $Y^h$ corresponding to the parabolic endpoint $f(c_H)$, and set $a' = f(a)$, $b' = f(b)$, and $c_H' = f(c_H)$. 
    By \autoref{prop_relative_cross_ratio_vs_quasi_centres}, there exists a constant $C_7 = C_7(\delta) \ge 0$ such that
    \[
    |[a,b,c_H]| \approx_{C_7} d_{X^h}(\pi_{X^h}(a,b,c_H), H) \quad \text{and} \quad |[a',b',c_H']| \approx_{C_7} d_{Y^h}(\pi_{Y^h}(a',b',c_H'), H').
    \]
    Since quasi-isometries coarsely map quasi-centres to quasi-centres, by \autoref{lem_qi_coarsely_preserves_the_quasi-centres}, there exists a constant $C_3 = C_3(\delta,\lambda_1,\epsilon_1) \ge 0$ such that
    \[
    d_{Y^h}\!\left(\Phi f(\pi_{X^h}(a,b,c_H)), \pi_{Y^h}(a',b',c_H')\right) \le C_3.
    \]
    Since $\Phi f$ is a $(\lambda_1,\epsilon_1)$-quasi-isometry, using the above inequalities we obtain
    \begin{align*}
        d_{Y^h}(\Phi f(h), H') 
        & \le d_{Y^h}(\Phi f(h), \pi_{Y^h}(a',b',c_H')) + d_{Y^h}(\pi_{Y^h}(a',b',c_H'), H') \\ 
        & \le d_{Y^h}(\Phi f(h), \Phi f(\pi_{X^h}(a,b,c_H))) + C_3 + |[a',b',c_H']| + C_7 \\ 
        & \le (\lambda_1 R_1 + \epsilon_1) + C_3 + (A|[a,b,c_H]| + B) + C_7 \\ 
        & \le A\bigl(d_{X^h}(\pi_{X^h}(a,b,c_H), H) + C_7\bigr) + B + (\lambda_1 R_1 + \epsilon_1) + C_3 + C_7 \\ 
        & \le AR_1 + (A+1)C_7 + B + (\lambda_1 R_1 + \epsilon_1) + C_3. 
    \end{align*}
    
    Therefore, setting \(K = (A+\lambda_1)R_1 + (A+1)C_7 + C_3 + B + \epsilon_1\), we conclude that for each $H \in \mathcal{H}_{G_1}$ there exists $H' \in \mathcal{H}_{G_2}$ such that $\Phi f(H)$ is contained in the $K$-neighbourhood of $H'$. 

    On the other hand, since $\Phi f^{-1}$ is also a $(\lambda_1,\epsilon_1)$-quasi-isometry, it is a quasi-isometric inverse of $\Phi f$, induced by the $\psi$-quasi-M\"{o}bius homeomorphism $f^{-1}$. A similar argument applied to $\Phi f^{-1}$ shows that for each $H' \in \mathcal{H}_{G_2}$ there exists $H \in \mathcal{H}_{G_1}$ such that $\Phi f^{-1}(H')$ lies within the $K$-neighbourhood of $H$. Hence, $\Phi f$ (and similarly $\Phi f^{-1}$) is a $K$-coarsely cusp-preserving quasi-isometry. This completes the proof.  
\end{proof} 

Since $\Phi f \colon X^h \to Y^h$ is a quasi-isometry, it induces a boundary homeomorphism $\partial \Phi f \colon \partial X^h \to \partial Y^h$. Because $\Phi f$ coarsely preserves cusps, the induced map $\partial \Phi f$ agrees with $f$ on the parabolic endpoints of $\partial X^h$. Since these parabolic endpoints are dense in $\partial X^h$ (see \cite{Bowditch}, Section~9), it follows that $\partial \Phi f = f$ on the entire boundary.
However, in the absence of parabolic subgroups, as in the case of (non-elementary) hyperbolic groups, this argument no longer applies. In the next proposition, we give a general proof that $\partial \Phi f = f$, valid for both hyperbolic and relatively hyperbolic groups.

\begin{prop}\label{prop_boundaries_coincide_1}
    The boundary map $\partial\Phi f$, extension of $\Phi f$, coincides with the given homeomorphism $f$; that is, $\partial\Phi f=f$.
\end{prop}

\begin{proof}
    Let $z \in \partial X^h$ be an arbitrary boundary point. Fix a basepoint \(x_0 = \pi_{X^h}(a_0,b_0,c_0) \in X^h,\) where $(a_0,b_0,c_0) \in \partial^3 X^h$. Then there exists a sequence $\{g_i\} \subset G$ such that $g_i x_0 \to z$ in $X^h \cup \partial X^h$. To prove that $\partial \Phi f (z) = f(z)$, it suffices to show that the sequence $\{\Phi f(g_i x_0)\}$ converges to $f(z)$. 

    For each $i \ge 1$, let \(x_i = g_i x_0, a_i = g_i a_0, b_i = g_i b_0, c_i = g_i c_0\). By \autoref{lem_qi_coarsely_preserves_the_quasi-centres}, there exists a constant $C_3 = C_3(\delta,1,0) \ge 0$ such that 
    \[
    d_{X^h}\bigl(x_i,\pi_{X^h}(a_i,b_i,c_i)\bigr) \le C_3 \quad \text{for all } i.
    \]
    Since $\partial X^h$ is compact, after passing to a subsequence we may assume \((a_i,b_i,c_i) \to (a,b,c) \in \partial^3 X^h\).
    Because $x_i \to z$ in $X^h \cup \partial X^h$, at least two of the points $a,b,c$ must coincide with $z$. Otherwise, by \autoref{lem_quasi-continuity_of_quasi-centres}, there would exist $N \ge 1$ such that for all $i \ge N$, the quasi-centre $\pi_{X^h}(a_i,b_i,c_i)$ (and hence $x_i$) lies within a uniformly bounded distance of $\pi_{X^h}(a,b,c)$. This would imply that $z$ lies at finite distance from $\pi_{X^h}(a,b,c)$ in $X^h$, contradicting the fact that $z \in \partial X^h$. Without loss of generality, we may therefore assume that $a=b=z$.

    Since $f$ is a homeomorphism of $\partial X^h$, we have 
    \((fa_i,fb_i,fc_i) \to (fa,fb,fc) = (f(z),f(z),f(c))\), and hence \(\pi_{Y^h}(fa_i,fb_i,fc_i) \to f(z)\). Moreover, by \autoref{lem_E(x)_is_bdd}, there exists $D \ge 0$ such that 
    \[
    d_{Y^h}\bigl(\Phi f(x_i),\pi_{Y^h}(fa_i,fb_i,fc_i)\bigr) \le D \quad \text{for all } i.
    \] 
    Therefore, the sequence $\{\Phi f(x_i)\}$ converges to $f(z)$, and we conclude that
    \[
    \partial \Phi f (z) = f(z).
    \]
    This completes the proof.
\end{proof}

The maps $\Phi f$ and $\Phi f^{-1} $ will naturally induce  maps $\varphi_f: G_1\to G_2$ and $\varphi_{f^{-1}} :G_2\to G_1$, respectively, as follows: Since $\Phi f$ (respectively $\Phi f^{-1}$) is $K$-coarsely cusp-preserving, $\Phi f(G_1)$ (respectively $\Phi f^{-1}(G_2)$) is contained in the $K$-neighbourhood of $G_2$ (respectively $G_1$).
We now define $$\varphi_f \coloneqq Pr_{G_2} \circ \Phi f \restriction_{G_1} \quad \text{and} \quad \varphi_{f^{-1}} \coloneqq Pr_{G_1} \circ \Phi f^{-1} \restriction_{G_2},$$ where $Pr_{G_i} : N_{K}(G_i) \to G_i$ denotes a nearest point projection maps defined on the $K$-neighbourhoods $N_{K}(G_i)$, for $i=1,2$. We now show that $\varphi_f$ is the required coarsely cusp-preserving quasi-isometry satisfying the conclusion of \autoref{thm_main_theorem_2}.




\begin{thm}\label{thm_cusp_preserving_qi_implies_ambient_qi}
    There exist constants $\lambda = \lambda(\delta,\psi) \ge 1$, $\epsilon = \epsilon(\delta,\psi) \ge 0$ and $K' = K'(\delta,\psi) \ge 0$ such that $\varphi_f: (G_1,d_{G_1}) \to (G_2,d_{G_2})$ is a $K'$-coarsely cusp-preserving $(\lambda,\epsilon)$-quasi-isometry with quasi-isometry inverse $\varphi_{f^{-1}}$.
\end{thm}

\begin{proof}   
    From the above construction of $\varphi_f$ and $\varphi_{f^{-1}}$, it follows that for every $x\in G_1$ and $y \in G_2$ 
    $$d_{Y^h}(\Phi f(x), \varphi_f(x)) \le K \quad \text{and} \quad d_{X^h}(\Phi f^{-1}(y), \varphi_{f^{-1}}(y)) \le K.$$ 
   


   First, we show that $\varphi_f$ and $\varphi_{f^{-1}}$ are coarse inverses of each other. Let $x \in G_1$. Then
    \begin{align*}
         d_{Y^h}(\Phi f (x), \varphi_f (x)) \le K 
        \implies & d_{X^h}(\Phi f^{-1} \circ \Phi f (x), \Phi f^{-1} \circ \varphi_f (x)) \le \lambda_1 K + \epsilon_1 \\
        \implies & d_{X^h}(x, \Phi f^{-1} \circ \varphi_f (x)) \le (\lambda_1 K + \epsilon_1) + \epsilon_1 \\
        \implies & d_{X^h}(x, \varphi_{f^{-1}} \circ \varphi_f (x)) \le (\lambda_1 K + 2\epsilon_1) + K
    \end{align*}
    Since $G_1$, is properly embedded in $X^h$, using \autoref{lem_Uniformly_Properly_embedded_Lemma}, we have a constant $K_2' = K_2'(\lambda_1 K + 2\epsilon_1 + K) \ge 0$ such that for every $x \in G_1$, we have
    $d_{G_1}(x, \varphi_{f^{-1}} \circ \varphi_f (x)) \le K_2'.$     
    In a similar way, we can show that for every $y \in G_2$, $d_{G_2}(y, \varphi_f \circ \varphi_{f^{-1}} (y)) \le K_2'.$
    Hence, $\varphi_f$ and $\varphi_{f^{-1}}$ are coarse inverses of each other. 

    Next, we show that $\varphi_f$ is a quasi-isometry. Let $x,z \in G_1$ satisfy $d_{G_1}(x,z) = 1$. Then there exist $H_1, H_2 \in \mathcal{H}_{G_1}$ such that $x \in H_1$ and $z \in H_2$. Since $\Phi f$ is a $(\lambda_1,\epsilon_1)$-quasi-isometry, we have
    $$d_{X^h}(x,z) \le 1 \implies d_{Y^h}(\Phi f (x), \Phi f (z)) \le \lambda_1 + \epsilon_1  \implies  d_{Y^h}(\varphi_f(x), \varphi_f(z)) \le \lambda_1 + \epsilon_1 + 2K.$$ 

    Since $Y$, and hence $G_2$, is properly embedded in $Y^h$, it follows from \autoref{lem_Uniformly_Properly_embedded_Lemma} that there exists a constant $K_2'' = K_2''(\lambda_1 + \epsilon_1 + 2K) \ge 0$ such that $$d_{G_2}(\varphi_f(x), \varphi_f (z)) \le K_2''.$$

    Now let $x,z \in G_1$ be arbitrary. Subdivide the geodesic segment $[x,z]_{G_1}$ in $G_1$ into points $x=x_0,x_1,\dots,x_{n}=z$, with $d_{G_1}(x_{i-1}, x_i) = 1$ for $1 \le i \le n$. Then, 
    \begin{align*}
        d_{G_2}(\varphi_f(x), \varphi_f(z)) & \le \sum_{i=1}^{n} d_{G_2}(\varphi_f (x_{i-1}), \varphi_f(x_i))  \le n K_2''  = K_2'' d_{G_1}(x,z).
    \end{align*} 
     Similarly, using the inverse quasi-isometry $\Phi f^{-1}$, for all $y,w \in G_2$, we have $$d_{G_1}(\varphi_{f^{-1}} (y), \varphi_{f^{-1}} (w)) \le K_2'' d_{G_2}(y,w).$$
    Therefore, for all $x,z \in G_1$,
    \begin{align*}
        & d_{G_1}(\varphi_{f^{-1}} \circ \varphi_f(x), \varphi_{f^{-1}} \circ \varphi_f(z)) \le K_2'' d_{G_2}(\varphi_f(x),\varphi_f(z))\\
        \implies &  d_{G_1}(x,z) \le K_2'' d_{G_2}(\varphi_f(x),\varphi_f(z)) + 2 K_2'.
    \end{align*} 

    Also, for all $y \in G_2$, $$\varphi_{f^{-1}}(y) \in G_1 \quad \text{and} \quad d_{G_2}(y,\varphi_f \circ \varphi_{f^{-1}}(y)) \le K_2'.$$ 

    Hence, setting $\lambda = K_2''$  and $\epsilon = 2K_2'$, we conclude that $\varphi_f$ (similarly $\varphi_{f^{-1}}$) is a $(\lambda,\epsilon)$-quasi-isometry. 
    

   Finally, we show that $\varphi_f$ is a coarsely cusp-preserving map. 
    For every $H \in \mathcal{H}_{G_1}$, there exists $H' \in \mathcal{H}_{G_2}$ such that $\Phi f(H) \subseteq N_K(H')$, which implies $\varphi_f (H) \subseteq N_{2K}(H')$.   
    Conversely, for all $H '\in \mathcal{H}_{G_2}$, there exists $H \in \mathcal{H}_{G_1}$ such that $\Phi f^{-1}(H') \subseteq N_K(H)$, which implies $\varphi_{f^{-1}} (H') \subseteq N_{2K}(H)$.  
    Hence, setting $K' = 2K$, we conclude that $\varphi_f$ (and similarly $\varphi_{f^{-1}}$) preserves cusps $K'$-coarsely. This completes the proof.    
\end{proof}

\begin{prop}\label{prop_boundaries_coincide}
    Let \(f:\partial X^h\to\partial Y^h\) be a relative $\psi$-quasi-M\"{o}bius map, and \(\varphi_f\) be the quasi-isometry as in \autoref{thm_cusp_preserving_qi_implies_ambient_qi}. Let $\varphi_f^h: X^h \to Y^h$ denote its extension as defined in \autoref{thm_cusp_prev_qi_implies_homeo}. Then \(\partial\varphi_f^h = f\). 
\end{prop}

\begin{proof}
    By \autoref{thm_cusp_prev_qi_implies_homeo}, the map \(\varphi_f^h \colon X^h \to Y^h\) is a coarsely cusp-preserving quasi-isometry. Let \(\gamma \colon [0,\infty) \to X^h\) be a vertical geodesic ray with \(\gamma(0) \in H\). Then the image \(\varphi_f^h(\gamma)\) is eventually a vertical geodesic ray, while $\Phi f(\gamma)$ is a quasigeodesic ray. Moreover, these two rays remain within a uniformly bounded Hausdorff distance of one another. It follows that for every \(x \in X^h\), the distance in \(Y^h\) between \(\varphi_f^h(x)\) and \(\Phi f(x)\) is uniformly bounded. Consequently, the induced boundary map \(\partial \varphi_f^h \colon \partial X^h \to \partial Y^h\), extension of \(\varphi_f^h\), coincides with \(\partial \Phi f\), and hence with the given homeomorphism \(f\).
\end{proof}


\section{Linear Distortion of Exit Points} \label{sec_linear_distortion_and_rel_hyp}

This section is devoted to presenting another criterion related to the previous results using exit points of geodesics from horoballs.



\subsection{Exit Points and their Linear Distortion} \label{subsec_exit_points}

Let $(G,\mathcal{H}_G)$ be a $\delta$-relatively hyperbolic group with Cayley graph $X$ for some $\delta \ge 0$. Let $X^h$ denote the cusped space and $\partial X^h$ denote the Bowditch boundary of $X$, or equivalently, of $(G,\mathcal{H}_G)$.

Consider the set $\Theta(G) = \{(a,b) \mid a,b \in \partial X^h, \text{ $a$ is parabolic, and } a \ne b\}$. Let $(a,b) \in \Theta(G)$, and let $H_a \in \mathcal{H}_{G}$ denote the horosphere associated with the parabolic endpoint $a$. Given any geodesic $\gamma$ in $X^h$ from $a$ to $b$, there exists a unique point $e_{ab} \in \gamma \cap H_a$, namely the last point of intersection of $\gamma$ with $H_a$, such that the subray of $\gamma$ from $e_{ab}$ to $b$ lies entirely outside the horoball $H_a^h$. This point $e_{ab}$ is called an \emph{exit point} of $(a,b)$ from $H_a$. The set of all exit points on $H_a$ for $(a,b)$ is denoted by $\mathcal{E}(a,b)$.
By the $\delta$-hyperbolicity of $X^h$, any two geodesics joining $a$ and $b$ lie within a $2\delta$-neighbourhood of each other. Thus, the set $\mathcal{E}(a,b)$ is uniformly bounded in $X^h$. By applying \autoref{lem_Uniformly_Properly_embedded_Lemma}, the set $\mathcal{E}(a,b)$ is also uniformly bounded with respect to the word metric $d_G$ on $G$.

\begin{prop}\label{prop_exit_points_are_bdd}
    There exists a constant $M_1 = M_1(\delta) \ge 0$ such that for every pair $(a,b) \in \Theta(G)$, the set $\mathcal{E}(a,b)$ is non-empty and has diameter at most $M_1$ in $(G,d_G)$.
\end{prop} 

We define a map $\vartheta_{G}: \Theta(G)\to G$ by assigning to each pair $(a,b) \in \Theta(G)$ a point in $\mathcal{E}(a,b)$. Any two such choices differ by at most $M_1$ by \autoref{prop_exit_points_are_bdd}.
Since the cosets of parabolic subgroups cover the group \(G\), for every \(g \in G\) there exists a parabolic subgroup \(H_i\) such that \(g \in gH_i\). Let \(a \in \partial X^h\) be the parabolic endpoint corresponding to the coset \(gH_i\). Choose a point \(b' \in \partial X^h \setminus \{a\}\), and let \(g' \in gH_i\) denote the exit point of the geodesic \([a,b']\) from the horoball \((gH_i)^h\). Set \(h = g'^{-1}g \in H_i\). Then \(gg'^{-1} = ghg^{-1} \in gH_i g^{-1}\). The left translation \(L_{gg'^{-1}}\) by \(gg'^{-1}\) on \(X\) induces an isometry \(L^h_{gg'^{-1}}\) of \(X^h\) satisfying \( L^h_{gg'^{-1}}(g') = g\) and \(\partial L^h_{gg'^{-1}}(a) = a \). Let \([a,b]\) be the bi-infinite geodesic obtained by applying \(L^h_{gg'^{-1}}\) to \([a,b']\). Then the exit point of \([a,b]\) from the horoball \((gH_i)^h\) is precisely \(g\). Therefore, as an immediate consequence of \autoref{prop_exit_points_are_bdd}, the map $\vartheta_G$ is coarsely surjective. 

\begin{lem}\label{lem_exit_points_are_coarsely_onto}
    For every $g \in G$, there exists a pair $(a,b) \in \Theta(G)$ with $d_{G}(\vartheta_{G}(a,b),g)\le M_1,$ where $M_1 = M_1(\delta) \ge 0$ is the constant considered from \autoref{prop_exit_points_are_bdd}. 
\end{lem}


\begin{defn}[Linear distortion of exit points] \label{defn_linear_distortion}
    Let $(G_1,\mathcal{H}_{G_1})$ and $(G_2,\mathcal{H}_{G_2})$ be two $\delta$-relatively hyperbolic groups with respective Cayley graphs $X$ and $Y$, where $\delta \ge 0$. A homeomorphism $f:\partial X^h \to \partial Y^h$ is said to linearly distort exit points if:
    \begin{enumerate}[$(i)$]
        \item The maps $f$ and $f^{-1}$ send parabolic endpoints to parabolic.

        \item There exist constants $A \ge 1$ and $B \ge 0$ such that for all $(a,b),(c,d) \in \Theta(G_1)$,  
        $$d_{G_2}(\vartheta_{G_2}(fa,fb),\vartheta_{G_2}(fc,fd)) \approxeq_{A,B} d_{G_1}(\vartheta_{G_1}(a,b),\vartheta_{G_1}(c,d)).$$
    \end{enumerate}    
    We say $A$ and $B$ are the constants of distortion.
\end{defn}


\subsection{Proof of \autoref{thm_main_theorem_3}} \label{subsec_main_theorem_3}

Given constants $\delta, \epsilon, K \ge 0$ and $\lambda \ge 1$, let $(G_1, \mathcal{H}_{G_1})$ and $(G_2, \mathcal{H}_{G_2})$ be two $\delta$-relatively hyperbolic groups with Cayley graphs $X$ and $Y$, respectively. 
Suppose $\phi : G_1 \to G_2$ is a $K$-coarsely cusp-preserving $(\lambda,\epsilon)$-quasi-isometry. Then $\phi$ induces a $(\lambda',\epsilon')$-quasi-isometry $\phi^h: X^h \to Y^h$ between the cusped spaces $X^h$ and $Y^h$, where $\lambda'\ge 1$ and $\epsilon'\ge 0$ depend on $\delta, \lambda, \epsilon$, and $K$. Moreover, $\phi^h$ induces a homeomorphism $\partial\phi^h: \partial X^h \to \partial Y^h$ between their Bowditch boundaries, such that $\partial \phi^h$ and $(\partial \phi ^h)^{-1}$ send parabolic endpoints to parabolic endpoints. 

\begin{thm}
  The induced homeomorphism $\partial \phi^h: \partial X^h \to \partial Y^h$ linearly distorts exit points, and the constants of the distortion depend only on $\delta,\lambda,\epsilon$, and $K$. 
\end{thm} 

\begin{proof}
    Let $(a,b),(c,d) \in \Theta(G_1)$ be two pairs. Let $a'=\partial\phi^h(a)$, $b'=\partial\phi^h(b)$, $c'=\partial\phi^h(c)$ and $d'=\partial\phi^h(d)$. Let $H_a$, $H_c$, $H_{a'}$ and $H_{c'}$ be the horospheres associated with the parabolic endpoints $a,c,a'$ and $c'$, respectively. 
    
    Let $\vartheta_{G_1}(a,b) = x_1 \in [a,b] \cap H_a$ and $\vartheta_{G_1}(c,d) = x_2 \in [c,d] \cap H_c$. Since $\phi$ is $K$-coarsely cusp-preserving, the points $\phi(x_1)$ and $\phi(x_2)$ lie within the $K$-neighbourhood of $H_{a'}$ and $H_{c'}$, respectively. Since $\phi^h$ is $(\lambda',\epsilon')$-quasi-isometry, it sends geodesics to $(\lambda',\epsilon')$-quasigeodesics. By \autoref{Prop_Stability_of_quasigeodesics}, there exists $K_0 = K_0(\delta,\lambda',\epsilon') \ge 0$ and two points $y_1'\in [a',b']$ and $y_2'\in [c',d']$ such that $d_{Y^h}(\phi^h(x_1),y_1') \le K_0$ and $d_{Y^h}(\phi^h(x_2),y_2') \le K_0.$
    
    \begin{figure}[H]
        \centering
            \includegraphics[scale=.28]{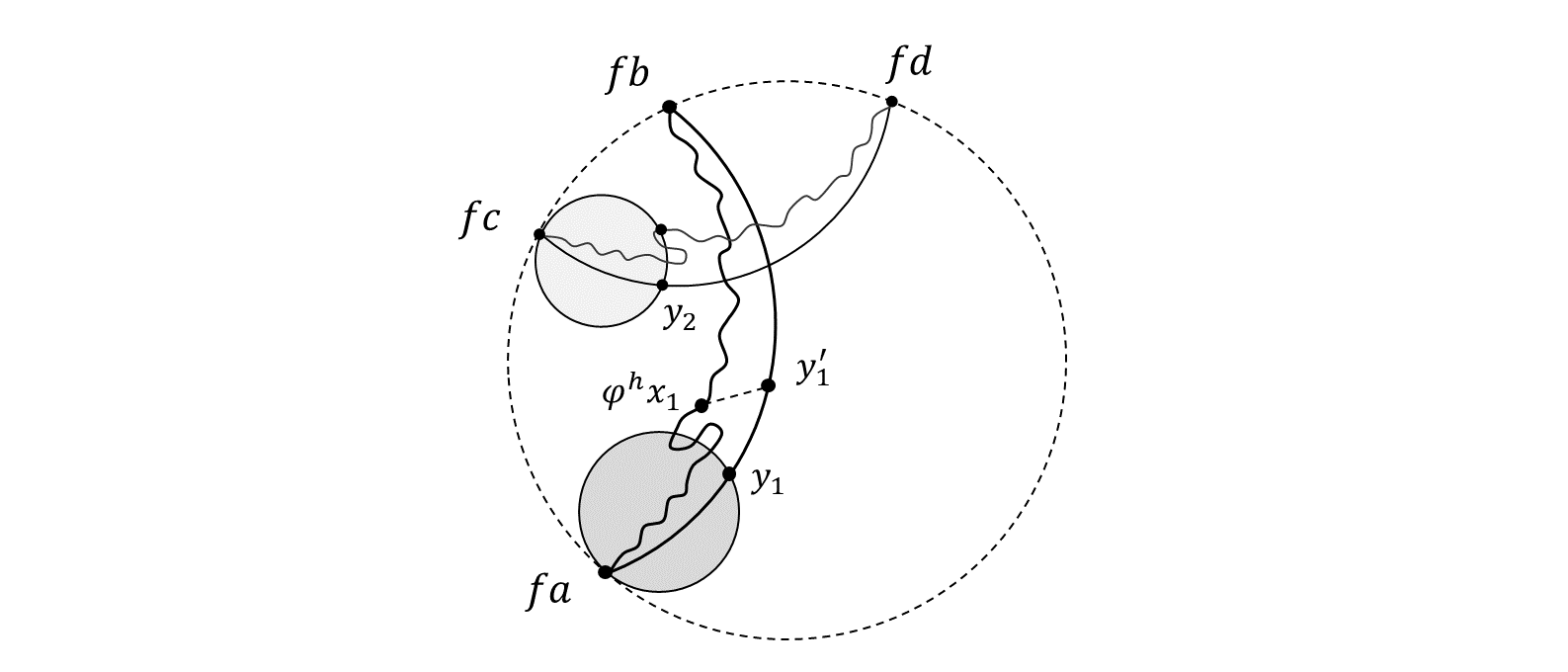}
        \caption{}
    \end{figure}
    
    Let $y_1$ and $y_2$ be the exit points of $[a',b']$ and $[c',d']$, respectively. By applying \autoref{prop_visually_bounded}, we have 
    \begin{align*}
        d_{Y^h}(y_1',y_1) 
        & \le d_{Y^h}(y_1',H_{a'}) + K_3 \\
        & \le d_{Y^h}(y_1',\phi^h(x_1)) +  d_{Y^h}(\phi^h(x_1),H_{a'}) + K_3 \\
        & \le K_0 + K + K_3.
    \end{align*} 
    Similarly, 
    $$d_{Y^h}(y_2',y_2) \le K_0 + K + K_3.$$
    Since $G_2$ is properly embedded in $Y^h$, by \autoref{lem_Uniformly_Properly_embedded_Lemma}, there exists a constant $K_2 = K_2(2K_0 + K + K_3) \ge 0$ such that 
    $$d_{G_2}(\phi^h(x_1),y_1) \le K_2 \quad \text{and} \quad d_{G_2}(\phi^h(x_2),y_2) \le K_2.$$
    By applying \autoref{prop_exit_points_are_bdd} and the quasi-isometry property of $\phi^h$, we have 
    \begin{align*}
        d_{G_2}(\vartheta_{G_2}(a',b'),\vartheta_{G_2}(c',d')) 
        & \le d_{G_2}(y_1,y_2) + 2M_1 \\      
        &\le d_{G_2}(\phi^h(x_1),\phi^h(x_2)) + 2K_2 + 2M_1\\
        &\le \lambda' d_{G_1}(x_1,x_2) + \epsilon' + 2K_2 + 2M_1.
    \end{align*} 
    Setting $A_1 = \lambda'$ and $B_1 = \epsilon' + 2K_2 + 2M_1$, we conclude that
    $$d_{G_2}(\vartheta_{G_2}(a',b'),\vartheta_{G_2}(c',d')) \precapprox_{A_1,B_1} d_{G_1}(\vartheta_{G_1}(a,b),\vartheta_{G_1}(c,d)).$$
    Since the constants $\lambda'$, $\epsilon'$, $K_2$ and $M_1$ depend on $\delta$, $\epsilon$, $\lambda$ and $K$, the constants $A_1$ and $B_1$ depend only on $\delta$, $\epsilon$, $\lambda$ and $K$.
    
    A similar argument applied to a quasi-isometry inverse $(\phi^h)^{-1} : Y^h \to X^h$ yields constants $A_2 \ge 1$ and $B_2 \ge 0$ depending on \(\delta,\lambda,\epsilon,\) and $K$ such that 
    $$d_{G_1}(\vartheta_{G_1}(a,b),\vartheta_{G_1}(c,d)) \precapprox_{A_2,B_2} d_{G_2}(\vartheta_{G_2}(a',b'),\vartheta_{G_2}(c',d')).$$
    
    Finally, by setting $A = \max\{A_1,A_2\}$ and $B = \max\{B_1,B_2\}$, we conclude that $$d_{G_2}(\vartheta_{G_2}(a',b'),\vartheta_{G_2}(c',d')) \approxeq_{A,B} d_{G_1}(\vartheta_{G_1}(a,b),\vartheta_{G_1}(c,d)).$$
    Hence, the induced boundary homeomorphism $\partial\phi^h$ linearly distorts exit points.
\end{proof} 


\subsection{Proof of \autoref{thm_main_theorem_4}}\label{subsec_proof_of_thm_4}


Let $(G_1,\mathcal{H}_{G_1})$ and $(G_2,\mathcal{H}_{G_2})$ be two relatively hyperbolic groups with Cayley graphs $X$ and $Y$, respectively.
Let $f:\partial X^h \to \partial Y^h$ be a homeomorphism between the Bowditch boundaries.

\begin{defn}
    For each $x \in G_1$, define the set
    $$F(x) \coloneqq \vartheta_{G_2}\circ f (\vartheta_{G_1}^{-1}(\overline{B}(x,M_1))),$$
    where $M_1$ is the constant considered from \autoref{lem_exit_points_are_coarsely_onto}. Here, $\overline{B}(x,M_1)$ denotes the closed ball in $(G_1,d_{G_1})$ of radius $M_1$ centred at $x$, and $\vartheta_{G_1}^{-1}(\overline{B}(x,M_1))$ denote the preimage of $\overline{B}(x,M_1)$ in $\Theta(G_1)$ under $\vartheta_{G_1}$.
\end{defn}

\begin{prop} \label{Prop_Boundedness_of_F(x)_2}
    If $f:\partial X^h \to \partial Y^h$ linearly distorts exit points with distortion constants $A \ge 1$ and $B \ge 0$, then there exists a constant $D_2 = D_2(\delta,A,B) \ge 0$ such that for every $x \in G_1$, the set $F(x)$ is nonempty, bounded, and of diameter at most $D_2$ in $(G_2,d_{G_2})$. 
\end{prop} 

\begin{proof}
    The non-emptiness of $F(x)$ follows from \autoref{lem_exit_points_are_coarsely_onto}. For boundedness, let $y_1,y_2 \in F(x)$ be arbitrary. There exist pairs $(a,b),(c,d)\in \Theta(G_1)$ with $d_{G_1}(x,\vartheta_{G_1}(a,b))\le M_1$ and $d_{G_1}(x,\vartheta_{G_1}(c,d))\le M_1$ such that $y_1 = \vartheta_{G_2}(fa,fb)$ and $y_2 = \vartheta_{G_2}(fc,fd)$. Since $f$ linearly distorts exit points, we have 
    \begin{align*}
        d_{G_2}(y_1,y_2) 
        & = d_{G_2}(\vartheta_{G_2}(fa,fb),\vartheta_{G_2} (fc,fd)) \\
        & \le A d_{G_1}(\vartheta_{G_1}(a,b),\vartheta_{G_1}(c,d)) + B \\
        & \le 2AM_1 + B.
    \end{align*}
    Hence, $F(x)$ has diameter at most \(D_2 := 2AM_1 + B\) in $(G_2,d_{G_2})$, which completes the proof.
\end{proof}

We now define a map $\phi_f: G_1 \to G_2$ by assigning to each $x \in G_1$ a point in $F(x)$. By \autoref{Prop_Boundedness_of_F(x)_2}, any two such maps differ by $D_2$ whenever $f$ linearly distorts the exit points. 

\begin{thm}
    If the homeomorphism $f: \partial X^h \to \partial Y^h$ linearly distorts exit points with distortion constants $A \ge 1$ and $B \ge 0$, then there exist constants $\lambda_3 = \lambda_3(\delta, A, B) \ge 1$, $\epsilon_3 = \epsilon_3(\delta, A, B) \ge 0$ and $K'' = K''(\delta, A, B) \ge 0$ such that the induced map $\phi_f : G_1 \to G_2$ is a $(\lambda_3,\epsilon_3)$-quasi-isometry that is $K''$-coarsely cusp-preserving. 
\end{thm}

\begin{proof}
    First, we show that $\phi_f$ is a quasi-isometry. For any two elements $x_1,x_2 \in G_1$, by \autoref{lem_exit_points_are_coarsely_onto}, there exist pairs $(a,b),(c,d)\in \Theta(G_1)$ such that $d_{G_1}(x_1,\vartheta_{G_1}(a,b))\le M_1$ and $d_{G_1}(x_2,\vartheta_{G_1}(c,d))\le M_1$. Hence, 
    $$d_{G_1}(\vartheta_{G_1}(a,b),\vartheta_{G_1}(c,d)) \approx_{2M_1} d_{G_1}(x_1,x_2).$$
    Since $\vartheta_{G_2}(f a,f b)\in F(x_1)$ and $\vartheta_{G_2}(f c,f d)\in F(x_2)$, by applying \autoref{Prop_Boundedness_of_F(x)_2}, we have 
    $$d_{G_2}(\phi_f(x_1),\phi_f(x_2)) \approx_{2D_2} d_{G_2}(\vartheta_{G_2}(f a,f b),\vartheta_{G_2}(f c,f d)).$$
    Since $f$ linearly distorts exit points, we obtain
    \begin{align*}
        d_{G_2}(\phi_f(x_1),\phi_f(x_2)) 
        & \le d_{G_2}(\vartheta_{G_2}(f a,f b),\vartheta_{G_2}(f c,f d)) + 2D_2 \\
        & \le A d_{G_1}(\vartheta_{G_1}(a,b),\vartheta_{G_1}(c,d)) + B + 2D_2 \\
        & \le A d_{G_1}(x_1,x_2) + 2A M_1 + B + 2D_2.
    \end{align*}
    Similarly, 
    \begin{align*}
        d_{G_1}(x_1,x_2) 
        & \le d_{G_1}(\vartheta_{G_1}( a, b),\vartheta_{G_1}( c, d)) + 2M_1\\
        &\le A d_{G_2}(\vartheta_{G_2}(f a,f b),\vartheta_{G_2}(f c,f d)) + B + 2M_1\\
        &\le A d_{G_2}(\phi_f(x_1),\phi_f(x_2)) + 2AD_2 + B + 2M_1.
    \end{align*}    
    Hence, setting $\lambda_3 = A$ and $\epsilon_3' = \max\{ 2A M_1 + B + 2D_2, 2AD_2 + B + 2M_1 \}$, we conclude that $\phi_f$ is a $(\lambda_3,\epsilon_3')$-quasi-isometric embedding. 
    
    Now, for every $y \in G_2$, there exists a pair $(a',b') \in \Theta(G_2)$ such that $d_{G_2}(\vartheta_{G_2}(a',b'),y) \le M_1$. 
    Choose $a,b \in \partial X^h$ with $fa = a'$ and $fb = b'$. Since $f^{-1}$ preserves parabolic endpoints, we have $(a,b)\in \Theta(G_1)$. Let $x = \vartheta_{G_1}( a, b) \in G_1$, which implies $\vartheta_{G_2}(f a,f b)\in F(x)$. Then, by applying \autoref{Prop_Boundedness_of_F(x)_2}, we have $d_{G_2}(\phi_f(x),\vartheta_{G_2}(f a,f b))\le D_2$, and hence $$d_{G_2}(\phi_f(x),y)\le D_2 + M_1.$$ 
    Hence, setting $\epsilon_3'' = D_2 + M_1$, we conclude that $\phi_f$ is $\epsilon_3''$-coarsely surjective.  
    Finally, setting $\epsilon_3 = \max\{\epsilon_3',\epsilon''\}$, we conclude that $\phi_f$ is a $(\lambda_3,\epsilon_3)$-quasi-isometry. 

    Now, we show that $\phi_f$ is coarsely cusp-preserving. Let $H \in \mathcal{H}_{G_1}$ be a horosphere with parabolic endpoint $a_H \in \partial X^h$, and let $x \in H$. There exists a bi-infinite geodesic $[a_H,b]$, with $b \in \partial X^h$, such that $x$ is the exit point of $[a_H,b]$ from $H^h$.     
    Then, the pair $(a_H,b) \in \vartheta_{G_1}^{-1}(\overline{B}(x,M_1))$, and hence $\vartheta_{G_2}(fa,fb) \in F(x) \cap H_{fa}$, where $H_{fa}$ is the horosphere corresponding to the parabolic endpoint $fa$. By \autoref{Prop_Boundedness_of_F(x)_2}, $d_{G_2}(\phi_f(x),\vartheta_{G_2}(fa,fb)) \leq D_2$, and hence $d_{G_2}(\phi_f(x),H_{fa}) \le D_2$. Since $x\in H$ was arbitrary, $$\phi_f(H) \subseteq N_{D_2}(H_{fa}).$$ 

    For the reverse inequality, let \(\phi_f^{-1}\) be a quasi-isometry inverse of \(\phi_f\). Without loss of generality, assume that \(\phi_f^{-1}\) is also a \((\lambda_3,\epsilon_3)\)-quasi-isometry. Let \(H' \in \mathcal{H}_{G_2}\) be a horosphere with parabolic endpoint \(a_{H'} \in \partial Y^h\), and let \(y \in H'\). There exists a bi-infinite geodesic \([a_{H'},b']\), with \(b' \in \partial Y^h\), such that \(y\) is the exit point of \([a_{H'},b']\) from \(H'^h\). Let \(a = f^{-1}(a_{H'})\) and \(b = f^{-1}(b')\), and let \(H_a \in \mathcal{H}_{G_1}\) denote the horosphere with parabolic endpoint \(a\). Let \(x = \vartheta_{G_1}(a,b) \in [a,b]\cap H_a\).
    By \autoref{Prop_Boundedness_of_F(x)_2}, we have $d_{G_2}(\phi_f(x),\vartheta_{G_2}(fa,fb)) \leq D_2$. Since \(y\) is an exit point of \([a_{H'},b']\), by \autoref{prop_exit_points_are_bdd} it follows that  
    \begin{align*}
        d_{G_2}(\phi_f(x), y) \leq D_2+M_1
        &\implies d_{G_2}(\phi_f^{-1} \circ \phi_f(x), \phi_f^{-1}(y)) \leq \lambda_3(D_2+M_1)+\epsilon_3\\
        &\implies  d_{G_2}(x, \phi_f^{-1}(y)) \leq \lambda_3(D_2+M_1)+2\epsilon_3.
    \end{align*}
    Since $y\in H'$ was arbitrary, by setting \(k''= \lambda_3(D_2+M_1)+2\epsilon_3\), we conclude $$\phi_f^{-1}(H') \subseteq N_{K''}(H_a).$$
    This completes the proof.
\end{proof}

\begin{rem}
    Since $\phi_f: G_1 \to G_2$ is a coarsely cusp-preserving quasi-isometry,  it induces a quasi-isometry $\phi_f^h: X^h \to Y^h$ and further it can be extended to a homeomorphism $\partial\phi_f^h: \partial X^h \to \partial Y^h$.
 For each parabolic endpoint $a\in \partial X^h$, since $\phi_f$ coarsely maps $H_a$ to $H_{fa}$, we have  $\partial\phi_f^h(a)=f(a)$. As parabolic endpoints are dense in $\partial X^h$, we have $\partial\phi_f^h=f$.

\end{rem}




\begin{thebibliography}{wide entry}   
    \bibitem[ACGH17]{arzhan} 
     Goulnara N. Arzhantseva, Christopher H. Cashen, Dominik Gruber, and David Hume. 
     \newblock \emph{Characterizations of {M}orse quasi-geodesics via superlinear divergence and sublinear contraction.} 
     \newblock Doc. Math. \textbf{22} (2017), 1193--1224.  MR3690269. 

     \bibitem[BH99]{BH99} 
     Martin R. Bridson and Andr\'e Haefliger. 
     \newblock \emph{Metric Spaces of Non-Positive Curvature}. 
     \newblock Grundlehren der mathematischen Wissenschaften [Fundamental Principles of Mathematical Sciences], Vol.~319. Springer-Verlag, Berlin, 1999. MR1744486. 

     \bibitem[Bow91]{Bow1991} 
     Brian H. Bowditch. 
     \newblock \emph{Notes on Gromov's hyperbolicity criterion for path-metric spaces.} 
     \newblock Group Theory from a Geometrical Viewpoint (Trieste, 1990), 64--167. World Sci. Publ., River Edge, NJ, 1991.  MR1170364. 

     \bibitem[Bow12]{Bowditch}
     Brian H. Bowditch.
     \newblock \emph{Relatively hyperbolic groups.}
     \newblock Int. J. Algebra Comput. \textbf{22} (2012), no.~3. MR2922380. 

     \bibitem[CCM19]{CCM2019}
     Ruth Charney, Matthew Cordes, and Devin Murray.
     \newblock \emph{Quasi-M\"obius homeomorphisms of Morse boundaries.}
     \newblock Bull. Lond. Math. Soc. \textbf{51} (2019), 501--515. MR3964503. 

     \bibitem[Dru09]{Drutu}
     Cornelia Dru\c{t}u.
     \newblock \emph{Relatively hyperbolic groups: geometry and quasi-isometric invariance.} 
     \newblock Comment. Math. Helv. \textbf{84} (2009), 503--546. MR2507252. 

     \bibitem[Far98]{Farb} 
     Benson Farb.
     \newblock \emph{Relatively hyperbolic groups.}
     \newblock Geom. Funct. Anal. \textbf{8} (1998), 810--840.  MR1650094.

     \bibitem[GM08]{GM2008}
     Daniel Groves and Jason Fox Manning.
     \newblock \emph{Dehn filling in relatively hyperbolic groups.}
     \newblock Israel J. Math. \textbf{168} (2008), 317--429.  MR2448064.

    \bibitem[Gro87]{Gro87}
    Mikhail Gromov.
    \newblock \emph{Hyperbolic groups.}
    \newblock In \emph{Essays in Group Theory}, MSRI Publ., Vol.~8, 75--263, Springer, New York, 1987. MR0919829.

    \bibitem[HH21]{Healy-Hruska}
    Brendan Burns Healy and G. Christopher Hruska.
    \newblock \emph{Cusped spaces and quasi-isometries of relatively hyperbolic groups.}
    \newblock arXiv:2010.09876, 2021. arXiv2010.09876.
    
    \bibitem[Hru10]{Hruska}
    G. Christopher Hruska.
    \newblock \emph{Relative hyperbolicity and relative quasiconvexity for countable groups.}
    \newblock Algebr. Geom. Topol.
    \textbf{10} (2010), 1807--1856. MR2684983.
    
    \bibitem[KM13]{Kolpakov-Martelli}
    Alexander Kolpakov and Bruno Martelli.
    \newblock \emph{Hyperbolic four-manifolds with one cusp.}
    \newblock Geom. Funct. Anal.
    \textbf{23} (2013), 1903--1933. MR3132905.
    
    \bibitem[MS24]{Mackay-Sisto}
    John M. Mackay and Alessandro Sisto.
    \newblock \emph{Maps between relatively hyperbolic metric spaces and between their boundaries.}
    \newblock Trans. Amer. Math. Soc.
    \textbf{377} (2024), 1409--1454. MR4688555.
    
    \bibitem[MT10]{Mackay-Tyson}
    John M. Mackay and Jeremy T. Tyson.
    \newblock \emph{Conformal Dimension}.
    \newblock Univ. Lecture Ser., Vol.~54,
    Amer. Math. Soc., Providence, RI, 2010. MR2662522.
    
    \bibitem[Pal10]{Pal}
    Abhijit Pal.
    \newblock \emph{Cannon--Thurston Maps and Relative Hyperbolicity}.
    \newblock Ph.D. Thesis, Indian Statistical Institute, Kolkata, 2010.
    
    \bibitem[Pau96]{Pau96}
    Fr\'ed\'eric Paulin.
    \newblock \emph{Un groupe hyperbolique est d\'etermin\'e par son bord.}
    \newblock J. Lond. Math. Soc. (2) \textbf{54} (1996), 50--70. MR1395067.
    
    \bibitem[Sch95]{Schwartz}
    Richard Evan Schwartz.
    \newblock \emph{The quasi-isometry classification of rank one lattices.}
    \newblock Inst. Hautes \'Etudes Sci. Publ. Math. \textbf{82} (1995), 133--168. MR1383215.
\end{thebibliography}
\end{document}